\documentclass[a4paper]{article}

\usepackage{amsmath}
\usepackage{amsfonts}
\usepackage{amsthm}
\usepackage{amssymb}
\usepackage{mathtools}

\usepackage{booktabs}

\usepackage{orcidlink}

\usepackage[title]{appendix}

\usepackage{hyperref}
\hypersetup{
    pdftitle={Two-stage model reduction approaches for the efficient and certified solution of parametrized optimal control problems},
    pdfsubject={Parametrized linear-quadratic optimal control problems with time-varying system dynamics and their efficient solution using reduced basis model order reduction for final time adjoint states and dynamical systems},
    pdfauthor={Hendrik Kleikamp, Lukas Renelt},
    pdfkeywords={Parametrized optimal control, Linear time-varying systems, Model order reduction, Reduced basis, Offline-online decomposition, A posteriori error estimation}
}

\usepackage{cleveref}
\crefformat{equation}{\textup{#2(#1)#3}}
\crefrangeformat{equation}{\textup{#3(#1)#4--#5(#2)#6}}
\crefmultiformat{equation}{\textup{#2(#1)#3}}{ and \textup{#2(#1)#3}}
{, \textup{#2(#1)#3}}{, and \textup{#2(#1)#3}}
\crefrangemultiformat{equation}{\textup{#3(#1)#4--#5(#2)#6}}%
{ and \textup{#3(#1)#4--#5(#2)#6}}{, \textup{#3(#1)#4--#5(#2)#6}}{, and \textup{#3(#1)#4--#5(#2)#6}}

\usepackage{authblk}

\usepackage{multirow}

\usepackage{tikz}
\usepackage{pgfplots}
\pgfplotsset{compat=1.18}

\usepackage[noEnd=true]{algpseudocodex}
\usepackage{algorithm}

\algnewcommand{\LineComment}[1]{\State \(\triangleright\) #1}

\renewcommand{\Return}[1]{\State{\textbf{return} #1}}

\usepackage{caption}

\DeclareCaptionFormat{algor}{%
	\hrulefill\par\offinterlineskip\vskip1pt%
	\textbf{#1#2}#3\offinterlineskip\hrulefill}
\DeclareCaptionStyle{algori}{singlelinecheck=off,format=algor,labelsep=space}
\newenvironment{algo}[2]
{
	\begin{center}
		\captionof{algorithm}{#1}\label{#2}
		\begin{algorithmic}[1]
}
{
		\end{algorithmic}
		\offinterlineskip\vskip2pt\hrulefill
	\end{center}
}

\usepackage{geometry}
\usepackage[maxbibnames=99, backend=bibtex]{biblatex}
\pgfplotsset{
	DefaultStyle/.style={
		legend cell align={left},
		legend style={font=\scriptsize, fill opacity=1, draw opacity=1, text opacity=1, draw=white!80!black},
		tick align=outside,
		tick pos=left,
		xmajorgrids,
		x grid style={white!69.0196078431373!black},
		xtick style={color=black},
		ymajorgrids,
		y grid style={white!69.0196078431373!black},
		ytick style={color=black}
	}
}

\tikzset{
	cookiesdomain/.pic={
		\draw[viridisBlue, thick] (0,0) -- (1,0);
		\draw[viridisYellow, thick] (1,0) -- (1,1);
		\draw[viridisBlue, thick] (1,1) -- (0,1);
		\draw[matchingRed, thick] (0,1) -- (0,0);
		\node[viridisBlue] at (0.5,-0.05) {$\Gamma_N$};
		\node[viridisBlue] at (0.5,1.05) {$\Gamma_N$};
		\node[viridisYellow] at (1.06,0.5) {$\Gamma_D$};
		\node[matchingRed] at (-0.06,0.5) {$\Gamma_{\mathrm{in}}$};
		\node[draw, circle, fill, black, inner sep=1pt] at (0,0) {};
		\node[draw, circle, fill, black, inner sep=1pt] at (1,0) {};
		\node[draw, circle, fill, black, inner sep=1pt] at (0,1) {};
		\node[draw, circle, fill, black, inner sep=1pt] at (1,1) {};
		\draw[dashed, lightgray] (0.2,0) -- (0.2,1);
		\draw[dashed, lightgray] (0.4,0) -- (0.4,1);
		\draw[dashed, lightgray] (0.6,0) -- (0.6,1);
		\draw[dashed, lightgray] (0.8,0) -- (0.8,1);
		\draw[dashed, lightgray] (0,0.2) -- (1,0.2);
		\draw[dashed, lightgray] (0,0.4) -- (1,0.4);
		\draw[dashed, lightgray] (0,0.6) -- (1,0.6);
		\draw[dashed, lightgray] (0,0.8) -- (1,0.8);
		\draw (0.3,0.3) circle (0.1) node {$\Omega_1$};
		\draw (0.7,0.3) circle (0.1) node {$\Omega_2$};
		\draw (0.7,0.7) circle (0.1) node {$\Omega_3$};
		\draw (0.3,0.7) circle (0.1) node {$\Omega_4$};
		\node at (0.5, 0.5) {$\Omega_0$};
		\draw[black, thick] (-0.02,0.2) -- (0,0.2);
		\draw[black, thick] (-0.02,0.4) -- (0,0.4);
		\draw[black, thick] (-0.02,0.6) -- (0,0.6);
		\draw[black, thick] (-0.02,0.8) -- (0,0.8);
		\draw[black, thick] (0.2,-0.02) -- (0.2,0);
		\draw[black, thick] (0.4,-0.02) -- (0.4,0);
		\draw[black, thick] (0.6,-0.02) -- (0.6,0);
		\draw[black, thick] (0.8,-0.02) -- (0.8,0);
		\node at (-0.075,0.2) {$0.2$};
		\node at (-0.075,0.4) {$0.4$};
		\node at (-0.075,0.6) {$0.6$};
		\node at (-0.075,0.8) {$0.8$};
		\node at (0.2,-0.065) {$0.2$};
		\node at (0.4,-0.065) {$0.4$};
		\node at (0.6,-0.065) {$0.6$};
		\node at (0.8,-0.065) {$0.8$};
		\node at (-0.05,-0.05) {$(0,0)$};
		\node at (1.05,-0.05) {$(1,0)$};
		\node at (1.05,1.05) {$(1,1)$};
		\node at (-0.05,1.05) {$(0,1)$};
		\node at (0.15,1.0875) {\LARGE$\Omega$};
	}
}

\definecolor{viridisYellow}{RGB}{253,231,37}
\definecolor{viridisGreen}{RGB}{94,201,98}
\definecolor{viridisTeal}{RGB}{33,145,140}
\definecolor{viridisBlue}{RGB}{59,82,139}
\definecolor{viridisViolet}{RGB}{68,1,84}
\definecolor{matchingRed}{HTML}{D01F3C}
\definecolor{matchingOrange}{HTML}{FFC077}

\colorlet{colorFOM}{matchingRed}
\colorlet{colorGROM}{viridisBlue}
\colorlet{colorSRGROM}{viridisGreen}
\colorlet{colorGSRROM}{viridisTeal}
\colorlet{colorGCROM}{viridisYellow}
\colorlet{colorDGROM}{viridisViolet}

\colorlet{colorUnused}{viridisViolet}

\renewcommand{\d}[1]{\,\mathrm{d}#1}  

\DeclareMathOperator*{\argmax}{arg\,max}  
\DeclareMathOperator*{\argmin}{arg\,min}  

\newcommand{\innerProd}[2]{\left\langle #1,#2\right\rangle}  

\newcommand{\X}{X}  
\newcommand{\U}{U}  
\newcommand{\Y}{Y}  
\newcommand{\dualX}{\X'}  
\newcommand{\dualU}{\U'}  
\newcommand{\dualY}{\Y'}  
\newcommand{\G}{G}  
\renewcommand{\H}{H}  
\newcommand{\Hpr}{\H_{\mathrm{pr}}}  
\newcommand{\Had}{\H_{\mathrm{ad}}}  

\newcommand{\Lin}[2]{\mathcal{L}\left(#1,#2\right)}  
\newcommand{\Riesz}[1]{\mathcal{R}_{#1}}  
\newcommand{\Bil}[2]{\operatorname{Bil}\left(#1,#2\right)}  

\newcommand{\R}{\mathbb{R}}  
\newcommand{\N}{\mathbb{N}}  

\newcommand{\norm}[1]{\left\lVert\,#1\,\right\rVert}  
\newcommand{\Span}[1]{\operatorname{span}(#1)}  
\newcommand{\rank}[1]{\operatorname{rank}(#1)}  

\newcommand{\Gramian}[1]{\Lambda(#1)}  
\newcommand{\RedGramian}[1]{\hat{\Lambda}(#1)}  

\newcommand{\params}{\mathcal{P}}  
\newcommand{\paramstrain}{\params_{\mathrm{train}}}  
\newcommand{\paramstrainsys}{\params_{\mathrm{train,sys}}}  
\newcommand{\paramstrainfta}{\params_{\mathrm{train,fta}}}  
\newcommand{\paramstest}{\params_{\mathrm{test}}}  
\newcommand{\paramstraininner}{\params_{\mathrm{inner}}}

\newcommand{\const}{\gamma}  

\newcommand{\StateTransSign}{\Psi}  
\newcommand{\StateTrans}[3]{{\StateTransSign_{#1}(#2,#3)}}  
\newcommand{\RedPrStateTrans}[3]{\hat{\StateTransSign}_{#1}^{\mathrm{pr}}(#2,#3)}  
\newcommand{\RedAdStateTrans}[3]{\hat{\StateTransSign}_{#1}^{\mathrm{ad}}(#2,#3)}  

\newcommand{\ResPrim}[3]{R_{#1}^{\mathrm{pr}}\left[#2,#3\right]}  
\newcommand{\ResAdjo}[2]{R_{#1}^{\mathrm{ad}}\left[#2\right]}  

\newcommand{\EstPrim}[2]{\Delta_{#1}^{\mathrm{pr}}\left[#2\right]}  
\newcommand{\EstAdjo}[2]{\Delta_{#1}^{\mathrm{ad}}\left[#2\right]}  
\newcommand{\EstGram}[2]{\Delta_{#1}^\Lambda\left[#2\right]}  

\newcommand{\ErrPrim}[2]{e^{\mathrm{pr}}\left[#1,#2\right]}  
\newcommand{\ErrAdjo}[2]{e^{\mathrm{ad}}\left[#1,#2\right]}  

\newcommand{\EstFTA}[2]{\eta_{#1}^{#2}}  
\newcommand{\RedEstFTA}[2]{\hat{\eta}_{#1}^{#2}}  
\newcommand{\FullRedEstFTA}[2]{\eta_{#1}^{#2,\mathrm{red}}}  

\newcommand{\optFTA}[1]{\varphi_{#1}^*(T)}  
\newcommand{\redFTA}[2]{\tilde{\varphi}_{#1}^{#2}}  
\newcommand{\compredFTA}[2]{\hat{\varphi}_{#1}^{#2,\mathrm{red}}}  

\newcommand{\card}[1]{\##1}  
\newcommand{\normX}[1]{\left\lVert\,#1\,\right\rVert_{\X}}  
\newcommand{\normDualX}[1]{\left\lVert\,#1\,\right\rVert_{\dualX}}  
\newcommand{\normU}[1]{\left\lVert\,#1\,\right\rVert_{\U}}  
\newcommand{\normY}[1]{\left\lVert\,#1\,\right\rVert_{\Y}}  
\newcommand{\im}[1]{\operatorname{Im}(#1)}  

\newcommand{\eye}[1]{\operatorname{id}_{#1}}  

\newcommand{\Vpr}{V_\mathrm{pr}}  
\newcommand{\Wpr}{W_{\mathrm{pr}}}  
\newcommand{\Vad}{V_{\mathrm{ad}}}  
\newcommand{\Wad}{W_{\mathrm{ad}}}  
\newcommand{\Vred}{V_{\dimRedSpaceFTA}}  

\newcommand{\Projection}[1]{P_{#1}}

\newcommand{\littletaller}{\mathchoice{\vphantom{\big|}}{}{}{}}
\newcommand{\restricted}[2]{{
		\left.\kern-\nulldelimiterspace 
		#1 
		\littletaller 
		\right|_{#2} 
}}

\newcommand{\ProjectionMatrix}[1]{D_{#1}}

\newcommand{\spaceVpr}{\mathcal{V}_\mathrm{pr}}  
\newcommand{\spaceWpr}{\mathcal{W}_{\mathrm{pr}}}  
\newcommand{\spaceVad}{\mathcal{V}_{\mathrm{ad}}}  
\newcommand{\spaceWad}{\mathcal{W}_{\mathrm{ad}}}  
\newcommand{\spaceVred}[1]{\mathcal{X}^{#1}}  

\newcommand{\redBasisVred}[1]{\Phi^{#1}}
\newcommand{\redBasisVpr}{\Phi_{\spaceVpr}}
\newcommand{\redBasisWpr}{\Phi_{\spaceWpr}}
\newcommand{\redBasisVad}{\Phi_{\spaceVad}}
\newcommand{\redBasisWad}{\Phi_{\spaceWad}}

\newcommand{\kpr}{k_{\mathrm{pr}}}  
\newcommand{\kad}{k_{\mathrm{ad}}}  
\newcommand{\dimRedSpaceFTA}{N}  

\newcommand{\E}{E}  
\newcommand{\Ead}{E_{\mathrm{ad}}}  
\newcommand{\A}[1]{A(#1)}  
\newcommand{\B}[1]{B(#1)}  
\newcommand{\C}{C}  
\newcommand{\Mop}{M}  
\newcommand{\Rop}{R}  
\newcommand{\Gmat}{\underline{X}}  

\newcommand{\EplusMtimesGramian}{\E^*+\Riesz{\X}^{-1}\Mop\Gramian{\mu}}
\newcommand{\EplusMtimesRedGramian}{\E^*+\Riesz{\X}^{-1}\Mop\RedGramian{\mu}}
\newcommand{\MxNaughtMinusxT}{\Riesz{\X}^{-1}\Mop\left(\StateTrans{\mu}{T}{0}x_\mu^0-x_\mu^T\right)}
\newcommand{\RedMxNaughtMinusxT}{\Riesz{\X}^{-1}\Mop\left(\RedPrStateTrans{\mu}{T}{0}\Projection{\spaceVpr}x_\mu^0-x_\mu^T\right)}

\newcommand{\Ham}{\mathcal{H}}  

\newcommand{\optFunc}[1]{\mathcal{J}_{#1}}  

\newcommand{\redEpr}{\hat{E}_{\mathrm{pr}}}  
\newcommand{\redEad}{\hat{E}_{\mathrm{ad}}}  
\newcommand{\redApr}[1]{\hat{A}_{\mathrm{pr}}(#1)}  
\newcommand{\redBpr}[1]{\hat{B}_{\mathrm{pr}}(#1)}  
\newcommand{\redAad}[1]{\hat{A}_{\mathrm{ad}}(#1)}  
\newcommand{\redBad}[1]{\hat{B}_{\mathrm{ad}}(#1)}  

\newcommand{\SolutionManifold}{\mathcal{M}}  

\newcommand{\nt}{n_t}  
\newcommand{\normal}{\vec{n}}  

\newcommand{\eps}{\varepsilon}  
\newcommand{\epssys}{\eps_{\mathrm{sys}}}  
\newcommand{\epsFTA}{\eps_{\mathrm{fta}}}  
\newcommand{\epsinner}{\eps_{\mathrm{in}}}  

\newcommand{\FOM}{FOM}  
\newcommand{\GROM}{G-ROM}  
\newcommand{\SRGROM}{SR-G-ROM}  
\newcommand{\GSRROM}{G-SR-ROM}  
\newcommand{\GCROM}{GC-ROM}  
\newcommand{\DGROM}{DouG-ROM}  

\newtheorem{theorem}{Theorem}
\newtheorem{lemma}{Lemma}

\theoremstyle{remark}
\newtheorem{remark}{Remark}

\newcommand*\samethanks[1][\value{footnote}]{\footnotemark[#1]}

\begin{document}
    \title{Two-stage model reduction approaches for the efficient and certified solution of parametrized optimal control problems}
    \author{Hendrik Kleikamp\,\orcidlink{0000-0003-1264-5941}\,\thanks{Funded by the Deutsche Forschungsgemeinschaft (DFG, German Research Foundation) under Germany's Excellence Strategy EXC 2044 –390685587, Mathematics Münster: Dynamics–Geometry–Structure.}\ \thanks{Corresponding author: {\tt hendrik.kleikamp@uni-muenster.de}}\ }
    \author{Lukas Renelt\,\orcidlink{0009-0003-3161-5219}\,\samethanks[1]}
    \affil{Institute for Analysis and Numerics, Mathematics Münster, University of Münster, Einsteinstrasse 62, 48149 Münster, Germany, {\tt hendrik.kleikamp@uni-muenster.de, lukas.renelt@uni-muenster.de}}

	\maketitle

    \begin{abstract}
        \noindent In this contribution we develop an efficient reduced order model for solving parametrized linear-quadratic optimal control problems with linear time-varying state system. The fully reduced model combines reduced basis approximations of the system dynamics and of the manifold of optimal final time adjoint states to achieve a computational complexity independent of the original state space. Such a combination is particularly beneficial in the case where a deviation in a low-dimensional output is penalized. In addition, an offline-online decomposed a posteriori error estimator bounding the error between the approximate final time adjoint with respect to the optimal one is derived and its reliability proven. We propose different strategies for building the involved reduced order models, for instance by separate reduction of the dynamical systems and the final time adjoint states or via greedy procedures yielding a combined and fully reduced model. These algorithms are evaluated and compared for a two-dimensional heat equation problem. The numerical results show the desired accuracy of the reduced models and highlight the speedup obtained by the newly combined reduced order model in comparison to an exact computation of the optimal control or other reduction approaches.
    \end{abstract}

    \section{Introduction}
    Optimal control problems play an important role in several areas of applied mathematics. The governing dynamical systems frequently involve parameters and the resulting optimal control problem is to be solved fast for many different values of the parameters -- for instance in a real-time or many-query context. Solving the exact optimal control problem is usually already costly for a single parameter. Hence, doing so for many values of the parameter is prohibitively expensive and infeasible in most applications. A recent overview of methods for parametrized optimal control problems is given in~\cite{lazar2022control}.
    \par
    In order to solve parameter-dependent optimal control problems efficiently, model order reduction has been applied successfully in recent years. Several papers propose methods for model order reduction of parametrized dynamical systems, in particular considering linear time-invariant systems, see for instance~\cite{haasdonk2011efficient}. These approaches mainly focus on replacing the involved high-dimensional state systems by low-order approximations that can be solved efficiently. Some papers further treat in particular optimal control problems. For instance in~\cite{ballarin2022spacetime,dede2012reduced,lazar2016greedy,kaercher2014posteriori}, different model order reduction approaches for parametrized optimal control problems have been developed. In~\cite{lazar2016greedy,kleikamp2024greedy} the authors describe algorithms to obtain reduced bases for the manifold of optimal adjoint states at the final time of the time horizon for different optimal control problems involving linear time-invariant state systems. In this contribution we describe how to combine the two approaches in the case of time-varying systems, that is the reduced order modeling of the system dynamics and the model order reduction for the final time adjoint states. This combination will lead to a severe speedup compared to the individual reduced models. In particular for optimal control problems in which a deviation in a low-dimensional output is penalized, the set of optimal final time adjoint states also possesses a similar low-dimensional structure and we thus benefit significantly from an additional reduction of the final time adjoint states. Using the approaches proposed in this paper, optimal control problems with a high-dimensional state space and time-dependent system dynamics can then be solved efficiently with a complexity independent of the dimension of the state space.
    \par
    The idea of applying an additional system reduction to the reduced order model for final time adjoint states has also been considered in~\cite{fabrini2018reduced} where the authors propose to first compress the state and adjoint trajectories, and afterwards run a greedy algorithm using the reduced dynamical systems to determine a reduced model for the final time adjoint states. In this contribution we develop a reliable and efficiently computable a posteriori error estimator for the combined reduced order model which was, to the best of our knowledge, not available thus far. The error estimator builds on ideas from~\cite{haasdonk2011efficient} to estimate the error introduced by approximating the dynamical systems and~\cite{kleikamp2024greedy} to estimate the error in the approximation of the final time adjoint state. Furthermore, we discuss several approaches for constructing the involved reduced order models and demonstrate their behavior and performance numerically. In contrast to the previous contributions we consider time-varying systems and extend the reduced order models as well as the error estimators to this setting.
    \par
    Reduced basis methods have been applied to parametrized problems very successfully in the past, see for instance~\cite{benner2015survey,hesthaven2016certified}. In the context of control of parametrized dynamical systems, reduced basis methods were used for example in~\cite{dihlmann2016reduced} to compute a reduced Kalman filter for parametrized partial differential equations~(PDEs) or in~\cite{schmidt2015basis,schmidt2018reduced} to construct an approximate linear-quadratic regulator by solving a projected Riccati equation. In~\cite{przybilla2024model} reduced basis methods were applied to efficiently solve parametric differential-algebraic systems by means of balanced truncation, in~\cite{przybilla2024semiactive} to speed up the solution of parametric Lyapunov equations within an optimization process of damping values, and in~\cite{kaercher2017certified} to efficiently obtain certified solutions of parametrized elliptic optimal control problems. Recently, also data-driven techniques using machine learning approaches have been developed for optimal control problems, see for instance~\cite{verma2024neural,kleikamp2024greedy,keil2022adaptive}. Methods based on interpolation of reduced models~\cite{panzer2010parametric} or Gramians~\cite{son2021balanced} have been proposed as well.
    \par
    The main novel contributions of this paper are threefold: On the one hand, we extend the previously developed reduced order models for parametrized optimal control problems with time-invariant state system from~\cite{lazar2016greedy} and~\cite{fabrini2018reduced} to time-varying problems in the infinite-dimensional setting. We further describe a combined reduced order model and derive a reliable a posteriori error estimator that can be evaluated very efficiently using an offline-online decomposition. Finally, we propose and compare, by means of a two-dimensional numerical example, different approaches for constructing the involved reduced order models.
    \newline
    \newline
    The paper is organized as follows: In~\Cref{sec:parametrized-linear-time-varying-systems-optimal-control} we introduce the considered optimal control problem in an abstract Hilbert space setting. Afterwards, in~\Cref{sec:reduced-model-system-dynamics} the projection-based model order reduction approach for reducing the system dynamics is presented together with associated a posteriori error estimates and efficient offline-online decompositions. In~\Cref{sec:reduced-model-optimal-adjoint-states} we recall the reduced order model for final time adjoint states from~\cite{kleikamp2024greedy} and show its extension to time-varying systems. The previously defined reduced order models are combined to a fully reduced model in~\Cref{sec:fully-reduced-model}. We further present a reliable a posteriori error estimator that can be evaluated efficiently. Several strategies to construct the reduced order models entering the combined model are proposed in~\Cref{sec:strategies-computing-fully-reduced-model}. By means of a numerical example we show different features and behavior of the devised reduced order models and reduction algorithms in~\Cref{sec:numerical-experiments}. Furthermore, the efficiency and accuracy of the computed approximations are verified. The results obtained in this paper are summarized in~\Cref{sec:conclusion-outlook} and future perspectives are discussed.

	\section{Parametrized linear time-varying systems and optimal control}\label{sec:parametrized-linear-time-varying-systems-optimal-control}
	In this contribution, we consider linear-quadratic optimal control problems where the system dynamics are given by linear time-varying systems with parameter-dependent system components. In the following, we first introduce the dynamical systems and afterwards describe the corresponding optimal control problems. Furthermore, the optimality system associated with the optimal control problem is stated and the role of the adjoint state at final time discussed.
	
	\subsection{Linear time-varying control systems}\label{sec:linear-time-varying-control-systems}
	We describe the main ideas and systems in an abstract and possibly infinite-dimensional setting. To this end, let~$\X$, $\U$ and~$\Y$ be real Hilbert spaces. The space~$\X$ corresponds to the state space, whereas~$\U$ denotes the control space and~$\Y$ the output space. We further denote by~$\innerProd{\cdot}{\cdot}_\X$, $\innerProd{\cdot}{\cdot}_\U$ and~$\innerProd{\cdot}{\cdot}_\Y$ the inner products on~$\X$, $\U$ and~$\Y$, respectively, with their induced norms~$\normX{\,\cdot\,}$, $\normU{\,\cdot\,}$ and~$\normY{\,\cdot\,}$. The dual spaces associated to~$\X$, $\U$ and~$\Y$, i.e.~the Hilbert spaces of bounded linear functionals on~$\X$, $\U$ and~$\Y$, are denoted as~$\dualX$, $\dualU$ and~$\dualY$. Due to the Riesz representation theorem, the Riesz maps~$\Riesz{\X}\colon\X\to\dualX$, $\Riesz{\U}\colon\U\to\dualU$ and~$\Riesz{\Y}\colon\Y\to\dualY$ are isometric isomorphisms. Moreover, for Hilbert spaces~$Y_1$ and~$Y_2$, let~$\Lin{Y_1}{Y_2}$ denote the Banach space of bounded linear operators mapping from~$Y_1$ to~$Y_2$. Using this notation we have~$\Riesz{\X}\in\Lin{\X}{\dualX}$, $\Riesz{\U}\in\Lin{\U}{\dualU}$ and~$\Riesz{\Y}\in\Lin{\Y}{\dualY}$, as well as~$\dualX=\Lin{\X}{\R}$, $\dualU=\Lin{\U}{\R}$ and~$\dualY=\Lin{\Y}{\R}$. We further denote by~$\Bil{Y_1}{Y_2}$ the space of continuous bilinear forms mapping from~$Y_1\times Y_2$ to~$\R$. The dual paring between an element~$y\in Y_1$ and a continuous linear functional~$y'\in Y_1'$ is denoted as~$\innerProd{y}{y'}_{Y_1\times Y_1'}=\innerProd{y'}{y}_{Y_1'\times Y_1}\coloneqq y'(y)$.
	\par
	In order to define the parametrized control systems of interest, we consider a compact subset~$\params$ of some Banach space as the parameter set. We further consider the time interval~$[0,T]$ for some final time~$T>0$. For a given parameter~$\mu\in\params$, let~$\A{\mu}\in L^\infty([0,T];\Lin{\X}{\dualX})$ be the state operator, $\B{\mu}\in L^\infty([0,T];\Lin{\U}{\dualX})$ the control operator and~$\C(\mu)\in\Lin{\X}{\Y}$ the output operator. We are interested in optimal control problems with system dynamics for the state~$x_{\mu}\in\H\coloneqq L^2([0,T];\X)\cap H^1([0,T];\dualX)$ driven by the primal state equation which is for a parameter~$\mu\in\params$, control~$u\in\G\coloneqq L^2([0,T];\U)$, and initial state~$\bar{x}\in\X$ given as
	\begin{align}\label{equ:primal-equation}
		\E\frac{d}{dt}x_{\mu}(t) &= \A{\mu;t} x_{\mu}(t) + \B{\mu;t} u(t),\qquad x_{\mu}(0)=\bar{x},\\
		y_{\mu}(t) &= \C(\mu) x_{\mu}(t)\notag
	\end{align}
	for almost all~$t\in[0,T]$, where~$\E\in\Lin{\dualX}{\dualX}$ is an invertible operator (here,~$\E$ can also be the identity map~$\eye{\dualX}$). After discretization, the matrix representing~$\E$ arises for instance as the mass matrix in the discretization of parabolic PDEs using the finite element method. Since mass matrices are usually independent of the parameter, we refrain from considering a parameter-dependent operator~$\E$. However, all developed results extend readily to the more general case where~$\E$ depends on parameter.
	
	\subsection{Parametrized linear-quadratic optimal control problems}\label{sec:parametrized-linear-quadratic-optimal-control-problems}
	We consider parametrized optimal control problems with linear dynamics as introduced in~\Cref{sec:linear-time-varying-control-systems} and a quadratic objective functional. For a given parameter~$\mu\in\params$, the objective functional~$\optFunc{\mu}\colon\G\to[0,\infty)$ is for any control~$u\in\G$ defined as
	\begin{align*}
		\optFunc{\mu}(u) \coloneqq \frac{1}{2}\left[\normY{\C(\mu)\left(x_\mu(T)-x_\mu^T\right)}^2 + \int\limits_0^T \innerProd{u(t)}{\Rop(t)u(t)}_{\U\times\dualU}\d{t}\right]
	\end{align*}
	where~$x_\mu\in\H$ solves~\cref{equ:primal-equation}. Moreover,~$\Rop\colon[0,T]\to\Lin{\U}{\dualU}$ is an operator such that~$\Riesz{\U}^{-1}\Rop(t)$ is self-adjoint and strictly positive-definite for almost all times~$t\in[0,T]$. The assumptions on~$\Rop(t)$ imply in particular that~$\Rop(t)$ is invertible. By means of the operator~$\Rop$ it is possible to assign different weights to the individual components of the control. The weighting operator~$\Rop(t)$ can also be chosen in a parameter-dependent way. The objective functional measures the deviation in the output of the target state at final time~$T$ as well as the control energy (weighted by the operator~$\Rop(t)$ where the weighting may change over time). Since we only penalize a deviation of the output at the final time~$T$, we assume that the output operator~$\C(\mu)$ is constant with respect to time. For simplicity of notation we also omit the dependence of~$\C(\mu)$ on the parameter~$\mu\in\params$ in the following.
	\par
	The parametrized optimal control problem thus reads
	\begin{align}\label{equ:optimal-control-problem}
		\min_{u\in\G}\optFunc{\mu}(u),\quad\text{such that}\quad \E\frac{d}{dt}x_\mu(t)=\A{\mu}x_\mu(t)+\B{\mu}u(t)\quad\text{for almost all }t\in[0,T],\quad x_\mu(0)=x_\mu^0.
	\end{align}
	The objective functional~$\optFunc{\mu}$ is strongly convex with respect to the control~$u\in\G$ due to the positivity assumption on~$\Rop$. Furthermore, the state equation~\cref{equ:primal-equation} possesses a unique solution due to Carathéodory's existence theorem, see~\cite[Chapter~I.5]{hale1980ordinary}. The optimal control problem~\cref{equ:optimal-control-problem} therefore admits a unique solution, see for instance~\cite{peypouquet2015convex}.
	\par
	In the next section, we state the optimality system associated with the optimal control problem~\cref{equ:optimal-control-problem}.
	
	\subsection{Optimality system and optimal final time adjoint states}\label{sec:optimality-system}
	The following result has for instance been shown in~\cite{kleikamp2024greedy} for linear time-invariant systems without the operator~$\E$ on the left-hand side. Since significant changes to the statement and proof from~\cite{kleikamp2024greedy} are required, we again give a proof of the optimality system in the appendix.
	\begin{theorem}[Optimality system for the linear-quadratic optimal control problem]\label{thm:optimality-system}
		Given a parameter~$\mu\in\params$, the optimal control of the problem in~\cref{equ:optimal-control-problem} can equivalently be expressed as the solution to the optimality system
		\begin{subequations}\label{equ:optimality-system-main}
			\begin{align}
				\E\frac{d}{dt}x_\mu(t) &= \A{\mu;t} x_\mu(t)+\B{\mu;t} u_\mu(t), \label{equ:primal-system} \\
				u_\mu(t) &= -\Rop(t)^{-1}\B{\mu;t}^* \varphi_\mu(t),\label{equ:control-system} \\
				-\Ead\frac{d}{dt}\varphi_\mu(t) &= \A{\mu;t}^* \varphi_\mu(t), \label{equ:adjoint-system}
			\end{align}
			for almost all~$t\in[0,T]$ with initial respectively terminal conditions
			\begin{align}\label{equ:optimality-system-boundary-conditions}
				x_\mu(0) = x_\mu^0,\qquad \Ead\Riesz{\X}\varphi_\mu(T)=\Mop\left(x_\mu(T)-x_\mu^T\right),
			\end{align}
		\end{subequations}
		where~$\Mop\coloneqq\C^*\Riesz{\Y}\C\in\Lin{\X}{\dualX}$. The adjoint mass operator~$\Ead\in\Lin{\dualX}{\dualX}$ is given as~$\Ead=\Riesz{\X}\E^*\Riesz{\X}^{-1}$, i.e.~we have~$\Ead\Riesz{\X}=\Riesz{\X}\E^*$. We denote solutions of the optimality system~\cref{equ:optimality-system-main} as~$x_\mu^*\in\H$, $\varphi_\mu^*\in\H$, and~$u_\mu^*\in\G$, respectively.
	\end{theorem}
	\begin{proof}
		See~\Cref{app:proof-optimality-system}.
	\end{proof}
	From the optimality system we can in particular deduce that the system dynamics of the adjoint system~\cref{equ:adjoint-system} and therefore also the dynamics of the primal system~\cref{equ:primal-system} are already fully described by the optimal final time adjoint state~$\optFTA{\mu}\in\X$. It is in particular possible to compute the optimal final time adjoint~$\optFTA{\mu}$ by solving a linear system. In order to obtain a unique solution, we assume that the operator~$\EplusMtimesGramian\in\Lin{\X}{\X}$ is invertible where the (weighted) controllability Gramian~$\Gramian{\mu}\in\Lin{\X}{\X}$ is defined by its application to~$p\in\X$ as
	\begin{align}\label{equ:gramian-product}
		\Gramian{\mu} p \coloneqq -x_\mu(T).
	\end{align}
	Here,~$(x_\mu,u_\mu,\varphi_\mu)\in\H\times\G\times\H$ solve the optimality system~\cref{equ:primal-system,equ:control-system,equ:adjoint-system} with initial condition~$x_\mu(0)=0\in\X$ and terminal condition~$\varphi_\mu(T)=p$. Applying the controllability Gramian to a given final time adjoint state~$p$ thus solves (up to a minus sign) the optimality system with zero initial condition for the primal equation and~$p$ as the final time adjoint and returns the state at the final time~$T$. Since all involved operators are linear and no free dynamics need to be taken into account (due to the homogeneous initial condition for the primal equation), the Gramian operator~$\Gramian{\mu}$ is indeed a linear operator. The boundedness results from the boundedness of all involved operators and the choice of a setting with finite time horizon. The linear system identifying the optimal final time adjoint~$\optFTA{\mu}\in\X$ is then characterized in the following Lemma, see also~\cite[Section~2.2]{kleikamp2024greedy} for the linear time-invariant case without the operator~$\E$:
	\begin{lemma}[Linear system of equations for the optimal final time adjoint state]\label{lem:linear-system-optimal-final-time-adjoint}
		The optimal final time adjoint state~$\optFTA{\mu}\in\X$ that solves the optimality system~\cref{equ:primal-system,equ:control-system,equ:adjoint-system} is given as the unique solution to the linear system
		\begin{align}\label{equ:linear-system-optimal-final-time-adjoint}
			\big(\EplusMtimesGramian\big)\optFTA{\mu} = \MxNaughtMinusxT
		\end{align}
		where~$\StateTrans{\mu}{t_1}{t_0}\in\Lin{\X}{\X}$ denotes the state transition operator of the homogeneous system~$\E\frac{d}{dt}x_\mu(t)=\A{\mu;t}x_\mu(t)$ with initial condition at time~$t_0$, i.e.~we have~$x_\mu(t_1)=\StateTrans{\mu}{t_1}{t_0}x_\mu(t_0)$.
	\end{lemma}
	\begin{proof}
		We first observe that it holds
		\begin{align*}
			x_\mu^*(T) = \StateTrans{\mu}{T}{0}x_\mu^0 - \Gramian{\mu}\optFTA{\mu}
		\end{align*}
		due to the linearity of the state and adjoint equation. Inserting this result into the terminal condition~\cref{equ:optimality-system-boundary-conditions} we obtain
		\begin{align*}
			\Ead\Riesz{\X}\optFTA{\mu} = \Mop\left(\StateTrans{\mu}{T}{0}x_\mu^0 - \Gramian{\mu}\optFTA{\mu}-x_\mu^T\right).
		\end{align*}
		After rearranging this equation, we use that~$\Ead\Riesz{\X}=\Riesz{\X}\E^*$ and apply~$\Riesz{\X}^{-1}$ from the left. This yields the claimed identity.
	\end{proof}
	We remark that the operator~$\Gramian{\mu}$ does not need to be assembled, but can instead be evaluated by solving two dynamical systems. It is therefore possible to apply iterative solvers to the linear system in~\Cref{lem:linear-system-optimal-final-time-adjoint} without computing the operator~$\Gramian{\mu}$ explicitly.
	\par
	In the subsequent sections we derive reduced order models first for speeding up the solution of the primal and adjoint systems~\cref{equ:primal-system} and~\cref{equ:adjoint-system} in~\Cref{sec:reduced-model-system-dynamics}, and second for approximating the manifold~$\SolutionManifold$ of optimal final time adjoint states given as
	\begin{align}\label{equ:manifold}
		\SolutionManifold \coloneqq \{\optFTA{\mu}:\mu\in\params\}\subset\X
	\end{align}
	in~\Cref{sec:reduced-model-optimal-adjoint-states}.
	
	\section{Reduced order modeling for system dynamics}\label{sec:reduced-model-system-dynamics}
	The authors in~\cite{haasdonk2011efficient} describe a projection-based reduced order model for linear control systems together with a reliable a posteriori error estimator in a finite-dimensional setting. In this section, we extend the derivation of the reduced order model and error estimator to our generalized case. We will later apply this reduced order model to the setting of optimal control problems as introduced in the previous section. To this end, we will make use of a reduced version of the Gramian operator~\cref{equ:gramian-product} and derive an error estimate for its application. This also requires a corresponding a posteriori error estimate for the adjoint equation which we describe in this section as well.
	
	\subsection{Projection-based reduced basis model order reduction}\label{sec:reduced-basis-mor-for-system-dynamics}
	In order to define reduced order models for the system dynamics, we consider Petrov-Galerkin-projections of the primal and adjoint equations onto (low-dimensional) subspaces of the state space~$\X$. Here, we allow for different subspaces for the primal and the adjoint equation, respectively. To this end, let~$\spaceVpr\subset\X$ and~$\spaceWpr\subset\dualX$ be finite-dimensional subspaces of~$\X$ with dimension~$\dim{\spaceVpr}=\dim{\spaceWpr}=\kpr\in\N$. Furthermore, for the adjoint system we consider finite-dimensional subspaces~$\spaceVad\subset\X$ and~$\spaceWad\subset\dualX$ with~$\dim{\spaceVad}=\dim{\spaceWad}=\kad\in\N$. The goal is to find subspaces~$\spaceVpr$, $\spaceWpr$, $\spaceVad$, and~$\spaceWad$ such that the dimensions~$\kpr\in\N$ and~$\kad\in\N$ of the reduced spaces are small compared to the original state space while still capturing the respective system dynamics accurately in the reduced ansatz space. If this is the case, the reduced state and adjoint equations defined below can be solved efficiently. The reduced dynamical systems associated to the optimality system in~\Cref{thm:optimality-system} are given as
	\begin{subequations}\label{equ:reduced-optimality-system-main}
		\begin{align}
			\redEpr\frac{d}{dt}\hat{x}_{\mu}(t) &= \redApr{\mu;t} \hat{x}_{\mu}(t) + \redBpr{\mu;t} u(t),\label{equ:reduced-primal-system} \\
			\hat{u}_{\mu}(t) &= -\Rop(t)^{-1}\redBad{\mu;t}^* \hat{\varphi}_{\mu}(t),\label{equ:reduced-control-system} \\
			-\redEad\frac{d}{dt}\hat{\varphi}_{\mu}(t) &= \redAad{\mu;t}^* \hat{\varphi}_{\mu}(t),\label{equ:reduced-adjoint-system}
		\end{align}
		with projected initial and terminal conditions
		\begin{align}\label{equ:reduced-optimality-system-boundary-conditions}
			\hat{x}_{\mu}(0)=\hat{\bar{x}}=\Projection{\spaceVpr}\bar{x},\qquad \hat{\varphi}_{\mu}(T)=\hat{\bar{\varphi}}=\Projection{\spaceVad}\bar{\varphi},
		\end{align}
	\end{subequations}
	where the projected system operators are defined as
	\begin{align*}
		\redEpr &= \Projection{\spaceWpr}\restricted{\E}{\Riesz{\X}\spaceVpr}\in\Lin{\Riesz{\X}\spaceVpr}{\spaceWpr},& \redEad &= \Projection{\spaceWad}\restricted{\Ead}{\Riesz{\X}\spaceVad}\in\Lin{\Riesz{\X}\spaceVad}{\spaceWad},\\
		\redApr{\mu;t} &= \Projection{\spaceWpr}\restricted{\A{\mu;t}}{\spaceVpr}\in\Lin{\spaceVpr}{\spaceWpr},& \redBpr{\mu;t} &= \Projection{\spaceWpr}\B{\mu;t}\in\Lin{\U}{\spaceWpr},\\
		\redAad{\mu;t} &= \Projection{\spaceWad}\restricted{\A{\mu;t}}{\spaceVad}\in\Lin{\spaceVad}{\spaceWad},& \redBad{\mu;t}^* &= \restricted{\B{\mu;t}^*}{\spaceVad}\in\Lin{\spaceVad}{\dualU}.
	\end{align*}
	In order for the reduced systems to be well-defined, we assume that~$\redEpr$ and~$\redEad$ are invertible and thus isomorphisms.
	Furthermore, the initial condition~$\hat{\bar{x}}\in\spaceVpr$ and the terminal condition~$\hat{\bar{\varphi}}\in\spaceVad$ might be given by the projection of some high-dimensional initial state~$\bar{x}\in\X$ and terminal state~$\bar{\varphi}\in\X$ onto the spaces~$\spaceVpr$ and~$\spaceVad$, respectively, as shown in~\cref{equ:reduced-optimality-system-boundary-conditions}. We denote by~$\Hpr\coloneqq L^2([0,T];\spaceVpr)\cap H^1([0,T];\Riesz{\X}\spaceVpr)$ and~$\Had\coloneqq L^2([0,T];\spaceVad)\cap H^1([0,T];\Riesz{\X}\spaceVad)$ the spaces of reduced time-dependent primal and adjoint states. In the reduced case, different spaces for primal and adjoint states are possible to project the respective systems. Since we are interested in solving optimal control problems in an efficient manner, we furthermore consider a reduced version~$\RedGramian{\mu}\in\Lin{\spaceVad}{\spaceVpr}$ of the Gramian operator~$\Gramian{\mu}$ that uses the reduced systems~\cref{equ:reduced-primal-system,equ:reduced-control-system,equ:reduced-adjoint-system} and is defined for~$\hat{p}\in\spaceVad$ as
	\begin{align}\label{equ:reduced-gramian-product}
		\RedGramian{\mu}\hat{p} \coloneqq -\hat{x}_\mu(T),
	\end{align}
	where, similarly to the definition of the full Gramian operator, $(\hat{x}_\mu,\hat{u}_\mu,\hat{\varphi}_\mu)\in\Hpr\times\G\times\Had$ solve the reduced optimality system in~\cref{equ:reduced-primal-system,equ:reduced-control-system,equ:reduced-adjoint-system} with initial condition~$\hat{x}_\mu(0)=0\in\spaceVpr$ and terminal condition~$\hat{\varphi}_\mu(T)=\hat{p}$. Similar to the original Gramian~$\Gramian{\mu}$, we can compute~$\RedGramian{\mu}\hat{p}$ very efficiently by solving reduced systems. We further denote by~$\RedPrStateTrans{\mu}{t_1}{t_0}\in\Lin{\spaceVpr}{\spaceVpr}$ and~$\RedAdStateTrans{\mu}{t_1}{t_0}\in\Lin{\spaceVad}{\spaceVad}$ the state transition operators associated with~$(\redEpr,\redApr{\mu;t})$ and~$(\redEad,\redAad{\mu;t})$, respectively.
	
	\begin{remark}[Counterparts to the reduced operators in the matrix case]\label{rem:matrix-case}
		In the discrete setting, the reduced spaces~$\spaceVpr$, $\spaceWpr$, $\spaceVad$ and~$\spaceWad$ are typically represented via matrices~$\Vpr\in\R^{n\times\kpr}$, $\Wpr\in\R^{n\times\kpr}$, $\Vad\in\R^{n\times\kad}$ and~$\Wad\in\R^{n\times\kad}$. These matrices map reduced coefficients to the respective high-dimensional element from the state space~$\X=\R^n$. The reduced spaces are thus given as the images of the basis matrices, i.e.~$\spaceVpr=\im{\Vpr}$, $\spaceWpr=\im{\Wpr}$, $\spaceVad=\im{\Vad}$ and~$\spaceWad=\im{\Wad}$. The projected system matrices in the discrete case are given as
		\begin{align*}
			\redEpr &= \Wpr^\top\E\Vpr\in\R^{\kpr\times\kpr},& \redEad &= \Wad^\top\E\Vad\in\R^{\kad\times\kad},&&\\
			\redApr{\mu;t} &= \Wpr^\top\A{\mu;t}\Vpr\in\R^{\kpr\times\kpr},& \redBpr{\mu;t} &= \Wpr\B{\mu;t}\in\R^{\kpr\times m},\\
			\redAad{\mu;t} &= \Wad^\top\A{\mu;t}\Vad\in\R^{\kad\times\kad},& \redBad{\mu;t}^\top &= \B{\mu;t}^\top\Vad\in\R^{m\times\kad},
		\end{align*}
		where~$m\coloneqq\dim\U$ denotes the number of controls. The assumption that~$\redEpr$ is invertible corresponds to the condition that~$\Wpr^\top\E\Vpr$ is invertible (and similarly for~$\redEad$). This is for instance the case if~$\Vpr$ and~$\Wpr$ are biorthogonal with respect to~$\E$, i.e.~if we construct the projection matrices such that~$\Wpr^\top\E\Vpr=\eye{\kpr}$. This assumption is typically also advantageous in the context of discretized~PDEs where~$\E$ is the mass matrix of the discretization, see also the numerical experiment in~\Cref{sec:numerical-experiments}.
	\end{remark}
	
	\subsubsection*{Offline-online decomposition of the reduced order model}
	To be able to obtain efficient reduced order models, we assume that the operators~$\A{\mu;t}$ and~$\B{\mu;t}$ admit parameter-separable forms
	\begin{align}\label{equ:parameter-separability}
		\A{\mu;t} = \sum\limits_{q=1}^{Q_A} \theta_A^q(\mu;t)A_q\qquad\text{and}\qquad \B{\mu;t} = \sum\limits_{q=1}^{Q_B} \theta_B^q(\mu;t)B_q,
	\end{align}
	with continuous functions~$\theta_A^q\colon\params\times[0,T]\to\R$ and (parameter-independent) operators~$A_q\in\Lin{\X}{\dualX}$ for~$q=1,\dots,Q_A$, as well as continuous functions~$\theta_B^q\colon\params\times[0,T]\to\R$ and operators~$B_q\in\Lin{\U}{\dualX}$ for~$q=1,\dots,Q_B$. Further, the initial state~$x_\mu^0\in\X$ as well as the target state~$x_\mu^T\in\X$ are assumed to be parameter-separable similar to the system operators~$\A{\mu;t}$ and~$\B{\mu;t}$, i.e.~we have that
	\begin{align*}
		x_\mu^0 = \sum\limits_{q=1}^{Q_{x^0}} \theta_{x^0}^q(\mu)x_q^0\qquad\text{and}\qquad x_\mu^T = \sum\limits_{q=1}^{Q_{x^T}} \theta_{x^T}^q(\mu)x_q^T
	\end{align*}
	with continuous functions~$\theta_{x^0}^q\colon\params\times[0,T]\to\R$ and states~$x_q^0\in\X$ for~$q=1,\dots,Q_{x^0}$, as well as continuous functions~$\theta_{x^T}^q\colon\params\times[0,T]\to\R$ and states~$x_q^T\in\X$ for~$q=1,\dots,Q_{x^T}$.
	\par
	To solve the reduced systems in a fast manner, we also have to be able to assemble the reduced initial condition~$\Projection{\spaceVpr}x_\mu^0\in\R^{\kpr}$ as well as the reduced system operators~$\redApr{\mu;t}\in\Lin{\spaceVpr}{\spaceWpr}$, $\redBpr{\mu;t}\in\Lin{\U}{\spaceWpr}$, $\redAad{\mu;t}\in\Lin{\spaceVad}{\spaceWad}$, and~$\redBad{\mu;t}^*\in\Lin{\spaceVad}{\dualU}$ efficiently. To this end, one first precomputes~$v_0^q\coloneqq \Projection{\spaceVpr}x_q^0\in\spaceVpr$ for all~$q=1,\dots,Q_{x^0}$. During the online phase, we can thus derive~$\Projection{\spaceVpr}x_\mu^0$ as
	\begin{align*}
		\Projection{\spaceVpr}x_\mu^0 = \sum\limits_{q=1}^{Q_{x^0}}\theta_{x^0}^q(\mu)v_0^q
	\end{align*}
	for a new parameter~$\mu\in\params$. Then we have
	\begin{align*}
		\RedPrStateTrans{\mu}{T}{0}\Projection{\spaceVpr}x_\mu^0 = \hat{x}_\mu(T)\qquad\text{for }u=0\in\G\text{ and }\hat{x}_\mu(0)=\Projection{\spaceVpr}x_\mu^0.
	\end{align*}
	We also assemble~$\redApr{\mu;t}$, $\redBpr{\mu;t}$, $\redAad{\mu;t}$, and~$\redBad{\mu;t}$ efficiently using the parameter-separable form of~$\A{\mu;t}$ and~$\B{\mu;t}$ by first computing~$\hat{A}_{\text{pr}}^q\coloneqq \Projection{\spaceWpr}\restricted{A_q}{\spaceVpr}\in\Lin{\spaceVpr}{\spaceWpr}$ as well as~$\hat{A}_{\text{ad}}^q\coloneqq \Projection{\spaceWad}\restricted{A_q}{\spaceVad}\in\Lin{\spaceVad}{\spaceWad}$ for~$q=1,\dots,Q_A$ and similarly~$\hat{B}_{\text{pr}}^q\coloneqq\Projection{\spaceWpr}B^q\in\Lin{\U}{\spaceWpr}$ as well as~$\big(\hat{B}_{\text{ad}}^*\big)^q\coloneqq \restricted{\big(B^q\big)^*}{\spaceVad}\in\Lin{\spaceVad}{\dualU}$ for~$q=1,\dots,Q_B$ offline and afterwards assembling
	\begin{align*}
		\redApr{\mu;t} = \sum\limits_{q=1}^{Q_A}\theta_A^q(\mu;t)\hat{A}_{\text{pr}}^q\quad\text{and}\quad\redBpr{\mu;t} = \sum\limits_{q=1}^{Q_B}\theta_B^q(\mu;t)\hat{B}_{\text{pr}}^q
	\end{align*}
	as well as
	\begin{align*}
		\redAad{\mu;t} = \sum\limits_{q=1}^{Q_A}\theta_A^q(\mu;t)\hat{A}_{\text{ad}}^q\quad\text{and}\quad\redBad{\mu;t}^* = \sum\limits_{q=1}^{Q_B}\theta_B^q(\mu;t)\big(\hat{B}_{\text{ad}}^*\big)^q
	\end{align*}
	online in a time independent of the dimension of the state space~$\X$.
	
	\subsection{A posteriori error estimation}\label{sec:a-posteriori-for-primal-and-adjoint-systems}
	In this section, we recall the a posteriori error estimator introduced in~\cite{haasdonk2011efficient} and state the respective error estimate for the adjoint system for later usage. Recently, some other strategies for a posteriori error estimation of reduced basis methods for dynamical systems have been proposed, see for instance~\cite{rettberg2023improved,antoulas2018model}. In~\cite{rettberg2023improved} the authors formulate a hierarchical error estimator and an error bound based on an auxiliary linear problem for port-Hamiltonian systems. The general ideas of these approaches can also be applied in our setting. The approach in~\cite{antoulas2018model} builds on an improved error estimation by means of a dual system introduced in~\cite{haasdonk2011efficient} and successively reduces certain error systems.
	\par
	To describe the error estimator from~\cite{haasdonk2011efficient}, we start by defining the errors of the reconstruction of the primal and of the adjoint state as
	\begin{align*}
		\ErrPrim{x}{\hat{x}} &\coloneqq x-\hat{x}\in\H,\\
		\ErrAdjo{\varphi}{\hat{\varphi}} &\coloneqq \varphi-\hat{\varphi}\in\H,
	\end{align*}
	for~$x,\varphi\in\H$, $\hat{x}\in\Hpr$, and~$\hat{\varphi}\in\Had$. Furthermore, we consider the residuals of the primal and the dual equation which are given as
	\begin{align*}
		\ResPrim{\mu}{\hat{x}}{u}(t) &\coloneqq \A{\mu;t} \hat{x}(t) + \B{\mu;t} u(t) - \E\frac{d}{dt}\hat{x}(t)\in\dualX, \\
		\ResAdjo{\mu}{\hat{\varphi}}(t) &\coloneqq \A{\mu;t}^* \hat{\varphi}(t) + \Ead\frac{d}{dt}\hat{\varphi}(t)\in\dualX,
	\end{align*}
	for an approximate state trajectory~$\hat{x}\in\Hpr$, a control~$u\in\G$ and an approximate adjoint trajectory~$\hat{\varphi}\in\Had$. For a reduced primal solution~$\hat{x}_\mu$ that solves~\cref{equ:reduced-primal-system} for a control~$u\in\G$ and a reduced adjoint solution~$\hat{\varphi}_\mu$ that solves~\cref{equ:reduced-adjoint-system}, we have the Petrov-Galerkin-orthogonality conditions
	\begin{align*}
		\Projection{\spaceWpr}\ResPrim{\mu}{\hat{x}_\mu}{u}=0\in\spaceWpr\qquad\text{and}\qquad\Projection{\spaceWad}\ResAdjo{\mu}{\hat{\varphi}_\mu}=0\in\spaceWad.
	\end{align*}
	
	\subsubsection*{Error estimators for the primal and adjoint systems}
	Having this notation at hand, we can introduce the error estimators for the primal and the adjoint problem. Eventually, we are interested in the error with respect to the norm~$\normX{\,\cdot\,}$. A suitable choice of the inner product can lead to improved  error estimates in practice (see also~\cite[Section~4]{haasdonk2011efficient}). The error estimator for the primal system is stated in the following Lemma:
	\begin{lemma}[{A posteriori error estimator for the reduced order model of the primal system; see~\cite[Proposition~4.1]{haasdonk2011efficient}}]\label{lem:error-estimator-primal-system}
		Let~$\mu\in\params$ be a parameter and~$u\in\G$ be a control. Let further $\hat{x}_\mu\in\Hpr$ be the reduced primal state that solves the reduced primal equation in~\cref{equ:reduced-primal-system} for the initial condition~$\hat{x}_\mu(0)=\Projection{\spaceVpr}x_\mu^0$ and the control~$u$. Then we have at time~$t\in[0,T]$ that
		\begin{align*}
			\normX{x_\mu(t)-\hat{x}_\mu(t)}\ \leq\ \EstPrim{\mu}{u}(t).
		\end{align*}
		The error estimator~$\EstPrim{\mu}{u}$ is given as
		\begin{align*}
			\EstPrim{\mu}{u}(t) \coloneqq C_1(\mu)\normX{x_\mu^0-\Projection{\spaceVpr}x_\mu^0} + C_1(\mu)\int\limits_0^t \normX{\E^{-1}\ResPrim{\mu}{\hat{x}_\mu}{u}(s)}\d{s},
		\end{align*}
		where the constant~$C_1(\mu)$ is such that
		\begin{align}\label{equ:constant-c1}
			C_1(\mu) \geq \sup\limits_{t\in[0,T]}\sup\limits_{s\in[0,t]} \normX{\StateTrans{\mu}{t}{s}} = \sup\limits_{t\in[0,T]}\sup\limits_{s\in[0,t]}\sup\limits_{0\neq x\in\X}\frac{\normX{\StateTrans{\mu}{t}{s}x}}{\normX{x}}.
		\end{align}
	\end{lemma}
	\begin{proof}
		See the proof of Proposition~4.1 in~\cite{haasdonk2011efficient}. The proof in~\cite{haasdonk2011efficient} is conducted for the case of a linear time-invariant system but is readily generalized to the setting of linear time-varying systems as considered here, see also the discussion following the proof in~\cite{haasdonk2011efficient}. The proof of this Lemma is also similar to the one for the error estimator of the adjoint system which is presented in detail in~\Cref{lem:error-estimator-adjoint-system}.
	\end{proof}
	Similarly to the error estimator for the primal system, one can also derive an error estimator for the adjoint system:
	
	\begin{lemma}[{A posteriori error estimator for the reduced order model of the adjoint system}]\label{lem:error-estimator-adjoint-system}
		Let~$\mu\in\params$ be a parameter and~$\bar{\varphi}\in\X$ the terminal condition of the adjoint equation, i.e.~$\varphi_\mu(T)=\bar{\varphi}$. Let further $\hat{\varphi}_\mu\in\Had$ be the reduced adjoint state that solves the reduced adjoint equation in~\cref{equ:reduced-adjoint-system} for the terminal condition~$\hat{\varphi}_\mu(T)=\Projection{\spaceVad}\bar{\varphi}$. Then we have at time~$t\in[0,T]$ that
		\begin{align*}
			\normX{\varphi_\mu(t)-\hat{\varphi}_\mu(t)} \ \leq\ \EstAdjo{\mu}{\bar{\varphi}}(t).
		\end{align*}
		The error estimator~$\EstAdjo{\mu}{\bar{\varphi}}$ is defined as
		\begin{align*}
			\EstAdjo{\mu}{\bar{\varphi}}(t) \coloneqq C_1(\mu)\normX{\bar{\varphi} - \Projection{\spaceVad}\bar{\varphi}} + C_1(\mu)\int\limits_0^{T-t} \normDualX{\Ead^{-1}\ResAdjo{\mu}{\hat{\varphi}_\mu}(T-s)}\d{s},
		\end{align*}
		where~$C_1(\mu)$ is a constant such that the condition in~\cref{equ:constant-c1} is fulfilled.
	\end{lemma}
	\begin{proof}
		The a posteriori error estimate for the adjoint system is a special case of the a posteriori error estimate for the primal system where one has to take into account the reversal of time as well as the error in the terminal instead of the initial condition. For completeness and since we use the same technique for the proof of the a posteriori error estimator of the fully reduced model (see~\Cref{thm:a-posteriori-estimate-fully-reduced-model}) below as well in a more complex version, we shortly recall the main steps of the proof here. The proof follows along the lines of the proof of Proposition~4.1 in~\cite{haasdonk2011efficient}.
		\par
		From the definition of the residual of the adjoint equation, one immediately sees that for all~$t\in[0,T]$ the relation
		\begin{align*}
			\Ead\frac{d}{dt}\hat{\varphi}_\mu(t) = -\A{\mu;t}^*\hat{\varphi}_\mu(t)-\ResAdjo{\mu}{\hat{\varphi}_\mu}(t)
		\end{align*}
		holds. By subtracting this equation from the high-fidelity adjoint equation in~\cref{equ:adjoint-system}, we obtain an ordinary differential equation for the error of the form
		\begin{align*}
			\Ead\frac{d}{dt}\ErrAdjo{\varphi_\mu}{\hat{\varphi}_\mu}(t) = -\A{\mu;t}^*\ErrAdjo{\varphi_\mu}{\hat{\varphi}_\mu}(t)+\ResAdjo{\mu}{\hat{\varphi}_\mu}(t),
		\end{align*}
		where~$\varphi_\mu\in\H$ denotes the solution of~\cref{equ:adjoint-system}. This differential equation has the explicit solution
		\begin{align*}
			\ErrAdjo{\varphi_\mu}{\hat{\varphi}_\mu}(t) = \Riesz{\X}^{-1}\StateTrans{\mu}{T}{t}^*\ErrAdjo{\varphi_\mu}{\hat{\varphi}_\mu}(T)+\int\limits_{0}^{T-t}\Riesz{\X}^{-1}\StateTrans{\mu}{T-t}{s}^*\Ead^{-1}\ResAdjo{\mu}{\hat{\varphi}_\mu}(T-s)\d{s}.
		\end{align*}
		Considering the norm of the error, using the triangle inequality and inserting~$\ErrAdjo{\varphi_\mu}{\hat{\varphi}_\mu}(T)=\bar{\varphi} - \Projection{\spaceVad}\bar{\varphi}$, we obtain the claimed error bound.
	\end{proof}
	\begin{remark}[Numerical approximation of the integrals in the error estimators]
		In practice, the integrals in the definitions of the error estimators need to be approximated using appropriate quadrature rules. For simplicity, we do not take into account the error introduced by the quadrature rule in this paper.
	\end{remark}
	To obtain error estimators that can be evaluated efficiently, we discuss in the following how to compute the norm of the residuals with complexity independent of the dimension of the (high-dimensional) state space~$\X$.
	
	\subsubsection*{Offline-online decomposition of the error estimators}
	In order to obtain an offline-online decomposition, we will consider the square of the respective norms and present their parameter-separable decomposition. To compute the final error estimate, one has to take the square root which might lead to rounding errors for errors close to machine precision, see for instance~\cite{buhr2014numerically} for a similar problem with an error estimator for parametrized elliptic~PDEs. This issue constitutes a potential drawback of the approach.
	\par
	Using the separability of the system components, it is immediately clear that both the residual of the primal and the residual of the adjoint system are parameter-separable as well, i.e.~there exist parameter- and time-independent operators~$R^{\mathrm{pr},q}\in\Lin{\spaceVpr\times\U}{\dualX}$ and continuous functions~$\theta_{R^\mathrm{pr}}^q\colon\params\times[0,T]\to\R$ for~$q=1,\dots,Q_{R^\mathrm{pr}}$ such that
	\begin{align*}
		\ResPrim{\mu}{\hat{x}}{u}(t) = \sum\limits_{q=1}^{Q_{R^\mathrm{pr}}} \theta_{R^\mathrm{pr}}^q(\mu;t)R^{\mathrm{pr},q}(\hat{x}(t),u(t))
	\end{align*}
	and similarly for the adjoint residual. We can therefore rewrite the squared norm of the operator~$\E^{-1}$ applied to the primal residual~$\ResPrim{\mu}{\hat{x}}{u}(t)$ for~$\mu\in\params$, $\hat{x}\in\Hpr$, $u\in\G$ and~$t\in[0,T]$ as
	\begin{align*}
		\normDualX{\E^{-1}\ResPrim{\mu}{\hat{x}}{u}(t)}^2 &= \innerProd{\E^{-1}\ResPrim{\mu}{\hat{x}}{u}(t)}{\E^{-1}\ResPrim{\mu}{\hat{x}}{u}(t)}_{\dualX} \\
		&= \sum\limits_{p,q=1}^{Q_{R^\mathrm{pr}}} \theta_{R^\mathrm{pr}}^p(\mu;t)\theta_{R^\mathrm{pr}}^q(\mu;t)\innerProd{\E^{-1}R^{\mathrm{pr},p}(\hat{x}(t),u(t))}{\E^{-1}R^{\mathrm{pr},q}(\hat{x}(t),u(t))}_{\dualX}.
	\end{align*}
	The quantities~$\innerProd{\E^{-1}R^{\mathrm{pr},p}(\,\cdot\,,\,\cdot\,)}{\E^{-1}R^{\mathrm{pr},q}(\,\cdot\,,\,\cdot\,)}_{\dualX}$ can be precomputed during the offline phase since they are parameter-independent. The norm of the error in the initial condition can be decomposed in a similar way. The full offline-online decomposition of the adjoint residual can be performed equivalently as for the error estimator of the primal system. Details on the offline-online decomposition in the finite-dimensional case are provided in~\Cref{app:details-offline-online-decompositions-primal-and-adjoint}.
	
	\section{Model order reduction for final time optimal adjoint states}\label{sec:reduced-model-optimal-adjoint-states}
	In~\Cref{sec:optimality-system} we have already seen the importance of the optimal final time adjoint state~$\optFTA{\mu}$ for solving the linear-quadratic optimal control problem from~\Cref{sec:parametrized-linear-quadratic-optimal-control-problems}. We describe in this section how to compute a reduced order model to approximate the manifold~$\SolutionManifold$ of optimal final time adjoint states defined in~\cref{equ:manifold}. The starting point will be the characterization of the optimal final time adjoint as solution of the linear system stated in~\cref{equ:linear-system-optimal-final-time-adjoint}. For an approximate final time adjoint state, the residual of the aforementioned linear system will serve as an a posteriori error estimator that is discussed in~\Cref{sec:a-posteriori-rom-for-optimal-adjoint-states}. The reduced order model described in the following section provides a certified approximation of the optimal final time adjoint state. However, as we will see the reduction of the optimal final time adjoint states alone will not lead to a fully efficient reduced scheme since it still involves high-fidelity computations. For more details on the reduced order model described in the following, we refer to~\cite[Section~3]{kleikamp2024greedy}.
	
	\subsection{Projection-based reduced basis model order reduction}\label{sec:reduced-basis-mor-for-final-time-adjoints}
	The approach presented in this section was first proposed in~\cite{lazar2016greedy} and recently transferred to the time-invariant case in~\cite{kleikamp2024greedy}. Its main idea is to build a reduced subspace~$\spaceVred{\dimRedSpaceFTA}\subset\X$ of dimension~$\dimRedSpaceFTA\coloneqq\dim{\spaceVred{\dimRedSpaceFTA}}$ with~$\dimRedSpaceFTA\ll\dim{\X}$ to approximate the manifold~$\SolutionManifold$ of optimal final time adjoint states~$\optFTA{\mu}$ over the parameter set~$\params$. We will later describe how to compute such a reduced space~$\spaceVred{\dimRedSpaceFTA}$ efficiently. For the moment we assume that we are given a reduced space~$\spaceVred{\dimRedSpaceFTA}\subset\X$ and show how to compute an approximate final time adjoint~$\redFTA{\mu}{\dimRedSpaceFTA}\in\spaceVred{\dimRedSpaceFTA}$ for a parameter~$\mu\in\params$. Since the optimal final time adjoint~$\optFTA{\mu}$ can be characterized as solution of the linear system in~\cref{equ:linear-system-optimal-final-time-adjoint}, the approximation~$\redFTA{\mu}{\dimRedSpaceFTA}\approx\optFTA{\mu}$ should fulfill the same linear system as well as possible. Motivated by this idea, we define~$\redFTA{\mu}{\dimRedSpaceFTA}$ as
	\begin{align}\label{equ:least-squares-for-approximate-final-time-adjoint}
		\redFTA{\mu}{\dimRedSpaceFTA} \coloneqq \argmin\limits_{p\in \spaceVred{\dimRedSpaceFTA}}\normX{\MxNaughtMinusxT - \big(\EplusMtimesGramian\big)p}^2.
	\end{align}
	Hence, we can view~$\redFTA{\mu}{\dimRedSpaceFTA}\in\spaceVred{\dimRedSpaceFTA}$ as the least-squares solution of the linear system~\cref{equ:linear-system-optimal-final-time-adjoint} in the space~$\spaceVred{\dimRedSpaceFTA}$. To compute~$\redFTA{\mu}{\dimRedSpaceFTA}$ in practice, we characterize~$\redFTA{\mu}{\dimRedSpaceFTA}$ via the equation
	\begin{align*}
		\big(\EplusMtimesGramian\big)\redFTA{\mu}{\dimRedSpaceFTA} = \Projection{\big(\EplusMtimesGramian\big)\spaceVred{\dimRedSpaceFTA}}\MxNaughtMinusxT.
	\end{align*}
	One therefore needs to compute~$\big(\EplusMtimesGramian\big)\spaceVred{\dimRedSpaceFTA}$ which is computationally expensive due to the involved primal and adjoint systems. To be more precise, the primal and adjoint system have to be solved~$\dimRedSpaceFTA$ times, once for each element of a basis of the reduced space~$\spaceVred{\dimRedSpaceFTA}$. The complexity for solving these systems still scales with the dimension of the (usually high-dimensional) state space~$\X$. This issue limits the speedup possible by applying the reduced order model described in this section. We will therefore later improve the efficiency of the approach in~\Cref{sec:fully-reduced-model} by also applying a reduction of the primal and adjoint systems as depicted in~\Cref{sec:reduced-model-system-dynamics}. Before doing so, we present a reliable a posteriori error estimator for an approximate final time adjoint state by considering the residual of the linear system in~\cref{equ:linear-system-optimal-final-time-adjoint} for the optimal final time adjoint state.
	
	\subsection{A posteriori error estimation}\label{sec:a-posteriori-rom-for-optimal-adjoint-states}
	For a parameter~$\mu\in\params$ and an approximate final time adjoint~$p\in\X$, the error estimator~$\EstFTA{\mu}{\dimRedSpaceFTA}\colon\X\to[0,\infty)$ is defined as
	\begin{align}\label{equ:a-posteriori-error-estimator-reduced-optimal-adjoint}
		\EstFTA{\mu}{\dimRedSpaceFTA}(p) \coloneqq \const\normX{\MxNaughtMinusxT-\big(\EplusMtimesGramian\big)p},
	\end{align}
	where~$\const>0$ is chosen such that
	\begin{align*}
		\const \geq \norm{\big(\EplusMtimesGramian\big)^{-1}}_{\Lin{\X}{\X}}
	\end{align*}
	for all parameters~$\mu\in\params$. We assume here and in the following that such a constant~$\const$ exists that uniformly bounds the operator norm of~$\big(\EplusMtimesGramian\big)^{-1}$ over the parameter space. It is easy to see that the error estimator~$\EstFTA{\mu}{\dimRedSpaceFTA}$ is a reliable estimator, which we state for later reference in the following Lemma.
	\begin{lemma}[A posteriori error estimator for the reduced order model of the final time adjoint states]\label{lem:error-estimator-final-time-adjoint-states}
		For all parameters~$\mu\in\params$ and all approximate final time adjoint states~$p\in\X$ we have
		\begin{align*}
			\normX{\optFTA{\mu}-p}\ \leq\ \EstFTA{\mu}{\dimRedSpaceFTA}(p).
		\end{align*}
	\end{lemma}
	\begin{proof}
		It holds
		\begin{align*}
			\normX{\optFTA{\mu}-p} &= \normX{\big(\EplusMtimesGramian\big)^{-1}\big(\EplusMtimesGramian\big)(\optFTA{\mu}-p)} \\
			&\leq \norm{\big(\EplusMtimesGramian\big)^{-1}}_{\Lin{\X}{\X}}\normX{\big(\EplusMtimesGramian\big)(\optFTA{\mu}-p)} \\
			&\leq \const\normX{\MxNaughtMinusxT-\big(\EplusMtimesGramian\big)p} \\
			&= \EstFTA{\mu}{\dimRedSpaceFTA}(p).
		\end{align*}
	\end{proof}
	Additionally, one can prove that~$\EstFTA{\mu}{\dimRedSpaceFTA}$ is efficient, i.e.~the overestimation of the error produced by the estimator can be bounded by a constant given as the condition number of the operator~$\EplusMtimesGramian$. From the definition of the error estimator~$\EstFTA{\mu}{\dimRedSpaceFTA}$ it furthermore follows directly that the approximate final time adjoint state~$\redFTA{\mu}{\dimRedSpaceFTA}$ is given as~$\redFTA{\mu}{\dimRedSpaceFTA}=\argmin_{p\in \spaceVred{\dimRedSpaceFTA}}\EstFTA{\mu}{\dimRedSpaceFTA}(p)$. Evaluating the error estimator for the approximate final time adjoint, i.e.~computing~$\EstFTA{\mu}{\dimRedSpaceFTA}(\redFTA{\mu}{\dimRedSpaceFTA})$, results in a reliable and efficient estimate for the distance of~$\optFTA{\mu}$ to the reduced space~$\spaceVred{\dimRedSpaceFTA}$, see~\cite[Theorem~3.2]{kleikamp2024greedy}. This result enables the efficient construction of the reduced space~$\spaceVred{\dimRedSpaceFTA}$ using greedy methods, see for instance~\cite{cohen2016kolmogorov,buffa2012priori}.
	\par
	To evaluate the error estimator~$\EstFTA{\mu}{\dimRedSpaceFTA}$ for an approximate adjoint state~$p\in\X$, we mainly have to compute~$\big(\EplusMtimesGramian\big)p$ where we again make use of~\cref{equ:gramian-product}. However, these computations still involve high-dimensional quantities and thus their complexity again depends on the state space~$\X$. This means in particular that even the error estimator cannot be evaluated efficiently with the runtime only depending on the dimension~$\dimRedSpaceFTA$ of the reduced space. Similar to the online phase of the reduced model, we will also speed up the error estimator by reducing the primal and adjoint systems. We will further describe in~\Cref{sec:efficient-a-posteriori-fully-reduced} how to take the additional approximation error of the system dynamics into account to obtain a reliable reduced error estimator that can be evaluated very fast.

	\section{Efficient and fully reduced model for parametrized optimal control problems}\label{sec:fully-reduced-model}
	In the following, we combine the reduction strategies presented in the previous two sections. We hence build a reduced approximation for the state and the adjoint system and additionally
	perform a reduction of the manifold of optimal final time adjoint states. Furthermore, we derive a reliable a posteriori error estimator for this combined reduced model and show its efficient implementation using an offline-online decomposition. We also discuss several strategies for constructing the involved reduced bases. Numerical evaluations of the reduced models and the strategies are given in~\Cref{sec:numerical-experiments}.
	\par
	At this point, we would like to emphasize that combining both approaches has several advantages compared to an individual application of the two reduced models. First of all, the resulting fully reduced model can be solved with complexity independent of the state dimension. This is also true when using the reduced dynamical systems to solve a projected version of the optimal control problem. However, the combination of the approaches allows for an additional reduction of the optimal final time adjoint states and furthermore facilitates the efficient evaluation of a reliable error estimator for the fully reduced model. As we shall see in the numerical experiments, this additional reduction step of constructing a reduced basis for the final time adjoint states results in an additional acceleration of the method while maintaining its accuracy.
	
	\subsection{Fully reduced model by combining system and final time adjoint reduction}\label{sec:description-fully-reduced-model}
	To obtain a fully reduced model that can be evaluated independently of the dimension of the state space, we replace the solution of every primal and adjoint system that occurs in the online phase of the reduced order model for the final time adjoint state, see~\Cref{sec:reduced-basis-mor-for-final-time-adjoints}, by a respective reduced system as shown in~\Cref{sec:reduced-basis-mor-for-system-dynamics}. To this end, assume that we are given reduced spaces~$\spaceVpr\subset\X$, $\spaceWpr\subset\X$, $\spaceVad\subset\X$, and~$\spaceWad\subset\X$ for the system reduction and a reduced space~$\spaceVred{\dimRedSpaceFTA}\subset\X$ for the final time adjoint reduction. Then, the fully reduced model for a given parameter~$\mu\in\params$ is defined as follows: Similar to~\Cref{sec:reduced-basis-mor-for-final-time-adjoints}, we define the fully reduced solution~$\compredFTA{\mu}{\dimRedSpaceFTA}\in\spaceVred{\dimRedSpaceFTA}$ as
	\begin{align}\label{equ:def-fully-reduced-final-time-adjoint}
		\compredFTA{\mu}{\dimRedSpaceFTA} \coloneqq \argmin\limits_{p\in\spaceVred{\dimRedSpaceFTA}}\normX{\RedMxNaughtMinusxT-\left(\EplusMtimesRedGramian\Projection{\spaceVad}\right)p}^2.
	\end{align}
	This least squares problem can be solved as before by~orthogonal projection in~$\X$, i.e.~we have to solve the linear system
	\begin{align*}
		\left(\EplusMtimesRedGramian\Projection{\spaceVad}\right)\compredFTA{\mu}{\dimRedSpaceFTA} = \Projection{\left(\EplusMtimesRedGramian\Projection{\spaceVad}\right)\spaceVred{\dimRedSpaceFTA}}\RedMxNaughtMinusxT.
	\end{align*}
	The required components can be precomputed offline in such a way that it is sufficient to solve a reduced homogeneous problem for~$\RedPrStateTrans{\mu}{T}{0}\Projection{\spaceVpr}x_\mu^0$ and~$\dimRedSpaceFTA$ reduced adjoint and primal problems to compute~$\RedGramian{\mu}\Projection{\spaceVad}\spaceVred{\dimRedSpaceFTA}$ online, where~$\Projection{\spaceVad}\spaceVred{\dimRedSpaceFTA}\subset\spaceVad$ is also precomputed. The computational costs consist of the time required to solve the final linear system which can be represented as a linear system of size~$\dimRedSpaceFTA\times\dimRedSpaceFTA$, the costs for assembling all system components, and the time for solving the reduced primal and adjoint systems for a total of~$\dimRedSpaceFTA$~times. The required computations are in particular independent of the dimension of the state space since only reduced systems need to be solved. More details on the solution of the fully reduced model in the finite-dimensional case are given in~\Cref{app:solution-fully-reduced-model}.
	\par
	We observe at this point that the reduction of the adjoint system dynamics already implies a reduction of the set of possible final time adjoint states. To be more precise, when considering the reduced space~$\spaceVad$ for the adjoint trajectory, the set of final time adjoint states is already restricted to the~$\kad$-dimensional space~$\spaceVad$. Hence, an additional reduction using the subspace~$\spaceVred{\dimRedSpaceFTA}$ is only sensible if~$\dimRedSpaceFTA<\kad$ and if it holds~$\spaceVred{\dimRedSpaceFTA}\subsetneq\spaceVad$. We do not assume or strictly enforce these properties, but instead develop the error estimator for the general formulation stated in the last paragraph. The further reduction of the final time adjoint states can be particularly useful if the adjoint system dynamics require a larger reduced basis to be captured accurately compared to the manifold of optimal final time adjoint states.
	
	\subsection{Efficiently computable a posteriori error estimate}\label{sec:efficient-a-posteriori-fully-reduced}
	In this section we derive an a posteriori error estimator for the fully reduced model. We prove its reliability and show how to evaluate the error estimator efficiently during the online phase by deriving a full offline-online decomposition of all involved quantities.
	
	\subsubsection{Definition and reliability of the error estimator}
	For a parameter~$\mu\in\params$ and reduced solution~$\varphi^\dimRedSpaceFTA\in\spaceVred{\dimRedSpaceFTA}\subset\X$, we define the error estimator
	\begin{align}\label{equ:a-posteriori-error-estimator-fully-reduced-model}
		\FullRedEstFTA{\mu}{\dimRedSpaceFTA}(\varphi^\dimRedSpaceFTA) \coloneqq \const\big(\norm{\Mop}_{\Lin{\X}{\dualX}}\EstPrim{\mu}{0}(T) + \RedEstFTA{\mu}{\dimRedSpaceFTA}(\varphi^\dimRedSpaceFTA) + \norm{\Mop}_{\Lin{\X}{\dualX}}\EstGram{\mu}{\varphi^\dimRedSpaceFTA}\big).
	\end{align}
	Here, the error estimator~$\EstPrim{\mu}{u}$ for the primal system was defined in~\Cref{lem:error-estimator-primal-system}. The reduced error estimator~$\RedEstFTA{\mu}{\dimRedSpaceFTA}(\varphi^\dimRedSpaceFTA)$ for the final time adjoint~$\varphi^\dimRedSpaceFTA$ (which has a similar structure as the original error estimator~$\EstFTA{\mu}{\dimRedSpaceFTA}$ from~\cref{equ:a-posteriori-error-estimator-reduced-optimal-adjoint} but involves only solving reduced systems) is given as
	\begin{align}\label{equ:reduced-estimator-final-time-adjoint}
		\RedEstFTA{\mu}{\dimRedSpaceFTA}(\varphi^\dimRedSpaceFTA) \coloneqq \normX{\RedMxNaughtMinusxT-\left(\EplusMtimesRedGramian\Projection{\spaceVad}\right)\varphi^\dimRedSpaceFTA}.
	\end{align}
	It then holds in particular~$\compredFTA{\mu}{\dimRedSpaceFTA}=\argmin_{p\in\spaceVred{\dimRedSpaceFTA}}\RedEstFTA{\mu}{\dimRedSpaceFTA}(p)$ by combining~\cref{equ:def-fully-reduced-final-time-adjoint} and~\cref{equ:reduced-estimator-final-time-adjoint}. The error estimator~$\EstGram{\mu}{\varphi^\dimRedSpaceFTA}$ for the deviation in the application of the reduced and full Gramian to a vector~$\varphi^\dimRedSpaceFTA\in\spaceVred{\dimRedSpaceFTA}$ is defined as
	\begin{align}\label{equ:def-estimator-gramian}
		\EstGram{\mu}{\varphi^\dimRedSpaceFTA} \coloneqq C_1(\mu)C_2(\mu)\int\limits_0^T\EstAdjo{\mu}{\varphi^\dimRedSpaceFTA}(s)\d{s} + C_1(\mu)\int\limits_0^T\normDualX{\E^{-1}\ResPrim{\mu}{\hat{x}_\mu}{\hat{u}_\mu}(s)}\d{s},
	\end{align}
	where~$(\hat{x}_\mu,\hat{u}_\mu,\hat{\varphi}_\mu)\in\Hpr\times\G\times\Had$ solve the reduced optimality system in~\cref{equ:reduced-primal-system,equ:reduced-control-system,equ:reduced-adjoint-system} with initial condition~$\hat{x}_\mu(0)=0$ and terminal condition~$\hat{\varphi}_\mu(T)=\Projection{\spaceVad}\varphi^\dimRedSpaceFTA$. The constant~$C_2(\mu)$ is chosen such that
	\begin{align*}
		C_2(\mu) \geq \sup\limits_{t\in[0,T]}\norm{\E^{-1}\B{\mu;t}\Rop(t)^{-1}\B{\mu;t}^*}_{\Lin{\X}{\dualX}}.
	\end{align*}
	The second term in~\cref{equ:def-estimator-gramian} corresponds to the residual of the primal equation when computing both the control and the primal solution by means of the reduced adjoint and primal systems, respectively. We then have the following result on the reliability of the error estimator~$\FullRedEstFTA{\mu}{\dimRedSpaceFTA}$:
	\begin{theorem}[A posteriori error estimator for the fully reduced model]\label{thm:a-posteriori-estimate-fully-reduced-model}
		Let~$\mu\in\params$ be a parameter, $\varphi^\dimRedSpaceFTA\in\spaceVred{\dimRedSpaceFTA}$ an element from the reduced space~$\spaceVred{\dimRedSpaceFTA}$ and~$\FullRedEstFTA{\mu}{\dimRedSpaceFTA}$ the error estimator for the fully reduced model defined in~\cref{equ:a-posteriori-error-estimator-fully-reduced-model}. Then, $\FullRedEstFTA{\mu}{\dimRedSpaceFTA}(\varphi^\dimRedSpaceFTA)$ is a reliable estimate of the error between the optimal final time adjoint~$\optFTA{\mu}$ and the approximation~$\varphi^\dimRedSpaceFTA$, i.e.~we have
		\begin{align*}
			\normX{\optFTA{\mu}-\varphi^\dimRedSpaceFTA}\ \leq\ \FullRedEstFTA{\mu}{\dimRedSpaceFTA}(\varphi^\dimRedSpaceFTA).
		\end{align*}
	\end{theorem}
	\begin{proof}
		We first observe that we can bound the reconstruction error using the error estimate from~\Cref{lem:error-estimator-final-time-adjoint-states} in the following way, where we also added some appropriate reduced terms, applied the triangle inequality several times and made use of the definition of the operator norm:
		\begin{align*}
			\normX{\optFTA{\mu}-\Vred\varphi^\dimRedSpaceFTA}&\leq\EstFTA{\mu}{\dimRedSpaceFTA}(\Vred\varphi^\dimRedSpaceFTA)\\
			&=\const\normX{\MxNaughtMinusxT-(\EplusMtimesGramian)\varphi^\dimRedSpaceFTA} \\
			&=\const\left\lVert \MxNaughtMinusxT-\RedMxNaughtMinusxT\right.\\
			&\hphantom{===}\left.+\RedMxNaughtMinusxT-(\EplusMtimesRedGramian\Projection{\spaceVad})\varphi^\dimRedSpaceFTA\right.\\
			&\hphantom{===}\left.+(\EplusMtimesRedGramian\Projection{\spaceVad})\varphi^\dimRedSpaceFTA-(\EplusMtimesGramian)\varphi^\dimRedSpaceFTA\right\rVert_{\dualX} \\
			&\leq\const\left(\norm{\Mop}_{\Lin{\X}{\dualX}}\normX{\StateTrans{\mu}{T}{0}x_\mu^0-\Vpr \RedPrStateTrans{\mu}{T}{0}\Projection{\spaceVpr}x_\mu^0}\right.\\
			&\hphantom{===}+\normX{\RedMxNaughtMinusxT-(\EplusMtimesRedGramian\Projection{\spaceVad})\varphi^\dimRedSpaceFTA}\\
			&\hphantom{===}\left.+\norm{\Mop}_{\Lin{\X}{\dualX}}\normX{(\RedGramian{\mu}\Projection{\spaceVad}-\Gramian{\mu})\varphi^\dimRedSpaceFTA}\right) \\
			&\stackrel{(\ast)}{\leq} \const\left(\norm{\Mop}_{\Lin{\X}{\dualX}}\EstPrim{\mu}{0}(T) + \RedEstFTA{\mu}{\dimRedSpaceFTA}(\varphi^\dimRedSpaceFTA) + \norm{\Mop}_{\Lin{\X}{\dualX}}\EstGram{\mu}{\varphi^\dimRedSpaceFTA}\right)\\
			&=\FullRedEstFTA{\mu}{\dimRedSpaceFTA}(\varphi^\dimRedSpaceFTA).
		\end{align*}
		It is left to prove that the inequality
		\begin{align}\label{equ:estimate-gramian-error}
			\normX{(\RedGramian{\mu}\Projection{\spaceVad}-\Gramian{\mu})\varphi^\dimRedSpaceFTA} \leq \EstGram{\mu}{\varphi^\dimRedSpaceFTA}
		\end{align}
		used in~$(\ast)$ holds. All other estimates follow form the definition of~$\RedEstFTA{\mu}{\dimRedSpaceFTA}(\varphi^\dimRedSpaceFTA)$ and from~\Cref{lem:error-estimator-primal-system} for~$\EstPrim{\mu}{0}(T)$.
		\par
		To prove the estimate~\cref{equ:estimate-gramian-error}, we first use that
		\begin{align*}
			\Gramian{\mu}\varphi^\dimRedSpaceFTA = -x_\mu(T)
		\end{align*}
		according to~\cref{equ:gramian-product}, where~$(x_\mu,u_\mu,\varphi_\mu)\in\H\times\G\times\H$ solve the optimality system in~\cref{equ:primal-system,equ:control-system,equ:adjoint-system} with initial condition~$x_\mu(0)=0$ and terminal condition~$\varphi_\mu(T)=\varphi^\dimRedSpaceFTA$. Similarly, it holds for the reduced Gramian that
		\begin{align*}
			\RedGramian{\mu}\Projection{\spaceVad}\varphi^\dimRedSpaceFTA = -\hat{x}_\mu(T),
		\end{align*}
		according to~\cref{equ:reduced-gramian-product}, where~$(\hat{x}_\mu,\hat{u}_\mu,\hat{\varphi}_\mu)\in\Hpr\times\G\times\Had$ solve the reduced optimality system in~\cref{equ:reduced-control-system,equ:reduced-control-system,equ:reduced-adjoint-system} with initial condition~$\hat{x}_\mu(0)=0$ and terminal condition~$\hat{\varphi}_\mu(T)=\Projection{\spaceVad}\varphi^\dimRedSpaceFTA$. However, due to the different controls~$u_\mu$ and~$\hat{u}_\mu$ used to determine the primal states~$x_\mu(T)$ and~$\hat{x}_\mu(T)$, we cannot directly apply the a posteriori error estimate for the primal system from~\Cref{lem:error-estimator-primal-system} here. Instead, we have to take the difference in the controls into account. 
		We would then like to bound~$\normX{x_\mu(T)-\hat{x}_\mu(T)}=\normX{\ErrPrim{x_\mu}{\hat{x}_\mu}(T)}$ by~$\EstGram{\mu}{\varphi^\dimRedSpaceFTA}$ to finish the proof. To this end, we observe that
		\begin{align}\label{equ:ode-for-hat-x-with-residual}
			\E\frac{d}{dt}\hat{x}_\mu(t) = \A{\mu;t}\hat{x}_\mu(t)+\B{\mu;t}\hat{u}_\mu(t)-\ResPrim{\mu}{\hat{x}_\mu}{\hat{u}_\mu}(t)
		\end{align}
		for
		\begin{align*}
			\ResPrim{\mu}{\hat{x}_\mu}{\hat{u}_\mu}(t) \coloneqq \A{\mu;t}\hat{x}_\mu(t)+\B{\mu;t}\hat{u}_\mu(t)-\E\frac{d}{dt}\hat{x}_\mu(t).
		\end{align*}
		Subtracting equation~\cref{equ:ode-for-hat-x-with-residual} from the primal equation~\cref{equ:primal-system} yields
		\begin{align*}
			\E\frac{d}{dt}\ErrPrim{x_\mu}{\hat{x}_\mu}(t) = \A{\mu;t}\ErrPrim{x_\mu}{\hat{x}_\mu}(t)+\B{\mu;t}\left[u_\mu(t)-\hat{u}_\mu(t)\right]+\ResPrim{\mu}{\hat{x}_\mu}{\hat{u}_\mu}(t).
		\end{align*}
		Together with~$\ErrPrim{x_\mu}{\hat{x}_\mu}(0)=0$ (since~$x_\mu(0)=\hat{x}_\mu(0)=0\in\X$), we obtain the analytic solution of this ordinary differential equation as
		\begin{align*}
			\ErrPrim{x_\mu}{\hat{x}_\mu}(t) = \int\limits_{0}^{t}\StateTrans{\mu}{t}{s}\Riesz{\X}^{-1}\E^{-1}\Big[\B{\mu;s}\left[u_\mu(s)-\hat{u}_\mu(s)\right]+\ResPrim{\mu}{\hat{x}_\mu}{\hat{u}_\mu}(s)\Big]\d{s}.
		\end{align*}
		This results in the estimate
		\begin{align}\label{equ:estimate-norm-primal-error-proof-gramian-estimate}
			\normX{\ErrPrim{x_\mu}{\hat{x}_\mu}(T)} \leq C_1(\mu)\int\limits_{0}^{T}\normDualX{\E^{-1}\B{\mu;s}\left[u_\mu(s)-\hat{u}_\mu(s)\right]}+\normDualX{\E^{-1}\ResPrim{\mu}{\hat{x}_\mu}{\hat{u}_\mu}(s)}\d{s}.
		\end{align}
		It remains to bound the difference in the control. We observe that there holds
		\begin{equation}\label{equ:estimate-control-proof-gramian-estimate}
			\begin{aligned}
				\normDualX{\E^{-1}\B{\mu;s}\left[u_\mu(s)-\hat{u}_\mu(s)\right]} &= \normDualX{-\E^{-1}\B{\mu;s}\Rop(s)^{-1}\big(\B{\mu;s}^*\varphi_\mu(s)+\redBad{\mu;s}^*\hat{\varphi}_\mu(s)\big)} \\
				&= \normDualX{\E^{-1}\B{\mu;s}\Rop(s)^{-1}\big(\B{\mu;s}^*\varphi_\mu(s)-\B{\mu;s}^*\hat{\varphi}_\mu(s)\big)} \\
				&\leq \norm{\E^{-1}\B{\mu;s}\Rop(s)^{-1}\B{\mu;s}^*}_{\Lin{\X}{\dualX}}\normX{\varphi_\mu(s)-\hat{\varphi}_\mu(s)} \\
				&\leq C_2(\mu)\EstAdjo{\mu}{\varphi^\dimRedSpaceFTA}(s)
			\end{aligned}
		\end{equation}
		using the a posteriori error bound for the adjoint system from~\Cref{lem:error-estimator-adjoint-system}. Combining the last two estimates~\cref{equ:estimate-norm-primal-error-proof-gramian-estimate,equ:estimate-control-proof-gramian-estimate} gives the desired result.
	\end{proof}
	\begin{remark}[Improved error estimator using an adjoint problem]\label{rem:improved-error-estimator-adjoint-problem}
		The error estimator presented in the previous Theorem might be improved by means of an output error estimator that makes use of an adjoint problem, see~\cite[Section~5]{haasdonk2011efficient}. At several places we are actually interested in the error after applying the operator~$\Riesz{\X}^{-1}\Mop$ to the state or the Gramian instead of the error in the state or the Gramian itself. A suitable error estimator that makes use of the structure of~$\Mop$ might lead to a further improvement of the error bound.
	\end{remark}
	
	\subsubsection{Offline-online decomposition}
	In contrast to the error estimator described in~\Cref{sec:a-posteriori-rom-for-optimal-adjoint-states} for the reduction of final time adjoint states, the error estimator from the previous section can be decomposed into parameter-independent components in an offline phase and afterwards assembled efficiently in the online phase. The complexity of evaluating this error estimator therefore becomes independent of the dimension of the state space. Consequently, we obtain a reduced model together with a reduced and reliable error estimator that can be evaluated very fast for a new parameter. In the following paragraphs, we show how to decompose the individual parts of the error estimator. We make in particular use of the decompositions for the error estimators of the primal and the adjoint system.
	
	\paragraph{Decomposition of the error estimator for the Gramian:}
	According to the definition of~$\EstGram{\mu}{\varphi^\dimRedSpaceFTA}$, we can decompose the involved quantities by combining the decompositions of the primal residual and of the error estimator for the adjoint system as presented in~\Cref{sec:a-posteriori-for-primal-and-adjoint-systems}. However, to also efficiently compute~$\normX{\varphi^\dimRedSpaceFTA - \Projection{\spaceVad}\varphi^\dimRedSpaceFTA}$, which is required when evaluating~$\EstAdjo{\mu}{\varphi^\dimRedSpaceFTA}$ in the first integral in~\cref{equ:def-estimator-gramian}, we also assemble the parameter-independent operator
	\begin{align*}
		\varphi^\dimRedSpaceFTA \mapsto \innerProd{\left(\eye{\spaceVred{\dimRedSpaceFTA}}-\Projection{\spaceVad}\right)\varphi^\dimRedSpaceFTA}{\left(\eye{\spaceVred{\dimRedSpaceFTA}}-\Projection{\spaceVad}\right)\varphi^\dimRedSpaceFTA}_\X.
	\end{align*}
	This projection error vanishes in particular if~$\spaceVred{\dimRedSpaceFTA}\subset\spaceVad$ is chosen.
	
	\paragraph{Decomposition of the reduced error estimator for the final time adjoint:}
	Similar to the previous sections one can consider the squared norm and can decompose~$\left[\RedEstFTA{\mu}{\dimRedSpaceFTA}(\varphi^\dimRedSpaceFTA)\right]^2$ such that only reduced systems need to be solved. We provide details on the exact derivation in~\Cref{app:details-offline-online-decompositions-reduced-error-estimator}.
	
	\paragraph{Full decomposition of the reduced error estimator:}
	Combining everything from the previous paragraphs and subsections results in a fully efficient and offline-online decomposed error estimator~$\FullRedEstFTA{\mu}{\dimRedSpaceFTA}(\varphi^\dimRedSpaceFTA)$ for a reduced final time adjoint state~$\varphi^\dimRedSpaceFTA\in\spaceVred{\dimRedSpaceFTA}$. The costs for evaluating this error estimator are independent of the dimension of the state space which is a major difference to the reduced model from~\Cref{sec:reduced-model-optimal-adjoint-states}.
	
	\subsection{Strategies for computing the fully reduced model}\label{sec:strategies-computing-fully-reduced-model}
	Different strategies for creating the reduced order models for the primal and adjoint equations and for the final time adjoint states are possible. In~\cite{fabrini2018reduced}, the authors propose to first compute reduced dynamical systems and to run a greedy algorithm using these reduced systems to obtain a reduced basis for the final time adjoint states afterwards. This strategy will be recalled in more detail and extended in~\Cref{sec:system-red-subsequent-greedy}. In~\Cref{sec:greedy-on-full-model-subsequent-system-red}, we consider an algorithm where the order of reduction steps is reversed compared to the first approach. In addition, we present in~\Cref{sec:greedy-fully-reduced} a greedy algorithm that operates only on the fully reduced model and builds reduced bases for the system dynamics and the final time adjoint states simultaneously using the error estimator from~\Cref{sec:efficient-a-posteriori-fully-reduced}. Finally, we also propose a double greedy procedure in which the Gramian error estimator is used in the inner loop to enrich the reduced spaces for the primal and adjoint systems whereas the reduced space for the final time adjoint states is enriched within the outer loop of the algorithm.
	
	\subsubsection{System reduction and subsequent greedy algorithm with reduced dynamical systems for optimal final time adjoint states}\label{sec:system-red-subsequent-greedy}
	The first strategy to construct the fully reduced model was previously introduced in~\cite{fabrini2018reduced}. In this approach, one first performs a reduction of the primal and the adjoint system. The required reduced spaces~$\spaceVpr$, $\spaceWpr$, $\spaceVad$, and~$\spaceWad$ can e.g.\ be computed by compressing a precomputed set of trajectories of the primal and adjoint system using, for instance, the proper orthogonal decomposition (POD) method~\cite{graessle2021model}. The~POD extracts dominant modes of the given trajectories similar to the leading left singular vectors of a singular value decomposition of the training data. Compared to the~POD, its hierarchical approximation variant called~HaPOD~\cite{himpe2018hierarchical} is better suited for our use case since it can also handle long time trajectories and high state space dimensions without leading to runtime and memory issues. The~HaPOD takes as inputs a set of vectors, a tolerance and an inner product where the tolerance refers to the desired mean squared approximation error (similar to the usual criterion for the~POD). The orthonormalization of the modes within the~HaPOD is performed with respect to the provided inner product. As the control for the primal trajectories, we use the optimal control computed by solving the full model. We choose the optimal control since we are later mainly interested in simulating the primal system for the optimal control. However, different strategies are possible in that regard as well. We further use here and in the following the same parameter set~$\paramstrainsys\subset\params$ and the same tolerance for the reduction of both the primal and adjoint trajectories. One might also consider different training sets or tolerances instead.
	\par
	After the system reduction, a greedy algorithm using the error estimator~$\FullRedEstFTA{\mu}{\dimRedSpaceFTA}$ from~\Cref{sec:efficient-a-posteriori-fully-reduced} can be applied to construct the reduced basis for the optimal final time adjoint states. We refer for instance to~\cite{buffa2012priori} for details on greedy algorithms in the context of parametrized systems and~\cite{devore2013greedy,cohen2016kolmogorov} for convergence results of these methods. The greedy algorithm applied here uses only the reduced primal and adjoint system to compute an error estimate for the greedy selection of the next parameter. Within the greedy algorithm, having selected a certain parameter, it is now possible to either solve the full order model for the selected parameter or the reduced model using only reduced dynamical systems. Since the greedy algorithm is performed during the offline phase and computational costs of this phase play a minor role, we decided to use the more costly approach of solving the full order model which might lead to an improved reduced basis for the optimal final time adjoint states. Hence, using the greedy algorithm, a reduced basis is built iteratively by solving in every step the full order model for the parameter whose solution is currently approximated worst by the fully reduced model, according to the reduced error estimator. The full order solution is afterwards used to extend the reduced basis. This iterative procedure is stopped once the estimated error drops below a given tolerance for all parameters in a finite training set.
	\par
	We emphasize at this point that we make use of the same training set~$\paramstrainsys\subset\params$ for the system reduction and the greedy procedure. This is necessary since the reduced bases for the system reduction cannot be adjusted during the greedy procedure anymore if they are not sufficient in order to reach the prescribed tolerance. If we were to choose a larger training set during the greedy algorithm, the reduced models for the system dynamics are not able to guarantee the prescribed tolerance~$\epssys$ also on the larger training set for the final time adjoint reduction. To preclude issues with the greedy algorithm not terminating as good as possible, we thus choose the same training set for all reduction steps here. However, it might still happen that the prescribed tolerance~$\epsFTA$ cannot be reached during the greedy iteration.
	\par
	We refer to the reduced model computed by using the strategy described in this section as~\SRGROM{}. The pseudocode for the approach is shown in~\Cref{alg:SRGROM}.
	\begin{algo}{Strategy for creating the~\SRGROM{}}{alg:SRGROM}
		\Require Tolerances~$\epssys,\epsFTA>0$, finite training set~$\paramstrainsys\subset\params$
		\Ensure Reduced spaces~$\spaceVpr$, $\spaceWpr$, $\spaceVad$, $\spaceWad$ and~$\spaceVred{\dimRedSpaceFTA}$ which define the fully reduced model called~\SRGROM{}
		\State $\spaceVpr\gets\Span{\Call{HaPOD}{\{x_\mu^*:\mu\in\paramstrainsys\},\epssys,\innerProd{\cdot}{\cdot}_\X}}$\Comment{primal system reduction to determine~$\spaceVpr$}
		\State choose~$\spaceWpr$ such that~$\redEpr$ is invertible, see also~\Cref{sec:reduced-basis-mor-for-system-dynamics}\Comment{choose~$\spaceWpr$ accordingly}
		\State $\spaceVad\gets\Span{\Call{HaPOD}{\{\varphi_\mu^*:\mu\in\paramstrainsys\},\epssys,\innerProd{\cdot}{\cdot}_\X}}$\Comment{adjoint system reduction to determine~$\spaceVad$}
		\State choose~$\spaceWad$ such that~$\redEad$ is invertible, see also~\Cref{sec:reduced-basis-mor-for-system-dynamics}\Comment{choose~$\spaceWad$ accordingly}
		\State $\dimRedSpaceFTA\gets 0$
		\State $\spaceVred{\dimRedSpaceFTA}\gets\{0\}\subset\X$\Comment{initialize~$\spaceVred{\dimRedSpaceFTA}$}
		\State $\mu_1\gets\argmax_{\mu\in\paramstrainsys}\FullRedEstFTA{\mu}{0}(\compredFTA{\mu}{0})$
		\While{$\FullRedEstFTA{\mu_{\dimRedSpaceFTA+1}}{\dimRedSpaceFTA}(\compredFTA{\mu_{\dimRedSpaceFTA+1}}{\dimRedSpaceFTA})>\epsFTA$}\label{lst:while-loop}\Comment{greedy iteration}
		\State $\spaceVred{\dimRedSpaceFTA}\gets\Span{\spaceVred{\dimRedSpaceFTA}\cup\{\optFTA{\mu_{\dimRedSpaceFTA+1}}\}}$\Comment{enrich reduced space~$\spaceVred{\dimRedSpaceFTA}$}
		\State $\dimRedSpaceFTA\gets\dimRedSpaceFTA+1$
		\State $\mu_{\dimRedSpaceFTA+1}\gets\argmax_{\mu\in\paramstrainsys}\FullRedEstFTA{\mu}{\dimRedSpaceFTA}(\compredFTA{\mu}{\dimRedSpaceFTA})$\Comment{greedy selection step}
		\EndWhile
		\Return{$\spaceVpr$, $\spaceWpr$, $\spaceVad$, $\spaceWad$, $\spaceVred{\dimRedSpaceFTA}$}
	\end{algo}
	
	\begin{remark}[Greedy algorithm for system reductions]\label{rem:greedy-for-system-reductions}
		As another possibility to construct the reduced spaces for the system dynamics, i.e.~the space~$\spaceVpr$ and~$\spaceVad$, it is possible to apply a greedy algorithm using the error estimators for the primal and the adjoint equations, respectively. Within the greedy algorithm, one can once again make use of the HaPOD to compress the reduced bases extended by individual solution trajectories to remove linearly dependent elements from the basis and to keep the basis size small. This idea is similar to the~POD-greedy method, see~\cite{haasdonk2008reduced}. We refrain from considering this alternative approach (which is also applicable to the other strategies mentioned below) in more detail.
	\end{remark}
	\begin{remark}[Termination of the greedy algorithm]
		Due to the approximation of the dynamical systems and the application of the reduced error estimator, it is not guaranteed that the~\textbf{while}-loop in~Line~\ref{lst:while-loop} of~\Cref{alg:SRGROM} terminates, see also the discussion above on the choice of the training parameter sets. In a practical implementation of the approach (and similarly of the strategies described below), one should therefore make sure to handle such a case properly. For instance, if~$\dimRedSpaceFTA$ reaches the value of~$\kad$, an extension of the primal and adjoint reduced space~$\spaceVpr$ and~$\spaceVad$ using a smaller tolerance~$\epssys$ and a (potentially) larger training parameter set~$\paramstrainsys$ in the HaPOD might be useful. Alternatively, running a greedy procedure as described in~\Cref{rem:greedy-for-system-reductions} could solve the issue by identifying specifically those parameters for which the system dynamics are approximated insufficiently thus far.
	\end{remark}
	\begin{remark}[Replacing the greedy algorithm by POD/HaPOD]
		Instead of constructing the reduced space~$\spaceVred{\dimRedSpaceFTA}$ by means of a greedy search as shown above, one could also make use of the optimal final time adjoints computed for all parameters in~$\paramstrainsys$ and perform a~POD or~HaPOD of those adjoint states. This could be a reasonable approach in a case where the error estimator~$\FullRedEstFTA{\mu}{\dimRedSpaceFTA}$ is not available.
	\end{remark}
	
	\subsubsection{Greedy algorithm for reducing optimal final time adjoint states using full dynamical systems and subsequent system reductions}\label{sec:greedy-on-full-model-subsequent-system-red}
	The second strategy for computing the fully reduced model, called~\GSRROM{} in the following, proceeds similar to the~\SRGROM{} strategy in the sense that the two reduction steps are split. However, in the~\GSRROM{}, the order of reduction is swapped. Here, we first run the standard greedy algorithm to construct a reduced space~$\spaceVred{\dimRedSpaceFTA}$ for the optimal final time adjoint states using the a posteriori error estimator~$\EstFTA{\mu}{\dimRedSpaceFTA}$ from~\Cref{sec:a-posteriori-rom-for-optimal-adjoint-states}. Afterwards, the system reductions using the HaPOD as described in the previous section are performed independently from the first step. In principle, both reduction steps can also be performed in parallel.
	\par
	An algorithmic description of the~\GSRROM{} construction strategy is given in~\Cref{alg:GROMSR}.
	\begin{algo}{Strategy for creating the~\GSRROM{}}{alg:GROMSR}
		\Require Tolerances~$\epssys,\epsFTA>0$, finite training sets~$\paramstrainsys,\paramstrainfta\subset\params$
		\Ensure Reduced spaces~$\spaceVpr$, $\spaceWpr$, $\spaceVad$, $\spaceWad$ and~$\spaceVred{\dimRedSpaceFTA}$ which define the fully reduced model called~\GSRROM{}
		\State $\dimRedSpaceFTA\gets 0$
		\State $\redBasisVred{\dimRedSpaceFTA}\gets\{0\}\subset\X,\quad\spaceVred{\dimRedSpaceFTA}\gets\Span{\redBasisVred{\dimRedSpaceFTA}}\subset\X$\Comment{initialize~$\redBasisVred{\dimRedSpaceFTA}$ and~$\spaceVred{\dimRedSpaceFTA}$}
		\State $\mu_1\gets\argmax_{\mu\in\paramstrainfta}\EstFTA{\mu}{0}(\redFTA{\mu}{0})$
		\While{$\EstFTA{\mu_{\dimRedSpaceFTA+1}}{\dimRedSpaceFTA}(\redFTA{\mu_{\dimRedSpaceFTA+1}}{\dimRedSpaceFTA})>\epsFTA$}\Comment{greedy iteration}
		\State $\redBasisVred{\dimRedSpaceFTA}\gets\redBasisVred{\dimRedSpaceFTA}\cup\{\optFTA{\mu_{\dimRedSpaceFTA+1}}\}$\Comment{enrich reduced basis~$\redBasisVred{\dimRedSpaceFTA}$, perform orthonormalization if desired}
		\State $\spaceVred{\dimRedSpaceFTA}\gets\Span{\redBasisVred{\dimRedSpaceFTA}}$\Comment{update reduced space~$\spaceVred{\dimRedSpaceFTA}$}
		\State $\dimRedSpaceFTA\gets\dimRedSpaceFTA+1$
		\State $\mu_{\dimRedSpaceFTA+1}\gets\argmax_{\mu\in\paramstrainfta}\EstFTA{\mu}{\dimRedSpaceFTA}(\redFTA{\mu}{\dimRedSpaceFTA})$\Comment{greedy selection step}
		\EndWhile
		\State \vphantom{$\tilde{V}$}$\spaceVpr\gets\Span{\Call{HaPOD}{\{x_\mu^*:\mu\in\paramstrainsys\},\epssys,\innerProd{\cdot}{\cdot}_\X}}$\Comment{primal system reduction to determine~$\spaceVpr$}
		\State choose~$\spaceWpr$ such that~$\redEpr$ is invertible, see also~\Cref{sec:reduced-basis-mor-for-system-dynamics}\Comment{choose~$\spaceWpr$ accordingly}
		\State $\spaceVad\gets\Span{\Call{HaPOD}{\{\varphi_\mu^*:\mu\in\paramstrainsys\},\epssys,\innerProd{\cdot}{\cdot}_\X}}$\Comment{adjoint system reduction to determine~$\redBasisVad$}
		\State choose~$\spaceWad$ such that~$\redEad$ is invertible, see also~\Cref{sec:reduced-basis-mor-for-system-dynamics}\Comment{choose~$\spaceWad$ accordingly}
		\Return{$\spaceVpr$, $\spaceWpr$, $\spaceVad$, $\spaceWad$, $\spaceVred{\dimRedSpaceFTA}$}
	\end{algo}
	
	\subsubsection{Greedy algorithm on the fully reduced model}\label{sec:greedy-fully-reduced}
	Instead of reducing the dynamical systems and the final time adjoint states separately, it is also possible to run a greedy algorithm using the error estimator~$\FullRedEstFTA{\mu}{\dimRedSpaceFTA}$ to select the next parameter for enriching all reduced models at once. Afterwards, the reduced bases for the optimal final time adjoint states and also for the primal and adjoint systems are created iteratively by selecting the next parameter according to the error estimator~$\FullRedEstFTA{\mu}{\dimRedSpaceFTA}$ of the current fully reduced model. The greedy selection procedure is similar to the one described in~\Cref{sec:system-red-subsequent-greedy}. Within one iteration of the greedy algorithm, all reduced models are extended simultaneously.
	\par
	In contrast to the previously shown strategies, running the greedy algorithm directly on the fully reduced model as suggested here does not require to solve the primal and adjoint systems for all parameters in the training set~$\paramstrainsys$. Instead, only the trajectories corresponding to the parameters selected within the greedy iteration are required. Depending on the problem at hand, this might result in a faster offline phase for constructing the reduced order model compared to the previous approaches.
	\par
	The reduced model obtained by applying a greedy algorithm directly on the fully reduced model, using the error estimator~$\FullRedEstFTA{\mu}{\dimRedSpaceFTA}$ for the selection of parameters, will be referred to as~\GCROM{}. \Cref{alg:GCROM} presents the~\GCROM{} strategy.
	\begin{algo}{Strategy for creating the~\GCROM{}}{alg:GCROM}
		\Require Tolerances~$\eps,\epssys>0$, finite training set~$\paramstrain\subset\params$
		\Ensure Reduced spaces~$\spaceVpr$, $\spaceWpr$, $\spaceVad$, $\spaceWad$ and~$\spaceVred{\dimRedSpaceFTA}$ which define the fully reduced model called~\GCROM{}
		\State $\dimRedSpaceFTA\gets 0$
		\State $\redBasisVred{\dimRedSpaceFTA}\gets\{0\}\subset\X,\quad\spaceVred{\dimRedSpaceFTA}\gets\Span{\redBasisVred{\dimRedSpaceFTA}}\subset\X$\Comment{initialize~$\redBasisVred{\dimRedSpaceFTA}$ and~$\spaceVred{\dimRedSpaceFTA}$}
		\State $\redBasisVpr\gets\{0\}\subset\X,\quad\spaceVpr\gets\Span{\redBasisVpr}\subset\X$\Comment{initialize~$\redBasisVpr$ and~$\spaceVpr$}
		\State $\redBasisWpr\gets\{0\}\subset\X,\quad\spaceWpr\gets\Span{\redBasisWpr}\subset\X$\Comment{initialize~$\redBasisWpr$ and~$\spaceWpr$}
		\State $\redBasisVad\gets\{0\}\subset\X,\quad\spaceVad\gets\Span{\redBasisVad}\subset\X$\Comment{initialize~$\redBasisVad$ and~$\spaceVad$}
		\State $\redBasisWad\gets\{0\}\subset\X,\quad\spaceWad\gets\Span{\redBasisWad}\subset\X$\Comment{initialize~$\redBasisWad$ and~$\spaceWad$}
		\State $\mu_1\gets\argmax_{\mu\in\paramstrain}\FullRedEstFTA{\mu}{0}(\compredFTA{\mu}{0})$
		\While{$\FullRedEstFTA{\mu_{\dimRedSpaceFTA+1}}{\dimRedSpaceFTA}(\compredFTA{\mu_{\dimRedSpaceFTA+1}}{\dimRedSpaceFTA})>\eps$}\Comment{greedy iteration}
		\State $\redBasisVred{\dimRedSpaceFTA}\gets\redBasisVred{\dimRedSpaceFTA}\cup\{\optFTA{\mu_{\dimRedSpaceFTA+1}}\}$\Comment{enrich reduced basis~$\redBasisVred{\dimRedSpaceFTA}$, perform orthonormalization if desired}
		\State $\spaceVred{\dimRedSpaceFTA}\gets\Span{\redBasisVred{\dimRedSpaceFTA}}$\Comment{update reduced space~$\spaceVred{\dimRedSpaceFTA}$}
		\State $\redBasisVpr\gets\Call{HaPOD}{\redBasisVpr\cup\{x_{\mu_{\dimRedSpaceFTA+1}}^*\},\epssys,\innerProd{\cdot}{\cdot}_\X}$\Comment{primal system reduction}
		\State $\spaceVpr\gets\Span{\redBasisVpr}$\Comment{update primal reduced space~$\spaceVpr$}
		\State choose~$\spaceWpr$ such that~$\redEpr$ is invertible, see also~\Cref{sec:reduced-basis-mor-for-system-dynamics}\Comment{choose~$\spaceWpr$ accordingly}
		\State $\redBasisVad\gets\Call{HaPOD}{\redBasisVad\cup\{\varphi_{\mu_{\dimRedSpaceFTA+1}}^*\},\epssys,\innerProd{\cdot}{\cdot}_\X}$\Comment{adjoint system reduction}
		\State $\spaceVad\gets\Span{\redBasisVad}$\Comment{update adjoint reduced space~$\spaceVad$}
		\State choose~$\spaceWad$ such that~$\redEad$ is invertible, see also~\Cref{sec:reduced-basis-mor-for-system-dynamics}\Comment{choose~$\spaceWad$ accordingly}
		\State $\dimRedSpaceFTA\gets\dimRedSpaceFTA+1$
		\State $\mu_{\dimRedSpaceFTA+1}\gets\argmax_{\mu\in\paramstrain}\FullRedEstFTA{\mu}{\dimRedSpaceFTA}(\compredFTA{\mu}{\dimRedSpaceFTA})$\Comment{greedy selection step}
		\EndWhile
		\Return{$\spaceVpr$, $\spaceWpr$, $\spaceVad$, $\spaceWad$, $\spaceVred{\dimRedSpaceFTA}$}
	\end{algo}
	\begin{remark}[Statements on the convergence of the greedy algorithm]
		Compared to the greedy algorithm described in~\cite{kleikamp2024greedy} for constructing a reduced basis for the manifold of optimal final time adjoint states, there are no performance guarantees for the greedy algorithm on the fully reduced model, i.e.~\Cref{alg:GCROM}, available. Such theoretical statements are (so far) impossible due to the missing efficiency of the error estimator~$\FullRedEstFTA{\mu}{\dimRedSpaceFTA}$. We hence cannot guarantee convergence or even state convergence rates for the greedy algorithm presented in this section.
	\end{remark}
	\begin{remark}[Greedy algorithm taking into account individual error contributions]
		The~\GCROM{} approach can be refined in the sense that the different components of the error estimator~$\FullRedEstFTA{\mu}{\dimRedSpaceFTA}$ in~\cref{equ:a-posteriori-error-estimator-fully-reduced-model} can be used to extend the different reduced bases individually. The error components could provide hints regarding which reduced basis, i.e.~the basis for the optimal final time adjoints (using the reduced error estimator~$\RedEstFTA{\mu}{\dimRedSpaceFTA}$) or the bases for the system dynamics (using the remaining parts of~$\FullRedEstFTA{\mu}{\dimRedSpaceFTA}$), should be extended due to a large contribution to the error estimate. The main goal of these separate basis extensions would be to retain the accuracy of the~\GCROM{} while obtaining smaller reduced models if possible. In that case, this reduced model would be similarly accurate but faster (in the offline construction and the online evaluation) compared to the~\GCROM{}. We do not discuss this approach in more detail to keep the presentation lucid. However, the double greedy approach that will be introduced in the next section follows a similar idea.
	\end{remark}
	
	\subsubsection{Double greedy algorithm}\label{sec:double-greedy}
	As the last approach discussed in this paper, we propose a double greedy procedure to construct the fully reduced model. In the context of reduced order modeling for inf-sup stable Petrov-Galerkin schemes and transport dominated problems, a double greedy algorithm has been considered in~\cite{dahmen2014double}.
	Compared to the~\GCROM{}, we do not only consider the state trajectories for the parameters chosen in
	the reduction of the optimal final time adjoint states but instead allow for an enrichment with
	trajectories corresponding to arbitrary parameters in the training set.
	This leads to a nested structure where an
	outer greedy loop enriches the reduced basis for the final time adjoint states
	while an inner greedy loop that follows each update ensures that the system dynamics
	for the new final time adjoint and arbitrary training parameters are still sufficiently
	captured by the reduced states. The reduced spaces for the primal and adjoint systems are thus tailored to the current reduced space for the final time adjoint states. As before, the outer greedy loop uses the reduced error estimator~$\FullRedEstFTA{\mu}{\dimRedSpaceFTA}$ to determine when the fully reduced model is sufficiently accurate.
	The inner loop considers only the Gramian error estimator (multiplied by~$\const$ and~$\norm{\Mop}_{\Lin{\X}{\dualX}}$ to mimic the scaling in the full error estimator~\cref{equ:a-posteriori-error-estimator-fully-reduced-model}) and enriches the reduced spaces for the system reductions until the estimated Gramian error drops below a prescribed tolerance. Here, one might in particular benefit from a tailored output error estimation as suggested in~\Cref{rem:improved-error-estimator-adjoint-problem}.
	\par
	In contrast to the~\GCROM{}, this strategy is motivated by the idea that the reduced spaces for the system dynamics can be constructed based on the reduced space for the final time adjoint states. This furthermore enables the usage of the Gramian error estimator and allows for treating the involved reduced spaces differently. However, we expect the double greedy algorithm to result in larger reduced spaces than the~\GCROM{} since the reduced spaces for the dynamical systems are constructed based on a complete training set and are not restricted to the optimal final time adjoints for the respective parameters. On the other hand, the reduced spaces might be more ``stable'' in the sense that the overestimation of the error by the corresponding error estimator might be smaller compared to the other fully reduced models.
	\par
	We refer to the resulting reduced models as~\DGROM{} and present the strategy in~\Cref{alg:DGROM}.
	\begin{algo}{Strategy for creating the~\DGROM{}}{alg:DGROM}
		\Require Tolerances~$\eps,\epsinner,\epssys>0$, finite training set~$\paramstrain\subset\params$ for the outer iteration, finite training set~$\paramstraininner\subset\params$ for the inner iteration
		\Ensure Reduced spaces~$\spaceVpr$, $\spaceWpr$, $\spaceVad$, $\spaceWad$ and~$\spaceVred{\dimRedSpaceFTA}$ which define the fully reduced model called~\DGROM{}
		\State $\dimRedSpaceFTA\gets 0$
		\State $\redBasisVred{\dimRedSpaceFTA}\gets\{0\}\subset\X,\quad\spaceVred{\dimRedSpaceFTA}\gets\Span{\redBasisVred{\dimRedSpaceFTA}}\subset\X$\Comment{initialize~$\redBasisVred{\dimRedSpaceFTA}$ and~$\spaceVred{\dimRedSpaceFTA}$}
		\State $\redBasisVpr\gets\{0\}\subset\X,\quad\spaceVpr\gets\Span{\redBasisVpr}\subset\X$\Comment{initialize~$\redBasisVpr$ and~$\spaceVpr$}
		\State $\redBasisWpr\gets\{0\}\subset\X,\quad\spaceWpr\gets\Span{\redBasisWpr}\subset\X$\Comment{initialize~$\redBasisWpr$ and~$\spaceWpr$}
		\State $\redBasisVad\gets\{0\}\subset\X,\quad\spaceVad\gets\Span{\redBasisVad}\subset\X$\Comment{initialize~$\redBasisVad$ and~$\spaceVad$}
		\State $\redBasisWad\gets\{0\}\subset\X,\quad\spaceWad\gets\Span{\redBasisWad}\subset\X$\Comment{initialize~$\redBasisWad$ and~$\spaceWad$}
		\State $\mu_1\gets\argmax_{\mu\in\paramstrain}\FullRedEstFTA{\mu}{0}(\compredFTA{\mu}{0})$
		\While{$\FullRedEstFTA{\mu_{\dimRedSpaceFTA+1}}{\dimRedSpaceFTA}(\compredFTA{\mu_{\dimRedSpaceFTA+1}}{\dimRedSpaceFTA})>\eps$}\Comment{greedy iteration}
		\State $\redBasisVred{\dimRedSpaceFTA}\gets\redBasisVred{\dimRedSpaceFTA}\cup\{\optFTA{\mu_{\dimRedSpaceFTA+1}}\}$\Comment{enrich reduced basis~$\redBasisVred{\dimRedSpaceFTA}$, perform orthonormalization if desired}
		\State $\spaceVred{\dimRedSpaceFTA}\gets\Span{\redBasisVred{\dimRedSpaceFTA}}$\Comment{update reduced space~$\spaceVred{\dimRedSpaceFTA}$}
		\State $\mu^\mathrm{inner}\gets\argmax_{\mu\in\paramstraininner}\EstGram{\mu}{\compredFTA{\mu}{\dimRedSpaceFTA}}$
		\While{$\const\norm{\Mop}_{\Lin{\X}{\dualX}}\EstGram{\mu^\mathrm{inner}}{\compredFTA{\mu^\mathrm{inner}}{\dimRedSpaceFTA}}>\epsinner$}\Comment{inner greedy iteration}
		\State \vphantom{$\tilde{V}$}$\redBasisVpr\gets\Call{HaPOD}{\redBasisVpr\cup\{x_{\mu^\mathrm{inner}}^*\},\epssys,\innerProd{\cdot}{\cdot}_\X}$\Comment{primal system reduction}
		\State $\spaceVpr\gets\Span{\redBasisVpr}$\Comment{update primal reduced space~$\spaceVpr$}
		\State choose~$\spaceWpr$ such that~$\redEpr$ is invertible, see also~\Cref{sec:reduced-basis-mor-for-system-dynamics}\Comment{choose~$\spaceWpr$ accordingly}
		\State $\redBasisVad\gets\Call{HaPOD}{\redBasisVad\cup\{\varphi_{\mu^\mathrm{inner}}^*\},\epssys,\innerProd{\cdot}{\cdot}_\X}$\Comment{adjoint system reduction}
		\State $\spaceVad\gets\Span{\redBasisVad}$\Comment{update adjoint reduced space~$\spaceVad$}
		\State choose~$\spaceWad$ such that~$\redEad$ is invertible, see also~\Cref{sec:reduced-basis-mor-for-system-dynamics}\Comment{choose~$\spaceWad$ accordingly}
		\State $\mu^\mathrm{inner}\gets\argmax_{\mu\in\paramstraininner}\EstGram{\mu}{\compredFTA{\mu}{\dimRedSpaceFTA}}$\Comment{inner greedy selection step}
		\EndWhile
		\State \vphantom{$\tilde{V}$}$\dimRedSpaceFTA\gets\dimRedSpaceFTA+1$
		\State $\mu_{\dimRedSpaceFTA+1}\gets\argmax_{\mu\in\paramstrain}\FullRedEstFTA{\mu}{\dimRedSpaceFTA}(\compredFTA{\mu}{\dimRedSpaceFTA})$\Comment{greedy selection step}
		\EndWhile
		\Return{$\spaceVpr$, $\spaceWpr$, $\spaceVad$, $\spaceWad$, $\spaceVred{\dimRedSpaceFTA}$}
	\end{algo}

    \section{Numerical experiment}\label{sec:numerical-experiments}
    In this section we verify the fully reduced model and the different reduction strategies described in the previous section on a numerical example. As a test case we consider a time-varying version of a thermalblock-type experiment called the~\emph{cookie baking problem} which is available in the MORwiki benchmark collection~\cite{morwiki_thermalblock}. Before describing the problem setting in detail and presenting numerical results, we provide some implementational details.
    
    \subsection{Implementational details}
    The implementation used for the numerical example in this section is based on the~\texttt{Python} programming language and uses the packages~\texttt{numpy}~\cite{harris2020array} and~\texttt{scipy}~\cite{virtanen2020SciPy} for dealing with matrices and vectors. The development of the methods described in this contribution is in particular motivated by the high-dimensional state space that occurs in the discretization of PDEs in multiple space dimensions. To be able to treat such dynamical systems at all, it is necessary to exploit the sparsity of the involved system matrices. This task is accomplished by means of the routines provided in the~\texttt{sparse} package in~\texttt{scipy}. For example, we apply the biconjugate gradient stabilized~(Bi-CGSTAB) method, see for instance~\cite{vanderVorst1992bicgstab,saad2003iterative}, to solve the linear system in the full order problem. The relative tolerance of the~Bi-CGSTAB~algorithm is set to~$10^{-8}$.
    \par
    In the algorithms described in~\Cref{sec:strategies-computing-fully-reduced-model} we apply the~HaPOD algorithm implemented in the model order reduction software~\texttt{pyMOR}~\cite{milk2016pyMOR}. In this numerical experiment, we apply the distributed version of the~HaPOD as described in detail in~\cite[Section~3.2]{himpe2018hierarchical}. As hyperparameters for the algorithm we set the number of slides to~$50$, the desired~$\ell^2$-mean approximation error to~$10^{-9}$, and the tuning parameter~$\omega$ to~$0.9$ (which prefers small reduced bases sizes while accepting a larger computational effort for computing the bases). Numerical tests using the incremental~HaPOD show similar results as presented below. Moreover, we use an inner product matrix~$\Gmat\in\R^{n\times n}$, where~$n\in\N$ denotes the dimension of the state space, for orthonormalization in the HaPOD. The matrix~$\Gmat$ will be specified below in the context of the concrete discretization of the problem.
    \par
    The matrices~$\Vpr$ and~$\Wpr$ (and similarly for~$\Vad$ and~$\Wad$) are constructed in such a way that they are biorthogonal as discussed in~\Cref{rem:matrix-case}. To this end, whenever necessary we initialize the biorthogonal operators~$\Vpr$, $\Wpr$, $\Vad$, and~$\Wad$ by compressing the respective (optimal) solution trajectories for a single parameter~$\mu_\text{init}\in\params$ using the HaPOD~method.
    \par
    The finite element mesh used in the test case below is generated using \texttt{gmsh}~\cite{geuzaine2024gmsh} (with the~\texttt{clscale} parameter set to~$0.1$ resulting in~$13142$ cells) and the discretization with~$P^1$-finite elements on the triangulation of the domain is performed using \texttt{FEniCS}~\cite{alnaes2015fenics,logg2012automated}.
    \par
    The numerical experiments have been conducted on a dual socket compute server with two Intel(R) Xeon(R) Gold~6254~CPUs running at~3.10GHz and 36~cores in each CPU.
    \par
    Furthermore, we also provide the source code for the numerical test case~\cite{sourcecode} which can be used to reproduce the numerical results\footnote{The corresponding \texttt{GitHub}-repository containing the source code is available at~\url{https://github.com/HenKlei/REDUCED-OPT-CONTROL}}.
    
    \subsection{The cookie baking problem}
    As numerical test case we consider a time-varying version of a benchmark problem known as the~\emph{cookie baking problem}. The benchmark is introduced in~\cite{rave2021nonstationary} and can be found in the MORwiki~\cite{morwiki_thermalblock}. In the following we briefly describe the problem setting and the optimal control problem considered in our numerical example. The presentation follows~\cite{rave2021nonstationary}. In contrast to the original benchmark, we consider a time-dependent heat conductivity with a slightly adjusted parametrization as shown below.
    \par
    We partition the domain~$\Omega=(0,1)^2$ into five subdomains of the form
    \begin{align*}
    	\Omega_1 &\coloneqq \{\xi\in\Omega:\norm{\xi-(0.3,0.3)}_2<0.1\}, &\Omega_2 &\coloneqq \{\xi\in\Omega:\norm{\xi-(0.7,0.3)}_2<0.1\},\\
    	\Omega_3 &\coloneqq \{\xi\in\Omega:\norm{\xi-(0.7,0.7)}_2<0.1\}, &\Omega_4 &\coloneqq \{\xi\in\Omega:\norm{\xi-(0.3,0.7)}_2<0.1\},\\
    	\Omega_0 &\coloneqq \Omega\setminus(\Omega_1\cup\Omega_2\cup\Omega_3\cup\Omega_4)
    \end{align*}
    where the subdomains~$\Omega_1,\dots,\Omega_4$ represent the ``cookies''. Further, the boundary~$\Gamma=\partial\Omega$ is divided into the parts
    \begin{align*}
    	\Gamma_{\mathrm{in}}\coloneqq\{0\}\times[0,1],\qquad\Gamma_D\coloneqq\{1\}\times[0,1],\qquad\text{and}\qquad\Gamma_N\coloneqq(0,1)\times\{0,1\}.
    \end{align*}
    The partitioning of the domain~$\Omega$ and its boundary~$\Gamma$ are visualized in the left part of~\Cref{fig:cookie-domain}.
    \par
    We define a heterogeneous and parameter-dependent heat conductivity~$\sigma\colon[0,T]\times\Omega\times\params\to\R_+$ which is, for~$\xi\in\Omega$, $t\in[0,T]$ and~$\mu\in\params\coloneqq[10^{-1},10^2]^2$, given as
    \begin{align*}
    	\sigma(t,\xi;\mu) \coloneqq \begin{cases}14\cdot (t-0.25)^2+0.125,&\text{for }\xi\in\Omega_0,\\\mu_1,&\text{for }\xi\in\Omega_1\cup\Omega_3,\\\mu_2,&\text{for }\xi\in\Omega_2\cup\Omega_4.\end{cases}
    \end{align*}
    This choice of conductivity describes a problem in which the four circular subdomains possess parameter-dependent material characteristics (i.e.~the cookies differ in the physical properties of their dough). Here, the subdomains~$\Omega_1$ and~$\Omega_3$ as well as~$\Omega_2$ and~$\Omega_4$ share the same conductivity. The time-dependency enters the problem via a background conductivity in~$\Omega_0$ that changes according to a function which is quadratic in time. This quadratic profile~$q\colon[0,T]\to\R$ given as~$q(t)\coloneqq 14\cdot (t-0.25)^2+0.125$ is constructed such that~$\min_{t\in[0,T]}q(t)=0.125$, $\argmin_{t\in[0,T]}q(t)=0.25$ and~$q(0)=1$. The input~$u\colon[0,T]\to\R$ acts on the inflow boundary~$\Gamma_{\mathrm{in}}$ as a Neumann boundary condition. We hence obtain a parametrized parabolic PDE describing the evolution of the temperature~$\theta\colon[0,T]\times\Omega\times\params\to\R$ over time as
    \begin{alignat*}{2}
    	\partial_t\theta(t,\xi;\mu)-\nabla\cdot(\sigma(t,\xi;\mu)\nabla\theta(t,\xi;\mu))&=0\qquad &t\in(0,T),\ &\xi\in\Omega,\\
    	\sigma(t,\xi;\mu)\nabla\theta(t,\xi;\mu)\cdot\normal(\xi)&=u(t),\qquad&t\in(0,T),\ &\xi\in\Gamma_{\mathrm{in}},\\
    	\sigma(t,\xi;\mu)\nabla\theta(t,\xi;\mu)\cdot\normal(\xi)&=0,&t\in(0,T),\ &\xi\in\Gamma_N,\\
    	\theta(t,\xi;\mu)&=0,&t\in(0,T),\ &\xi\in\Gamma_D,\\
    	\theta(0,\xi;\mu)&=0,&&\xi\in\Omega.
    \end{alignat*}
    Here,~$\normal\colon\Gamma\to\mathcal{S}^{1}$ denotes the unit outer normal on~$\Gamma$ (where~$\mathcal{S}^1\subset\R^2$ is the unit circle). Moreover, we define the components of the output~$y\colon[0,T]\times\params\to\R^4$ as the average over the respective subdomain, that is
    \begin{align*}
    	y_i(t;\mu)\coloneqq\frac{1}{\lvert\Omega_i\rvert}\int\limits_{\Omega_i}\theta(t,\xi;\mu)\d{\xi}\qquad\text{for }i=1,\dots,4.
    \end{align*}
    \par
    The system matrices are obtained by Galerkin projection of the weak formulation of the~PDE onto a space of first order finite elements on a triangulation of the domain~$\Omega$. The circular boundaries of the subdomains~$\Omega_1,\dots,\Omega_4$ are approximated by faces of elements from the triangulation, i.e.~the circles are approximated piecewise by straight lines. For every subdomain~$\Omega_i$ we obtain a corresponding matrix~$A_i\in\R^{n\times n}$ for~$i=1,\dots,4$. Furthermore, the parameter-independent background conductivity is represented in a matrix~$A_0\in\R^{n\times n}$. We consequently have
    \begin{align*}
    	\A{\mu;t}=\big(14\cdot (t-0.25)^2+0.125\big)A_0+\mu_1(A_1+A_3)+\mu_2(A_2+A_4).
    \end{align*}
    The action of the control at the inflow boundary is described by the (parameter- and time-independent) matrix~$B\in\R^{n\times m}$ with~$m=1$. We additionally obtain an output matrix~$C\in\R^{4\times n}$ that relates the discretized state to the output. It is important to also take the (symmetric and positive-definite) mass matrix~$\E\in\R^{n\times n}$ (including handling of boundary values) into account. The semi-discretized problem is then for~$t\in[0,T]$ given as
    \begin{align*}
    	\E\frac{d}{dt}x_\mu(t) &= \A{\mu;t}x_\mu(t)+Bu_\mu(t),\\
    	y(t;\mu) &= Cx_\mu(t).
    \end{align*}
    The time discretization scheme is rewritten in such a way that as few linear systems of equations involving the mass matrix as possible need to be solved. For the time discretization we employ the implicit Euler scheme with~$\nt=50$ time steps.
    \par
    We choose~$\Mop=C^\top C$ such that a deviation in the output is penalized and~$x_\mu^T=0.25\cdot\textbf{1}_n$ as target state, where~$\textbf{1}_n\in\R^n$ denotes the~$n$-dimensional vector consisting of ones in each entry. In this setting, the goal is to steer the system to a state in which every component of the output is close to a value of~$0.25$ at final time~$T=1$ since~$Cx_\mu^T=0.25\cdot\textbf{1}_4\in\R^4$. Further, we set~$\Gmat\in\R^{n\times n}$ to be the mass matrix of the finite element discretization (without boundary value treatment), which relates to the~$L^2$-norm in the finite element space. The space~$\X$ hence corresponds to the space~$\R^n$ equipped with the inner product induced by the matrix~$\Gmat$, i.e.~we define~$\normX{x}\coloneqq\sqrt{x^\top\Gmat x}$ for~$x\in\R^n=\X$. Moreover, we set~$\Rop=0.02$ as time-independent weighting of the control energy. The reduced operators are constructed to be biorthogonal with respect to the mass matrix, which is ensured by setting~$\Wpr=\Vpr$ as well as~$\Wad=\Vad$ and using the inner product induced by the mass matrix within the~HaPOD method for orthonormalization.
    \par
    Since the system matrix~$\A{\mu;t}$ is negative-definite for all parameters~$\mu\in\params$ at all times~$t\in[0,T]$ with a negative upper bound on the eigenvalues (recall that the background diffusivity is at least~$0.125$) that is independent of parameter and time (i.e.~the matrix~$\A{\mu;t}$ is uniformly negative definite with respect to time and parameter), the resulting system is stable. Therefore, the choice~$C_1(\mu)=1$ fulfills~\cref{equ:constant-c1} for all~$\mu\in\params$, see also~\cite[Section~4]{haasdonk2011efficient}. The constant~$C_2(\mu)\equiv C_2\geq 0$ can be computed offline as
    \begin{align*}
    	C_2=\norm{\Gmat^{1/2}\E^{-1}BR^{-1}B^\top\Gmat^{-1/2}}_2=\norm{\E^{-1}BR^{-1}B^\top}_{\Lin{\X}{\dualX}}
    \end{align*}
    since the matrices~$B$ and~$R$ are independent of parameter and time. For the constant~$\const$, we observe in the numerical experiments that~$\const=1$ is sufficient to guarantee a reliable error estimator. To give a precise estimate for~$\const$ is usually impossible in practice without additional assumptions on~$\E$ and~$\Mop$. We therefore resort to the choice given above which is motivated from our observations in the numerical experiment. We also emphasize at this point that the given value for~$\const$ was observed from the reduced model for the final time adjoint states without reduction of the system dynamics, i.e.~the reduced model described in~\Cref{sec:reduced-model-optimal-adjoint-states}. Since the same value of the constant also leads to reliable error estimates in the case of an additional system reduction, we believe that the given choice for~$\const$ is appropriate in this numerical example.
    \par
    \Cref{fig:cookie-domain} shows the domain~$\Omega$ on the left hand side and the solution state at a fixed time associated to the optimal control for a certain choice of the parameter on the right-hand side.
    \par
    \begin{figure}[htb]
    	\begin{minipage}{.47\textwidth}
    		\begin{tikzpicture}
    			\pic[scale=6]{cookiesdomain};
    		\end{tikzpicture}
    	\end{minipage}%
    	\hfill
    	\begin{minipage}{.53\textwidth}
    		\begin{tikzpicture}
    			\node at (3,3) {\includegraphics[width=6cm]{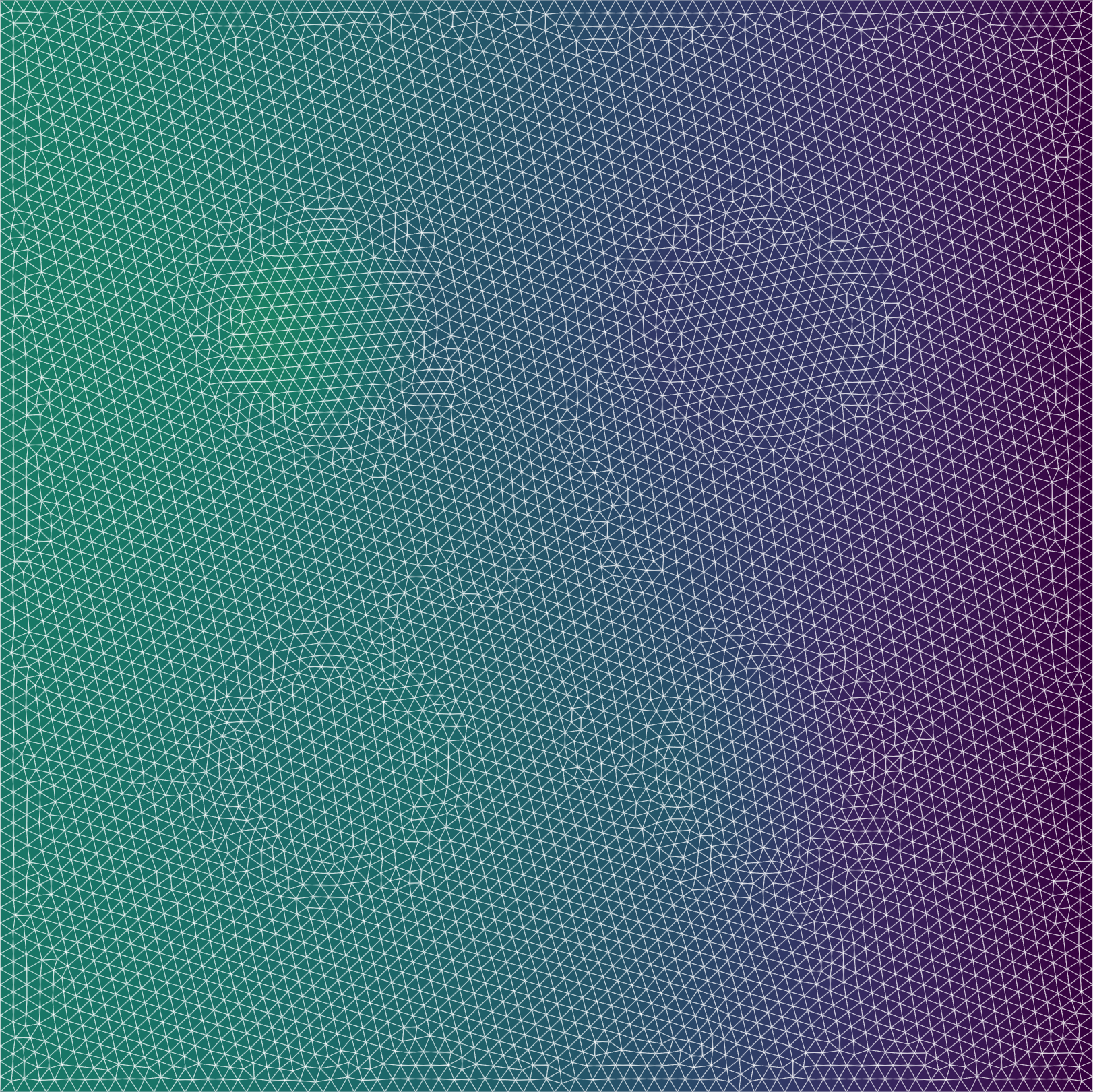}};
    			\pic[scale=6]{cookiesdomain};
    			\begin{axis}[
    				name=colorbar,
    				hide axis,
    				scale only axis,
    				xshift=6.25cm,
    				height=6cm,
    				width=0.5cm,
    				colorbar,
    				xmin=0, xmax=1, ymin=0, ymax=1,
    				colormap/viridis,
    				point meta min=0,
    				point meta max=0.52020485001641,
    				colorbar style={
    					scaled y ticks = false,
    					yticklabel style={/pgf/number format/.cd,precision=5},
    				},
    				clip=false,
    				]
    				\addplot [draw=none] coordinates {(0,0) (1,1)};
    			\end{axis}
    		\end{tikzpicture}
    	\end{minipage}
    	\caption{Visualization of the domain~$\Omega$ for the cookies baking example in the left plot, see~\cite[Figure~1]{rave2021nonstationary}. The solution state at time~$t=1$ on the finite element grid corresponding to the optimal control~$u_\mu^*$ for the parameter~$\mu=(100,0.1)$ is additionally shown in the plot on the right.}
    	\label{fig:cookie-domain}
    \end{figure}
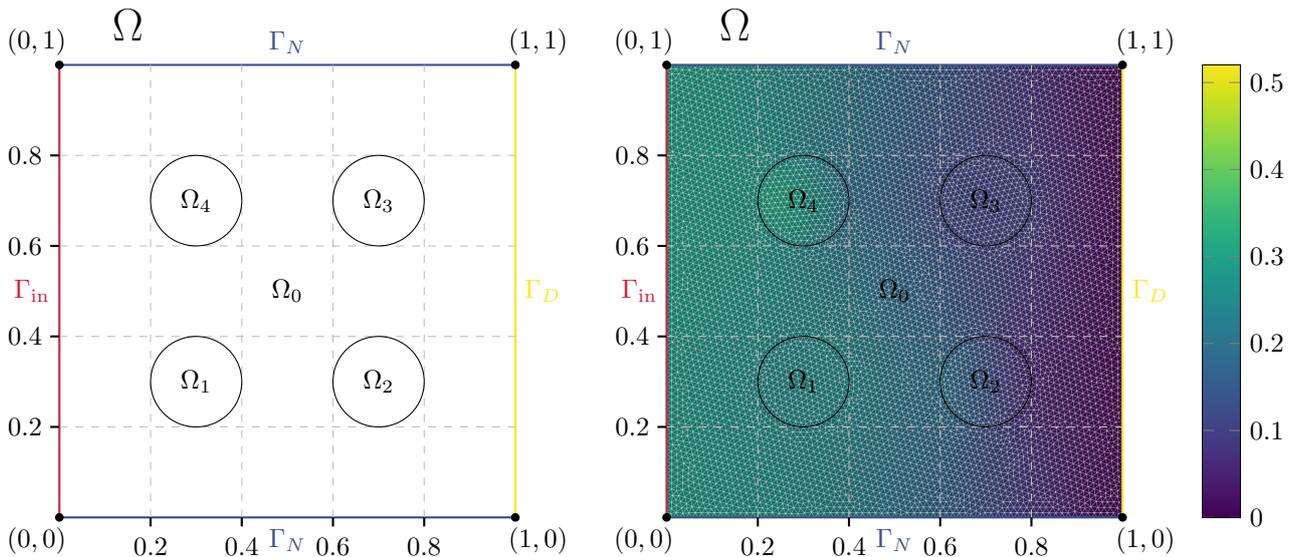
    
    We further present in~\Cref{fig:optimal-control-and-outputs} the optimal control over time with corresponding solution states at several different time steps for the parameter~$\mu=(100,0.1)$. Additionally, the outputs over time are depicted separately. We observe that for this choice of the parameter, the first and the last component of the output (associated to the two circular subdomains~$\Omega_1$ and~$\Omega_4$ in the left half of the domain~$\Omega$) are controlled to be close to the target of~$0.25$ at final time. The output for the other two subdomains~$\Omega_2$ and~$\Omega_3$ only reach a value of about~$0.1$ at the end of the time interval. The low conductivity of~$\mu_2=0.1$ in~$\Omega_2$ and~$\Omega_4$ makes it particularly costly in terms of control energy to heat up the subdomains~$\Omega_2$ and~$\Omega_4$ to the desired value of~$0.25$. Since a large control energy is penalized in the optimal control problem, the second subdomain~$\Omega_2$ thus remains relatively cold in favor of a smaller control energy. This behavior can also be seen in the solution plots over time given in the middle row of~\Cref{fig:optimal-control-and-outputs}. The third subdomain~$\Omega_3$ is affected by the low conductivity in~$\Omega_4$ and the larger distance to the heating boundary~$\Gamma_{\mathrm{in}}$ and thus cannot be heated up without consuming additional control energy as well. Although having only a low conductivity, due to its proximity to~$\Gamma_{\mathrm{in}}$ the subdomain~$\Omega_4$ reaches a value close to the target concerning the respective output.
    \par
    \begin{figure}[htb]
    	\centering
    	\pgfplotstableread[header=false]{data/optimal_control.txt}\mytablecont
    	\pgfplotstablegetelem{0}{[index]0}\of\mytablecont
    	\edef\firstvaluecont{\pgfplotsretval}
    	\pgfplotstablegetelem{10}{[index]0}\of\mytablecont
    	\edef\secondvaluecont{\pgfplotsretval}
    	\pgfplotstablegetelem{20}{[index]0}\of\mytablecont
    	\edef\thirdvaluecont{\pgfplotsretval}
    	\pgfplotstablegetelem{30}{[index]0}\of\mytablecont
    	\edef\fourthvaluecont{\pgfplotsretval}
    	\pgfplotstablegetelem{40}{[index]0}\of\mytablecont
    	\edef\fifthvaluecont{\pgfplotsretval}
    	\pgfplotstablegetelem{48}{[index]0}\of\mytablecont
    	\edef\intermediatevaluecont{\pgfplotsretval}
    	\pgfplotstablegetelem{50}{[index]0}\of\mytablecont
    	\edef\sixthvaluecont{\pgfplotsretval}
    	\pgfplotstableread[header=false]{data/output.txt}\mytableout
    	\pgfplotstablegetelem{0}{[index]0}\of\mytableout
    	\edef\firstvalueout{\pgfplotsretval}
    	\pgfplotstablegetelem{10}{[index]0}\of\mytableout
    	\edef\secondvalueout{\pgfplotsretval}
    	\pgfplotstablegetelem{20}{[index]0}\of\mytableout
    	\edef\thirdvalueout{\pgfplotsretval}
    	\pgfplotstablegetelem{30}{[index]0}\of\mytableout
    	\edef\fourthvalueout{\pgfplotsretval}
    	\pgfplotstablegetelem{40}{[index]0}\of\mytableout
    	\edef\fifthvalueout{\pgfplotsretval}
    	\pgfplotstablegetelem{48}{[index]0}\of\mytableout
    	\edef\intermediatevalueout{\pgfplotsretval}
    	\pgfplotstablegetelem{50}{[index]3}\of\mytableout
    	\edef\sixthvalueout{\pgfplotsretval}
    	\begin{tikzpicture}[remember picture]
    		\begin{axis}[
    			name=control,
    			anchor=south,
    			scale only axis,
    			width=14cm,
    			height=3cm,
    			DefaultStyle,
    			xmin=0,
    			xmax=50,
    			ymin=0,
    			ymax=5,
    			legend style={xshift=-23pt, yshift=-3pt},
    			xlabel={Time~$t$},
    			xtick pos=top,
    			ytick={0,1,2,3,4,5},
    			xtick={0,5,...,50},
    			xticklabels={0,0.1,0.2,0.3,0.4,0.5,0.6,0.7,0.8,0.9,1},
    			ylabel={Control},
    			ylabel shift=-4pt,
    			]
    			\addplot[thick, color=viridisViolet] table[mark=none, x expr=\coordindex, y index=0] {data/optimal_control.txt};
    			\addlegendentry{$u_\mu^*$: optimal control}
    			\node (val-0) at (0,\firstvaluecont) {};
    			\node (val-10) at (10,\secondvaluecont) {};
    			\node (val-20) at (20,\thirdvaluecont) {};
    			\node (val-30) at (30,\fourthvaluecont) {};
    			\node (val-40) at (40,\fifthvaluecont) {};
    			\node (val-48) at (48,\intermediatevaluecont) {};
    			\node (val-50) at (50,\sixthvaluecont) {};
    		\end{axis}
    		\begin{axis}[
    			name=plots,
    			anchor=north,
    			hide axis,
    			at=(control.south),
    			scale only axis,
    			width=14cm,
    			height=3.5cm,
    			DefaultStyle,
    			yshift=1cm,
    			xmin=0,
    			xmax=50,
    			ymin=-3,
    			ymax=2.,
    			clip=false,
    			]
    			\addplot graphics [xmin=-2.5, xmax=2.5, ymin=-2.5, ymax=-0.357] {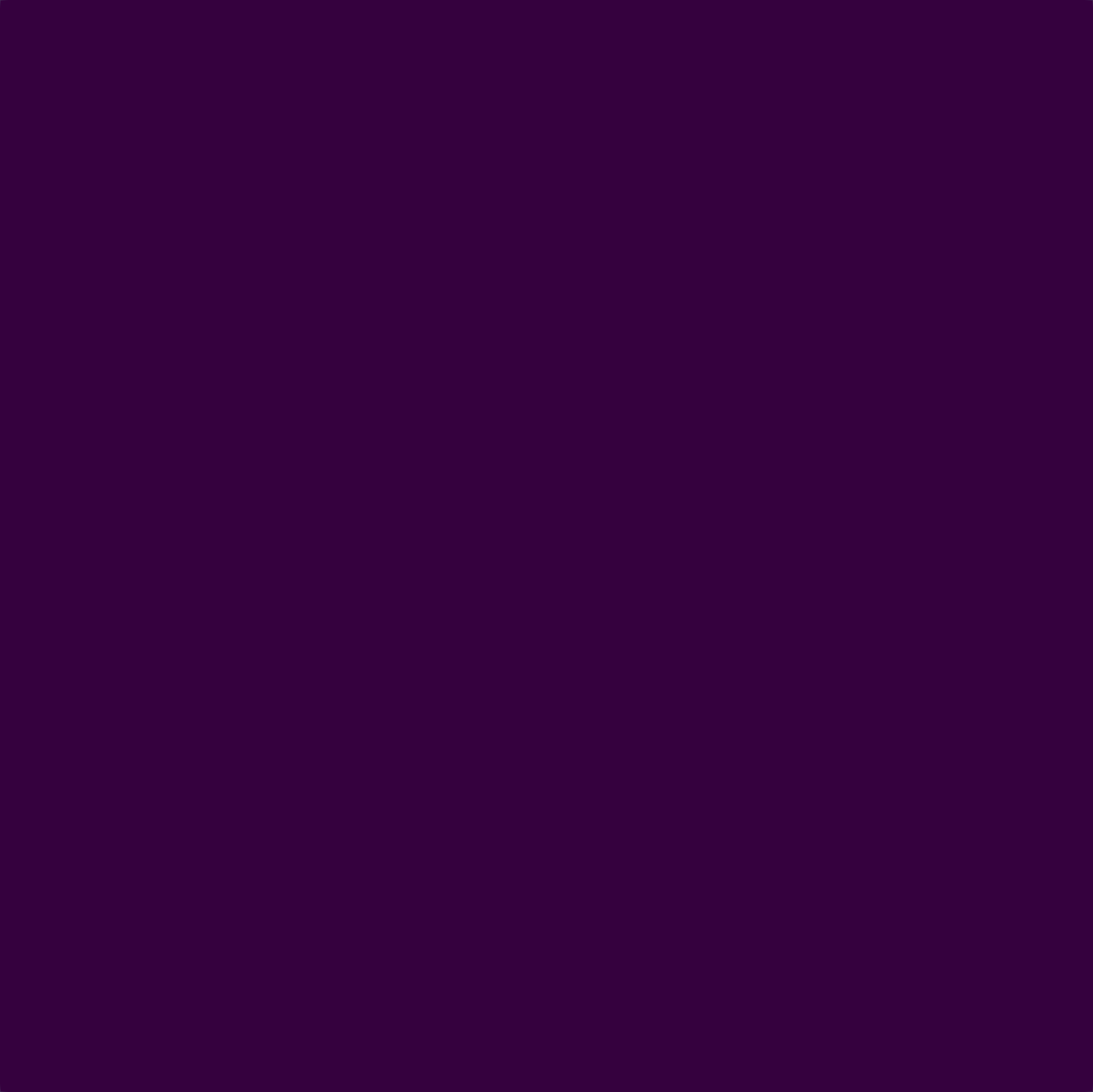};
    			\addplot graphics [xmin=7.5, xmax=12.5, ymin=-2.5, ymax=-0.357] {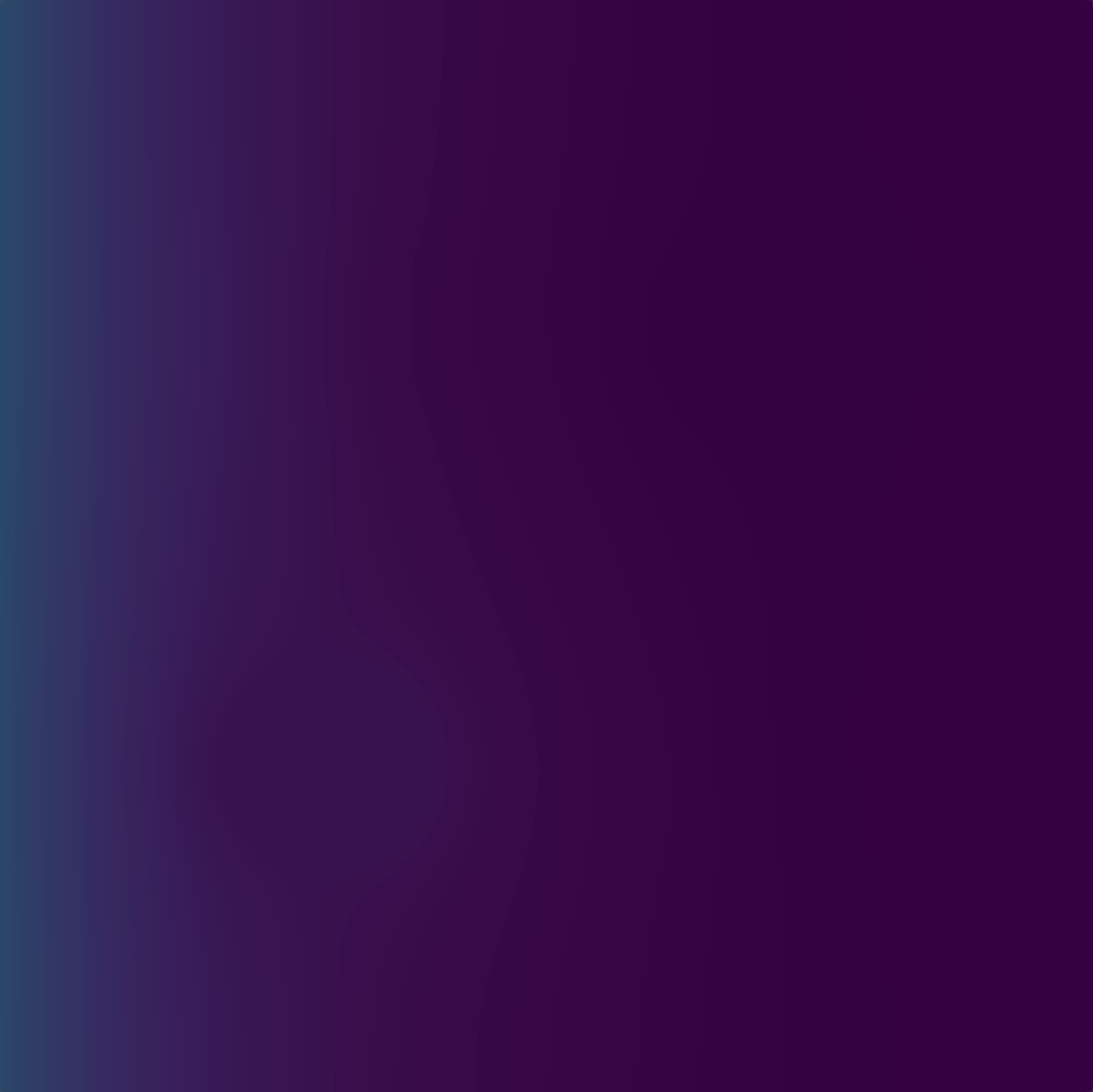};
    			\addplot graphics [xmin=17.5, xmax=22.5, ymin=-2.5, ymax=-0.357] {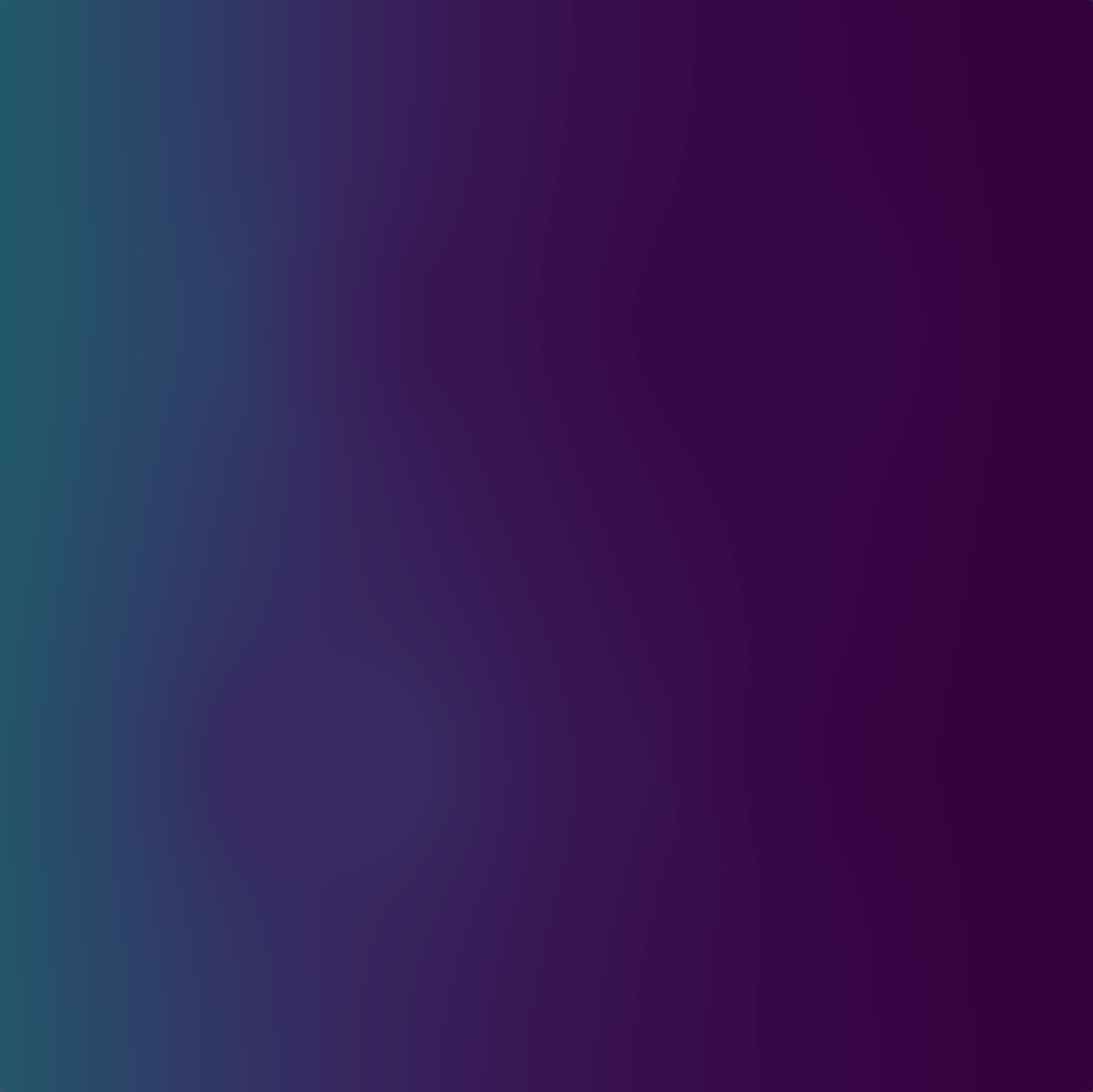};
    			\addplot graphics [xmin=27.5, xmax=32.5, ymin=-2.5, ymax=-0.357] {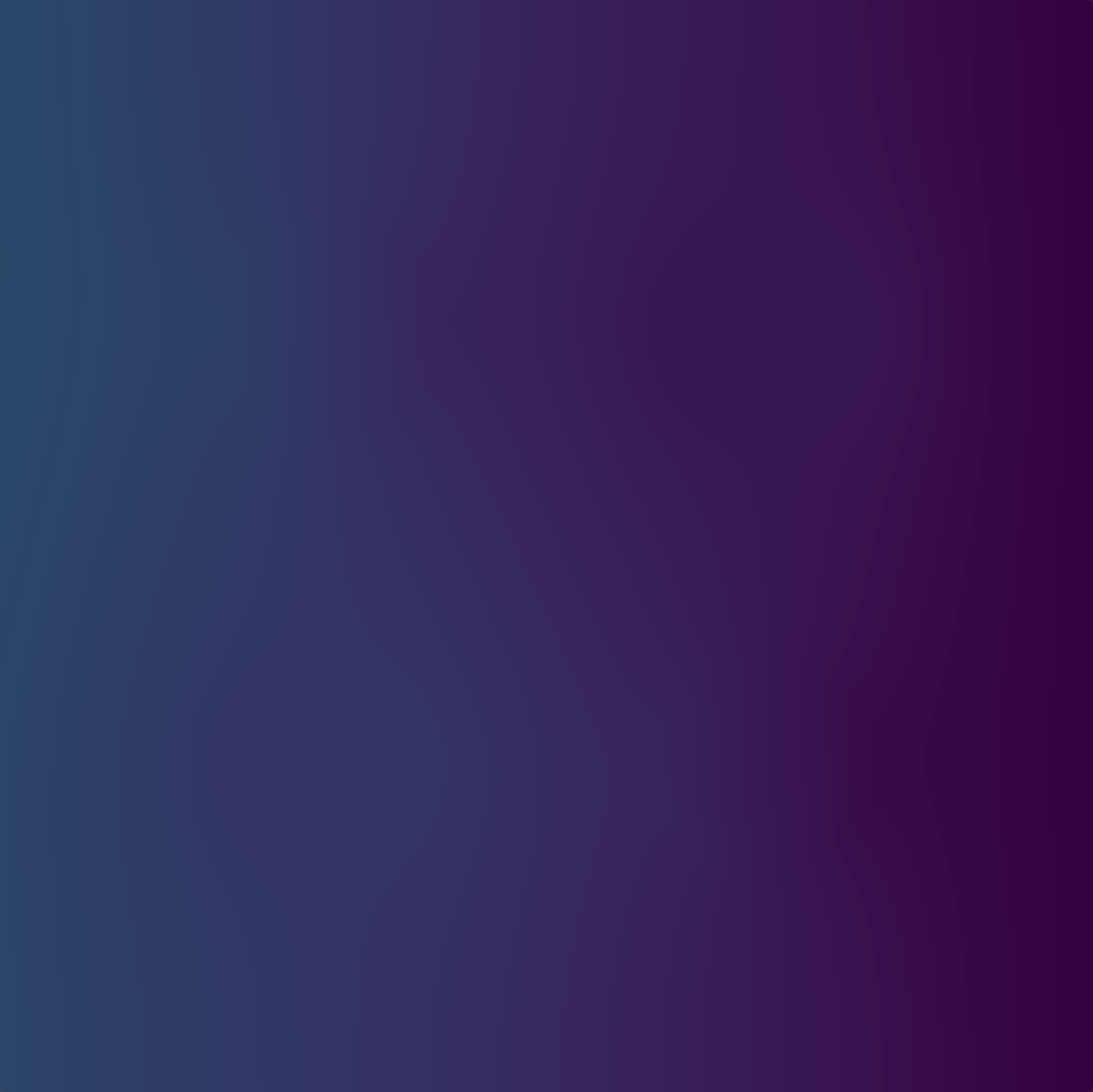};
    			\addplot graphics [xmin=35, xmax=40, ymin=-2.5, ymax=-0.357] {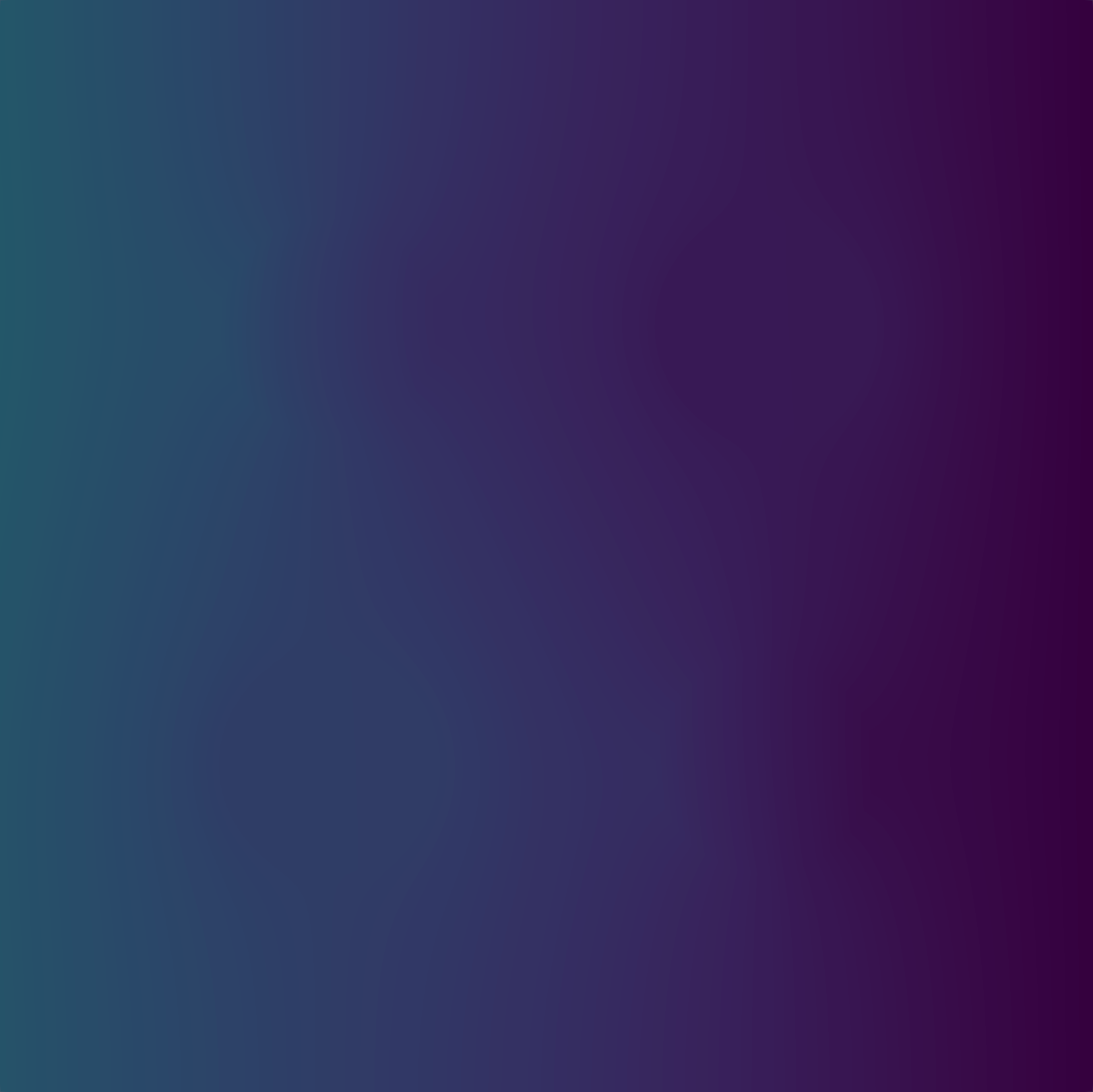};
    			\addplot graphics [xmin=41.25, xmax=46.25, ymin=-2.5, ymax=-0.357] {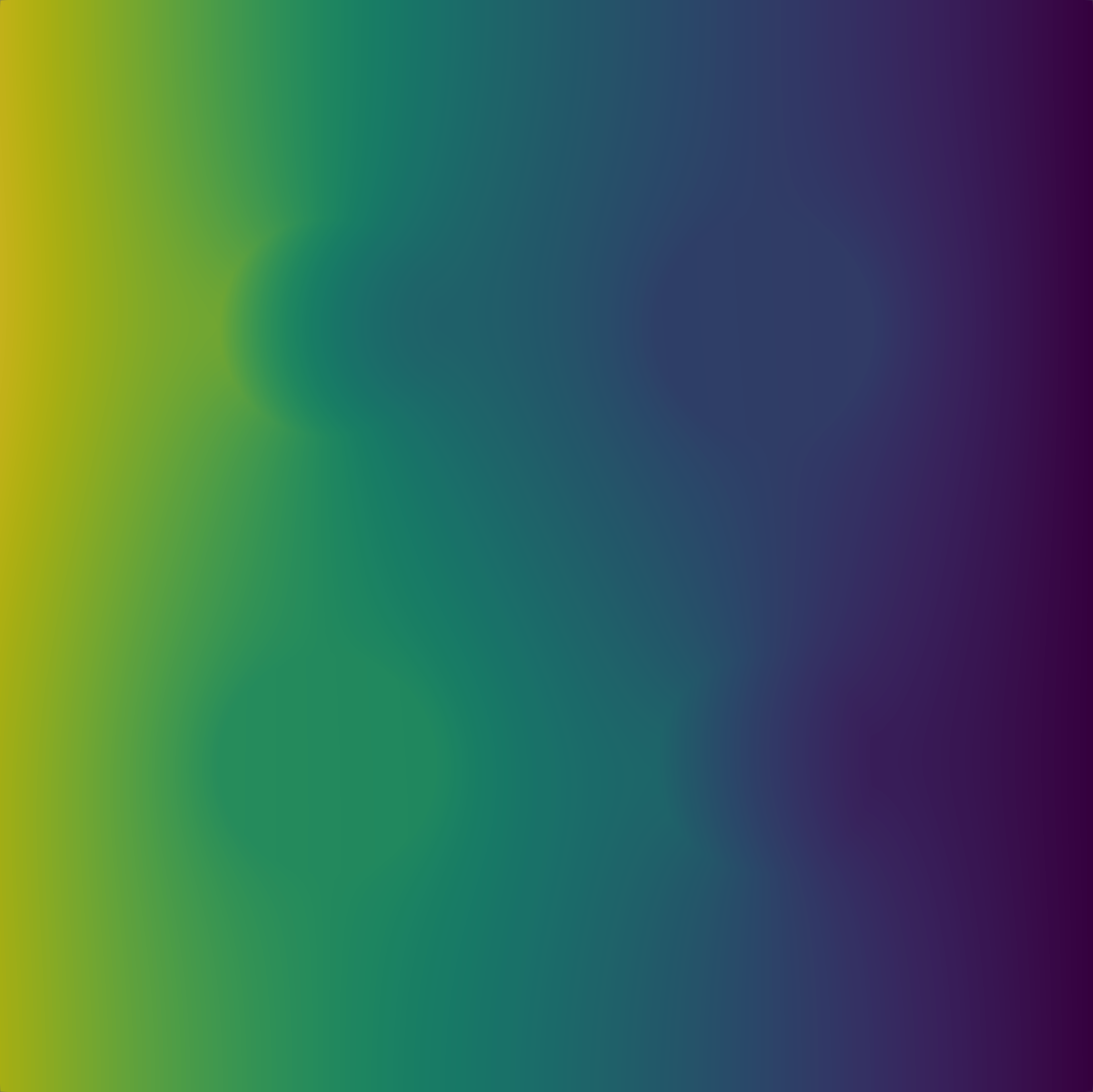};
    			\addplot graphics [xmin=47.5, xmax=52.5, ymin=-2.5, ymax=-0.357] {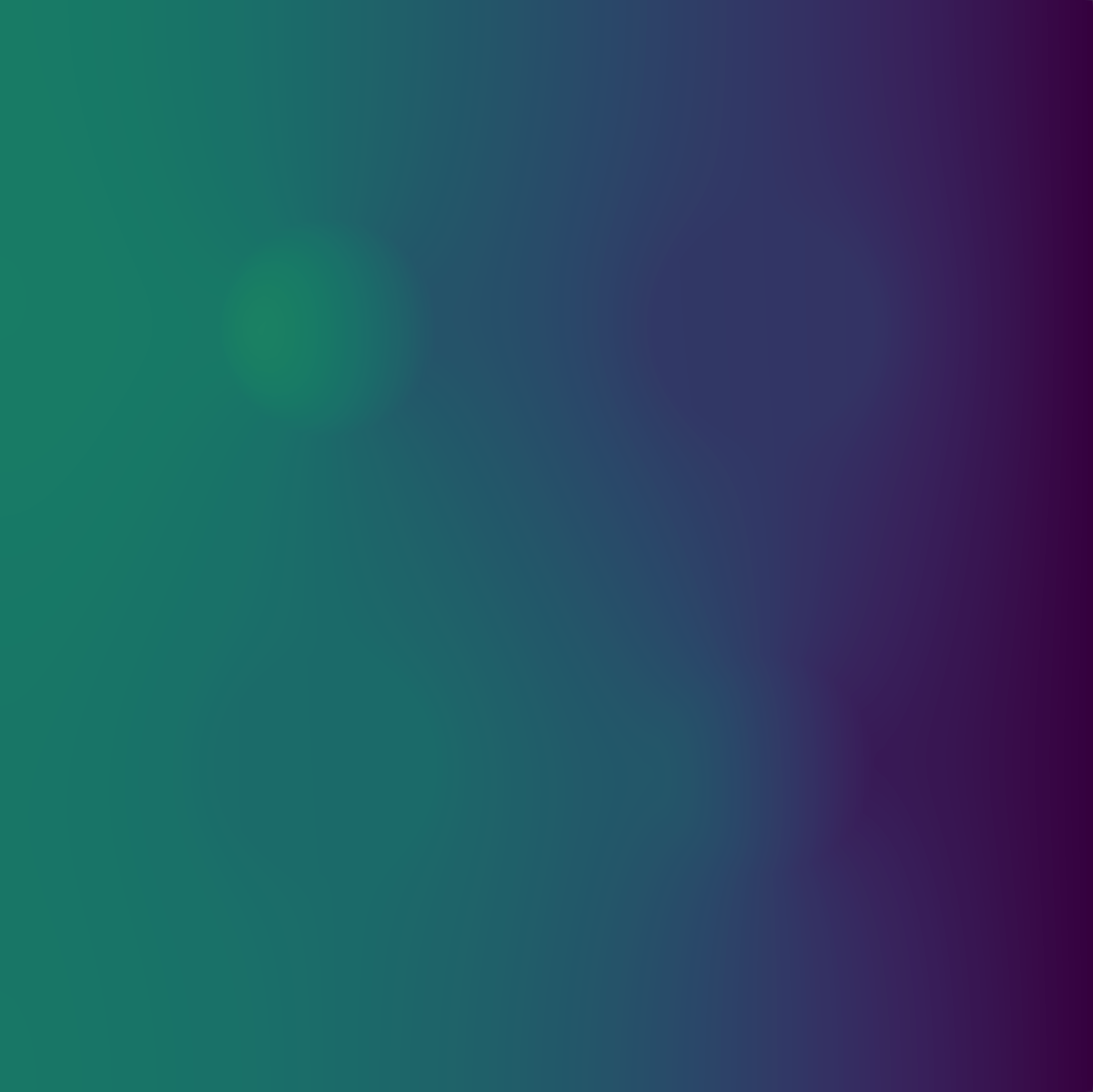};
    			\node (sol-0-top) at (axis cs:0,-0.357) {};
    			\node (sol-0-bottom) at (axis cs:0,-2.5) {};
    			\node (sol-10-top) at (axis cs:10,-0.357) {};
    			\node (sol-10-bottom) at (axis cs:10,-2.5) {};
    			\node (sol-20-top) at (axis cs:20,-0.357) {};
    			\node (sol-20-bottom) at (axis cs:20,-2.5) {};
    			\node (sol-30-top) at (axis cs:30,-0.357) {};
    			\node (sol-30-bottom) at (axis cs:30,-2.5) {};
    			\node (sol-40-top) at (axis cs:37.4,-0.357) {};
    			\node (sol-40-bottom) at (axis cs:37.5,-2.5) {};
    			\node (sol-48-top) at (axis cs:45,-0.357) {};
    			\node (sol-48-bottom) at (axis cs:45,-2.5) {};
    			\node (sol-50-top) at (axis cs:50,-0.357) {};
    			\node (sol-50-bottom) at (axis cs:50,-2.5) {};
    		\end{axis}
    		\begin{axis}[
    			name=colorbar,
    			anchor=north,
    			at=(plots.south),
    			hide axis,
    			scale only axis,
    			yshift=0.75cm,
    			height=0.5cm,
    			width=14cm,
    			colorbar,
    			xmin=0, xmax=1, ymin=0, ymax=1,
    			colormap/viridis,
    			point meta min=0,
    			point meta max=0.52020485001641,
    			colorbar horizontal,
    			colorbar style={
    				scaled y ticks = false,
    				yticklabel style={/pgf/number format/.cd,precision=5},
    				xtick={0,0.05,0.1,0.15,0.2,0.25,0.3,0.35,0.4,0.45,0.5},
    				xticklabels={0,0.05,0.1,0.15,0.2,0.25,0.3,0.35,0.4,0.45,0.5},
    			},
    			clip=false,
    			]
    			\addplot [draw=none] coordinates {(0,0) (1,1)};
    		\end{axis}
    		\begin{axis}[
    			name=outputs,
    			anchor=north,
    			at=(colorbar.south),
    			yshift=-1.5cm,
    			scale only axis,
    			width=14cm,
    			height=3cm,
    			DefaultStyle,
    			xmin=0,
    			xmax=50,
    			clip=false,
    			legend style={yshift=-3.55cm, xshift=1.5cm,},
    			legend columns=4,
    			xtick={0,5,...,50},
    			xticklabels={0,0.1,0.2,0.3,0.4,0.5,0.6,0.7,0.8,0.9,1},
    			xlabel={Time~$t$},
    			ylabel={Output},
    			ylabel shift=-4pt,
    			ymin=-0.02,
    			ymax=0.4,
    			ytick={0,0.1,...,0.4},
    			]
    			\node[label=right:{Target}, draw, fill, circle, inner sep=1.25pt] (target) at (50,0.25) {};
    			\addplot[thick, color=viridisYellow] table[mark=none, x expr=\coordindex, y index=0] {data/output.txt};
    			\addlegendentry{$y_1(t;\mu)$}
    			\addplot[thick, color=viridisBlue] table[mark=none, x expr=\coordindex, y index=1] {data/output.txt};
    			\addlegendentry{$y_2(t;\mu)$}
    			\addplot[thick, color=viridisGreen] table[mark=none, x expr=\coordindex, y index=2] {data/output.txt};
    			\addlegendentry{$y_3(t;\mu)$}
    			\addplot[thick, color=matchingRed] table[mark=none, x expr=\coordindex, y index=3] {data/output.txt};
    			\addlegendentry{$y_4(t;\mu)$}
    			\node (node-0) at (0,\firstvalueout) {};
    			\node (node-10) at (10,\secondvalueout) {};
    			\node (node-20) at (20,\thirdvalueout) {};
    			\node (node-30) at (30,\fourthvalueout) {};
    			\node (node-40) at (40,\fifthvalueout) {};
    			\node (node-48) at (48,\intermediatevalueout) {};
    			\node (node-50) at (50,\sixthvalueout) {};
    		\end{axis}
    		\draw[->, >=stealth, thick] (sol-0-top) to[out=75, in=-75] (val-0);
    		\draw[->, >=stealth, thick] (sol-10-top) to[out=75, in=-75] (val-10);
    		\draw[->, >=stealth, thick] (sol-20-top) to[out=75, in=-75] (val-20);
    		\draw[->, >=stealth, thick] (sol-30-top) to[out=75, in=-75] (val-30);
    		\draw[->, >=stealth, thick] (sol-40-top) to[out=75, in=-105] (val-40);
    		\draw[->, >=stealth, thick] (sol-48-top) to[out=75, in=-90] (val-48);
    		\draw[->, >=stealth, thick] (sol-50-top) to[out=75, in=-75] (val-50);
    		\draw[->, >=stealth, thick] (sol-0-bottom) to[out=-75, in=75] (node-0);
    		\draw[->, >=stealth, thick] (sol-10-bottom) to[out=-75, in=75] (node-10);
    		\draw[->, >=stealth, thick] (sol-20-bottom) to[out=-75, in=75] (node-20);
    		\draw[->, >=stealth, thick] (sol-30-bottom) to[out=-75, in=75] (node-30);
    		\draw[->, >=stealth, thick] (sol-40-bottom) to[out=-120, in=105] (node-40);
    		\draw[->, >=stealth, thick] (sol-48-bottom) to[out=-75, in=105] (node-48);
    		\draw[->, >=stealth, thick] (sol-50-bottom) to[out=-75, in=75] (node-50);
    	\end{tikzpicture}
    	\caption{Optimal control (top) and output values (bottom) for the parameter~$\mu=(100,0.1)$ in the cookie baking problem. Further, the corresponding solution states at time~$t\in\{0,0.2,0.4,0.6,0.8,0.96,1\}$ are shown in the middle row of the figure.}
    	\label{fig:optimal-control-and-outputs}
    \end{figure}
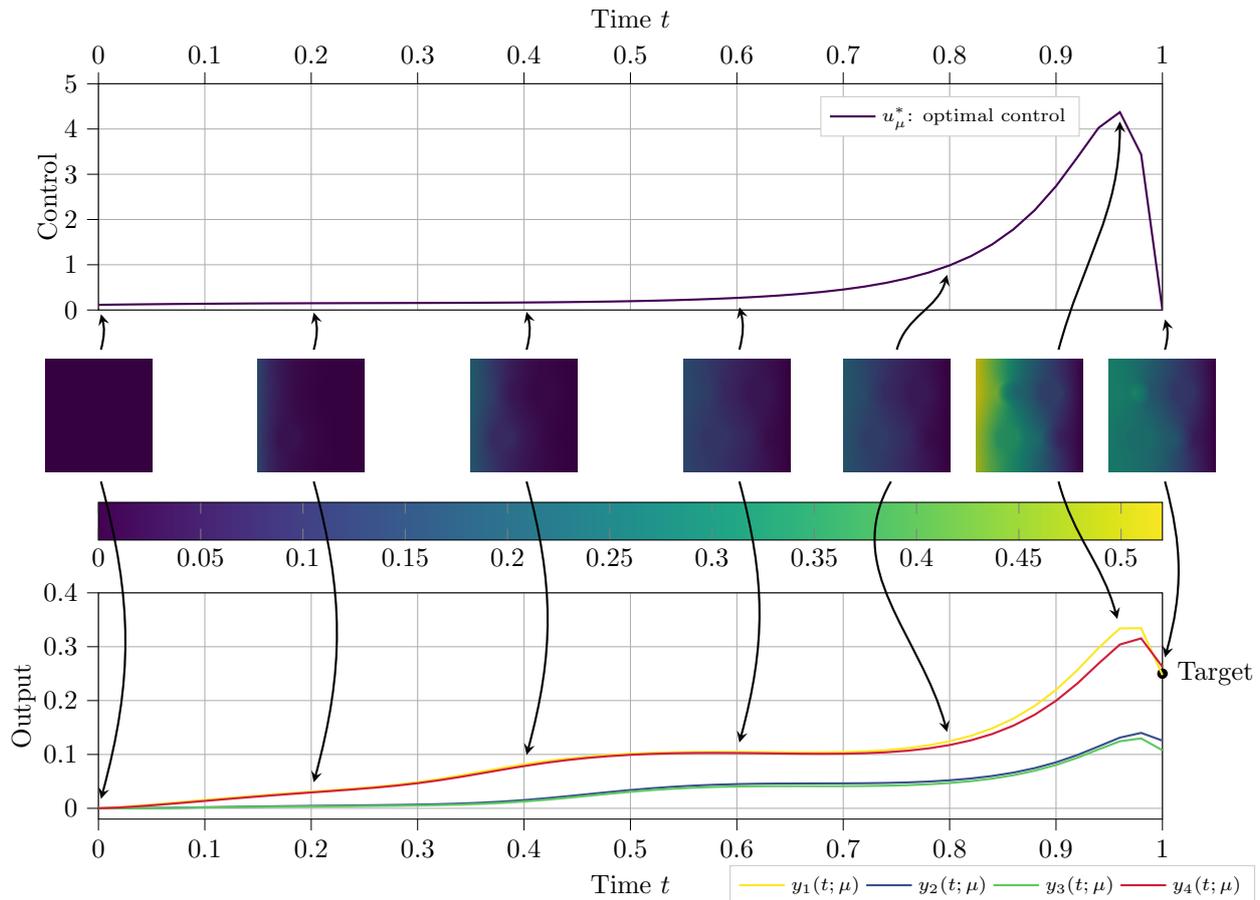
    
    The optimal final time adjoint states for the parameters~$\mu=(100, 0.1)$ and~$\mu=(0.5,50)$ are shown in~\Cref{fig:optimal-final-time-adjoint}. We observe that the adjoint state is mostly zero on the background domain~$\Omega_0$ and only differs in and close to the cookies for the two parameters. This observation will also become important in the next section.
    \par
    \begin{figure}[htb]
    	\begin{minipage}{.47\textwidth}
    		\begin{tikzpicture}
    			\node at (3,3) {\includegraphics[width=6cm]{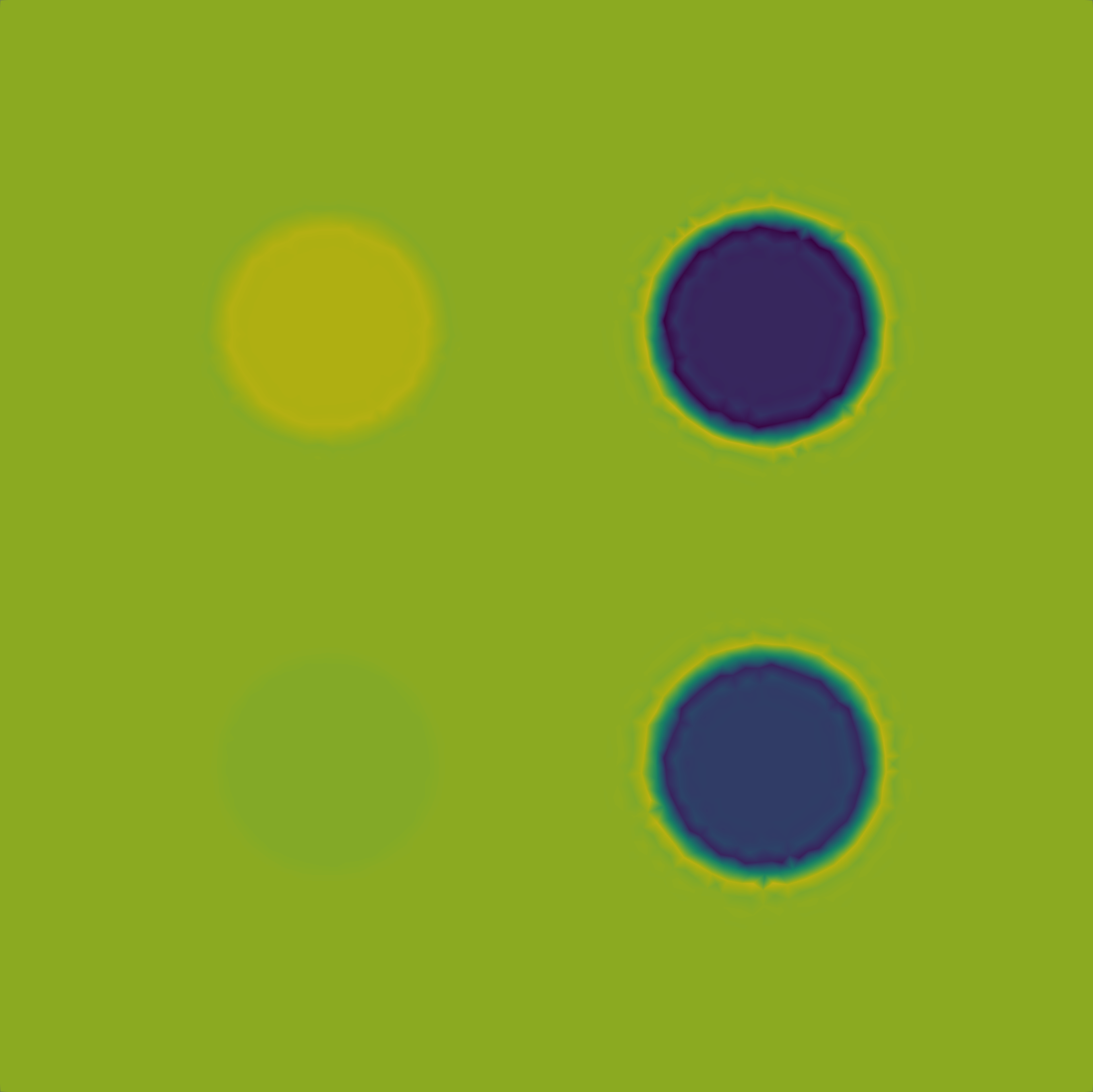}};
    			\pic[scale=6]{cookiesdomain};
    		\end{tikzpicture}
    	\end{minipage}%
    	\hfill
    	\begin{minipage}{.53\textwidth}
    		\begin{tikzpicture}
    			\node at (3,3) {\includegraphics[width=6cm]{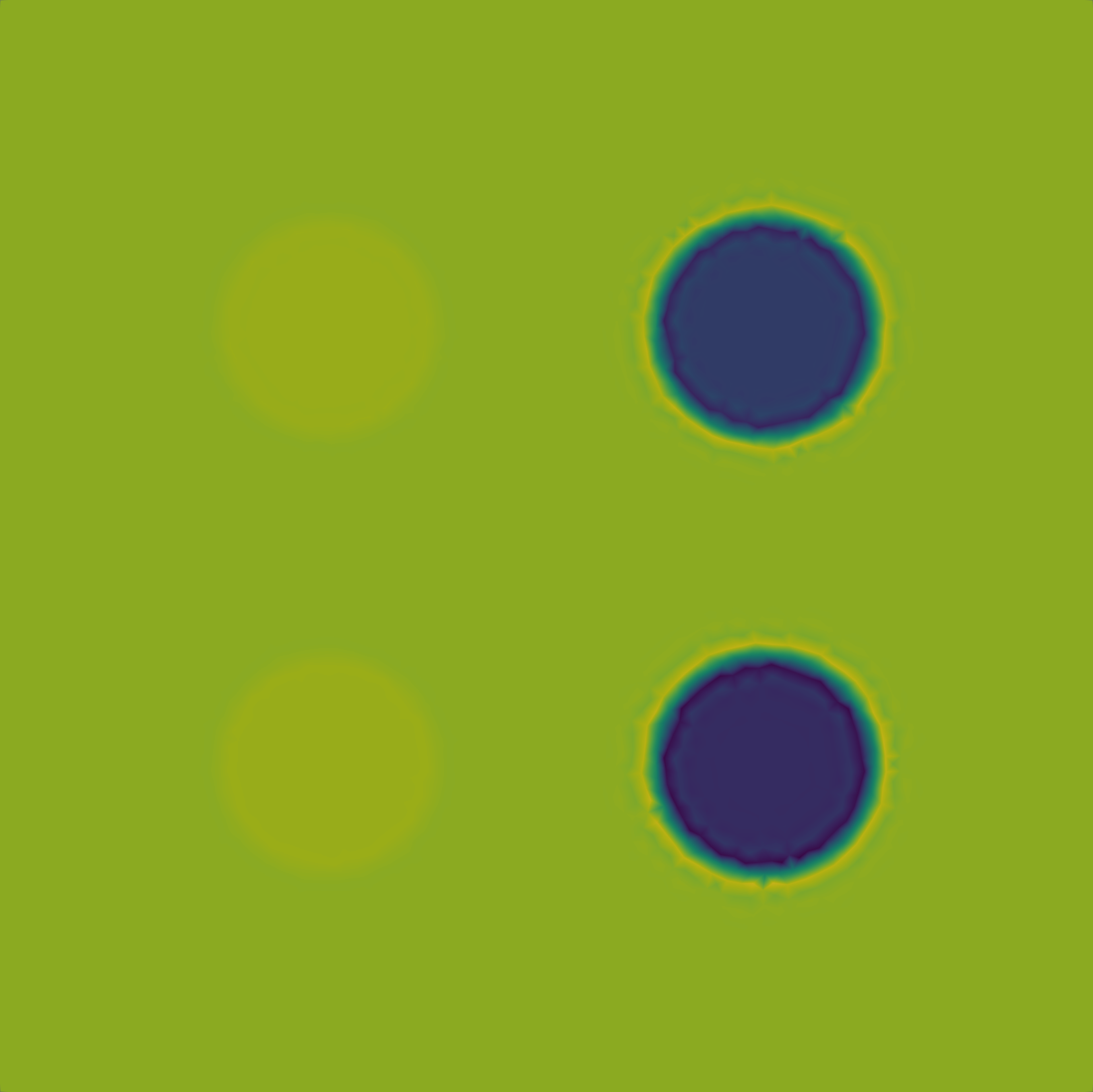}};
    			\pic[scale=6]{cookiesdomain};
    			\begin{axis}[
    				name=colorbar,
    				hide axis,
    				scale only axis,
    				xshift=6.25cm,
    				height=6cm,
    				width=0.5cm,
    				colorbar,
    				xmin=0, xmax=1, ymin=0, ymax=1,
    				colormap/viridis,
    				point meta min=-5.424152017032543,
    				point meta max=0.7126529779581272,
    				colorbar style={
    					scaled y ticks = false,
    					yticklabel style={/pgf/number format/.cd,precision=5},
    					ytick={-5,-4,-3,-2,-1,0},
    				},
    				clip=false,
    				]
    				\addplot [draw=none] coordinates {(0,0) (1,1)};
    			\end{axis}
    		\end{tikzpicture}
    	\end{minipage}
    	\caption{Optimal final time adjoint state~$\optFTA{\mu}$ for the parameter~$\mu=(100, 0.1)$ (left) and the parameter~$\mu=(0.5, 50)$ (right).}
    	\label{fig:optimal-final-time-adjoint}
    \end{figure}
    
    \Cref{tab:cookie-baking-full-model-parameters} gives an overview of different quantities related to the full order model for the cookie baking problem, such as the final time~$T$, the parameter domain~$\params$, the discretization sizes and the number of affine components in the decomposition of the system matrices and initial and target states. In addition, the parameters employed in the different reduction strategies from~\Cref{sec:strategies-computing-fully-reduced-model} are summarized in~\Cref{tab:cookie-baking-reduction-strategies-parameters}. We remark at this point that the training parameters are chosen on a~$\log$-uniform grid in the parameter set~$\params$ with the same number of samples in each direction.
    \par
    \begin{table}[htb]
    	\centering
    	\begin{minipage}{.49\textwidth}
    		\centering
    		\begin{tabular}{l c | c}
    			\toprule
    			\multicolumn{2}{l|}{Parameter} & Value \\ \midrule \midrule
    			\multicolumn{2}{l|}{Final time~$T$} & $1$ \\
    			\multicolumn{2}{l|}{Parameter domain~$\params$} & $[10^{-1},10^2]^2$ \\
    			\multicolumn{2}{l|}{State dimension~$n$} & $6714$ \\
    			\multicolumn{2}{l|}{Number of controls~$m$} & $1$ \\
    			\multicolumn{2}{l|}{Number of time steps~$\nt$} & $50$ \\
    			\midrule
    			\multirow{4}{3cm}{Number of affine components} & $Q_A$ & $3$ \\
    			& $Q_B$ & $1$ \\
    			& $Q_{x^0}$ & $1$ \\
    			& $Q_{x^T}$ & $1$ \\
    			\bottomrule
    		\end{tabular}
    		\caption{Parameters of the full order model in the cookie baking example.}
    		\label{tab:cookie-baking-full-model-parameters}
    	\end{minipage}%
    	\hfill
    	\begin{minipage}{.49\textwidth}
    		\centering
    		\begin{tabular}{l c | c}
    			\toprule
    			\multicolumn{2}{l|}{Parameter} & Value \\ \midrule \midrule
    			\multirow{4}{3cm}{Tolerances for reduction strategies} & $\epssys$ & $10^{-9}$ \\
    			& $\epsinner$ & $10^{-5}$ \\
    			& $\epsFTA$ & $10^{-4}$ \\
    			& $\eps$ & $10^{-4}$ \\
    			\midrule
    			Initial parameter & $\mu_\text{init}$ & $(1,1)\in\params$ \\
    			\multirow{4}{3cm}{Number of training parameters} & $\card{\paramstrainsys}$ & $100$ \\
    			& $\card{\paramstrainfta}$ & $400$ \\
    			& $\card{\paramstrain}$ & $400$ \\
    			& $\card{\paramstraininner}$ & $400$ \\
    			\bottomrule
    		\end{tabular}
    		\caption{Parameters for the reduction strategies applied to the cookie baking example.}
    		\label{tab:cookie-baking-reduction-strategies-parameters}
    	\end{minipage}%
    \end{table}
    
    \subsection{Computational results}\label{sec:computational-results}
    In this section we present the numerical results for the different reduction strategies and resulting fully reduced models. In order to guarantee a fair comparison and to show that the combined reduction is indeed beneficial, we consider an additional algorithm: The~\GROM{} (greedy reduced order model) is created by only performing a reduction of the final time adjoint state, i.e.~this reduced model is computed and evaluated as described in~Sections~3.2 and~3.4 in~\cite{kleikamp2024greedy} and was briefly introduced in~\Cref{sec:reduced-model-optimal-adjoint-states}. We remark that computing the reduced solution via the~\GROM{} as well as evaluating the error estimator of the~\GROM{} requires solving the high-dimensional (i.e.~not reduced) primal and adjoint systems several times.
    \par
    The numerical results for the different algorithms described in the previous sections are presented in~\Cref{tab:cookie-baking-results}. We depict reduced basis sizes, offline and online runtimes, as well as true absolute errors in the norm on~$\X$ and error estimations. The error in the control is measured as the temporal norm with respect to the Euclidean inner product. The online timings and errors are given as average values over a test set~$\paramstest$ consisting of~$\card{\paramstest}=50$ parameters. These test parameters are drawn randomly from~$\params$ according to a~$\log$-uniform distribution.
    \par
    \begin{table}[htb]
    	\centering
    	\resizebox{\columnwidth}{!}{%
    		\begin{tabular}{l p{3.26cm} | *6c}
    			\toprule
    			&  & {\FOM{}} & {\GROM{}} & {\SRGROM} & {\GSRROM{}} & {\GCROM{}} & {\DGROM{}} \\ \midrule \midrule
    			\multirow{3}{*}{\rotatebox{90}{\parbox[c]{1.4cm}{\medskip\centering Offline timings}}} & Overall [s] & - & $16687.64$ & $8836.19$ & $32742.67$ & $4005.50$ & $13815.39$ \\
    			\cmidrule{2-8}
    			& System reductions [s] & - & - & $7872.36$ & $16055.03$ & - & - \\
    			& Greedy algorithm [s] & - & - & $963.83$ & $16687.64$ & - & - \\
    			\midrule
    			\multirow{3}{*}{\rotatebox{90}{\parbox[c]{1.28cm}{\centering Reduced sizes}}} & $\dimRedSpaceFTA$ & - & $4$ & $4$ & $4$ & $4$ & $4$ \\
    			& $\kpr$ & - & - & $178$ & $178$ & $180$ & $505$ \\
    			& $\kad$ & - & - & $172$ & $176$ & $148$ & $371$ \\
    			\midrule
    			\multirow{7}{*}{\rotatebox{90}{\parbox[c]{4.5cm}{\bigskip\centering Online results}}} & Runtime solving [s] & $81.80$ & $11.54$ & $0.62$ & $0.63$ & $0.55$ & $3.87$ \\
    			& Speedup & - & $7.09$ & $131.77$ & $129.19$ & $149.10$ & $21.11$ \\
    			& Runtime error estimator [s] & - & $3.92$ & $0.49$ & $0.50$ & $0.45$ & $3.37$ \\
    			\cmidrule{2-8}
    			& Error final time adjoint & - & $3.10\cdot10^{-8}$ & $3.83\cdot10^{-8}$ & $3.11\cdot10^{-8}$ & $1.48\cdot10^{-7}$ & $3.10\cdot10^{-8}$ \\
    			& \emph{Estimated} error final time adjoint & - & $3.97\cdot10^{-8}$ & $3.59\cdot10^{-6}$ & $3.28\cdot10^{-6}$ & $9.23\cdot10^{-6}$ & $4.57\cdot10^{-7}$ \\
    			& Estimator efficiency & - & $1.80$ & $124.31$ & $157.17$ & $119.04$ & $22.10$ \\
    			& Error control & - & $1.36\cdot10^{-9}$ & $1.20\cdot10^{-5}$ & $1.28\cdot10^{-5}$ & $5.07\cdot10^{-6}$ & $2.54\cdot10^{-6}$ \\
    			\bottomrule
    		\end{tabular}
    	}
    	\caption{Numerical results of all considered full and reduced models when tested on the cookie baking problem. The values are rounded appropriately where necessary.}
    	\label{tab:cookie-baking-results}
    \end{table}
    The numerical results confirm a large speedup of all fully reduced models compared to the full-order model and the~\GROM{} method. In particular, compared to the~\FOM{} the fully reduced models all achieve average speedups between~$20$ and~$150$. We further observe that the error estimator for the fully reduced models can be evaluated much faster. At the same time, the fully reduced models~\SRGROM{}, \GSRROM{} and~\DGROM{} obtain true errors in the final time adjoint state of the same magnitude as the~\GROM{} without any system reduction. The~\GCROM{} performs worse than the other reduced models and looses about one order of magnitude in terms of accuracy. The error estimator of the fully reduced models possesses an average efficiency between~$20$ and~$160$ for all models which corresponds to a mild overestimation of the error. When comparing the different fully reduced models with each other, we observe that the average approximation error of the~\GCROM{} is larger than the error of the other fully reduced models while the primal reduced basis is a bit smaller. This results in a slightly larger speedup compared to the other methods. Moreover, since it is not required to solve the optimal control problem exactly for all parameters within the respective training set, the offline construction of the~\GCROM{} is much faster compared to the other algorithms, although the same number of training parameters was used, see~\Cref{tab:cookie-baking-reduction-strategies-parameters}. Furthermore, the offline time of the~\SRGROM{} is about half of the offline time of the~\GROM{}, whereas the~\SRGROM{} requires about twice as much time to be constructed. The~\GCROM{} is computed in about one-eighth of the time required to build the~\GSRROM{}. The~\DGROM{} performs different compared to the other fully reduced models in several respects. First of all, the reduced bases sizes for the system reduction are much larger which results in a speedup of only about~$21$.
    This behavior was expected since the reduced state spaces approximate the Gramian application
    over the whole training parameter set instead of only considering the Gramian application
    to states from the reduced space of optimal final time adjoints~(c.f.\ \Cref{sec:double-greedy}).
    On the other hand, the~\DGROM{} achieves errors in the final time adjoint similar to the~\GROM{}~(which uses the
    full high-dimensional state trajectories) and the best efficiency of the error estimator
    among all fully reduced models. This low overestimation by the error estimator is an immediate
    consequence of the inner greedy loop that ensures an accurate application
    of the Gramian to all final time states available in the reduced space~$\spaceVred{\dimRedSpaceFTA}$ by design.
    \par
    It is particularly conspicuous that the dimension~$\dimRedSpaceFTA$ of the reduced space for final time adjoint states equals~$4$ for all reduced models. This is because in the example considered here, a deviation from a prescribed output is penalized by means of the operator~$\Mop$. As can be deduced from~\cref{equ:linear-system-optimal-final-time-adjoint}, we have~$\optFTA{\mu}\in\im{(\Riesz{\X}\E^*)^{-1}\Mop}$ for all parameters~$\mu\in\params$, i.e.~the optimal final time adjoint state lies in the image of the operator~$(\Riesz{\X}\E^*)^{-1}\Mop$. This can be seen by rearranging the linear system for~$\optFTA{\mu}$ in order to obtain
    \begin{align*}
    	\optFTA{\mu}=(\Riesz{\X}\E^*)^{-1}\Mop\left(\StateTrans{\mu}{T}{0}x_\mu^0-x_\mu^T-\Gramian{\mu}\optFTA{\mu}\right)\in\im{(\Riesz{\X}\E^*)^{-1}\Mop}.
    \end{align*}
    However, in this numerical test case we have~$\rank{\Mop}=\rank{C^\top C}=4$ and therefore a reduced space of dimension~$\dimRedSpaceFTA=4$ for the final time adjoint states should be sufficient, since~$\rank{(\Riesz{\X}\E^*)^{-1}\Mop}\leq 4$. We observe this result also in the numerical experiments where all reduced models involving a system reduction can also benefit severely from the additional reduction of the final time adjoint states. Consequently, building a reduced space for the optimal final time adjoints is particularly useful when the deviation in a low-dimensional output is penalized in the objective functional via the matrices~$\C$ and~$\Mop$, respectively.
    \par
    Moreover, we present in~\Cref{fig:greedy-results} the maximum estimated errors over the training parameter set in the respective steps of the greedy algorithms when constructing the~\SRGROM{}, the~\GSRROM{}, the~\GCROM{} and the~\DGROM{}, i.e.~the estimated error corresponding to the selected parameter. The greedy algorithm in the computation of the~\GROM{} is the same as for the~\GSRROM{} and is therefore not shown separately. We recall at this point that during the greedy algorithm for the~\SRGROM{} and the~\GSRROM{}/\GROM{} only the matrix~$\Vred$ for the final time adjoint states is extend. In contrast, in the greedy for the~\GCROM{} and the~\DGROM{} also the reduced bases for the system dynamics are improved.
    \par
    \begin{figure}[htb]
    	\centering
    	\begin{tikzpicture}
    		\begin{axis}[
    			name=greedy,
    			DefaultStyle,
    			width=.48\textwidth,
    			xlabel={Greedy step~$\dimRedSpaceFTA$},
    			xmin=-0.1, xmax=11.1,
    			xtick={0,...,11},
    			xticklabels={0,...,11},
    			ytick={1e-7,1e-6,1e-5,1e-4,1e-3,1e-2,1e-1,1e0,1e1},
    			ymode=log,
    			log basis y={10},
    			ylabel={True and estimated errors},
    			legend style={yshift=115pt, xshift=-5.5pt},
    			]
    			\addplot[very thick, mark=none, colorUnused, dashed, samples=2] coordinates {(-0.1,1e-4) (11.1,1e-4)};
    			\addlegendentry{$\eps=10^{-4}$: Greedy tolerance}
    			\addplot[thick, colorSRGROM, mark=triangle*, mark size=3, mark options={solid, fill opacity=0.5}] table[y index=1] {data/SR-G-ROM.txt};
    			\addlegendentry{\SRGROM{} (estimator: $\FullRedEstFTA{\mu}{\dimRedSpaceFTA}$)}
    			\addplot[thick, colorSRGROM, mark=triangle, mark size=1.5, mark options={solid, fill opacity=0.5}, dashed] table[y index=2] {data/SR-G-ROM.txt};
    			\addlegendentry{\SRGROM{}: Maximum true error}
    			\addplot[thick, colorGSRROM, mark=square*, mark size=3, mark options={solid, fill opacity=0.5}] table[y index=1] {data/G-SR-ROM.txt};
    			\addlegendentry{\GSRROM{}/\GROM{} (estimator: $\EstFTA{\mu}{\dimRedSpaceFTA}$)}
    			\addplot[thick, colorGSRROM, mark=square, mark size=1.5, mark options={solid, fill opacity=0.5}, dashed] table[y index=2] {data/G-SR-ROM.txt};
    			\addlegendentry{\GSRROM{}/\GROM{}: Maximum error}
    			\addplot[thick, colorGCROM, mark=diamond*, mark size=3, mark options={solid, fill opacity=0.5}] table[y index=1] {data/GC-ROM.txt};
    			\addlegendentry{\GCROM{} (estimator: $\FullRedEstFTA{\mu}{\dimRedSpaceFTA}$)}
    			\addplot[thick, colorGCROM, mark=diamond, mark size=1.5, mark options={solid, fill opacity=0.5}, dashed] table[y index=2] {data/GC-ROM.txt};
    			\addlegendentry{\GCROM{}: Maximum true error}
    			\addplot[thick, colorDGROM, mark=*, mark size=3, mark options={solid, fill opacity=0.5}] table[y index=1] {data/DG-ROM.txt};
    			\addlegendentry{\DGROM{} (estimator: $\FullRedEstFTA{\mu}{\dimRedSpaceFTA}$)}
    			\addplot[thick, colorDGROM, mark=o, mark size=1.5, mark options={solid, fill opacity=0.5}, dashed] table[y index=2] {data/DG-ROM.txt};
    			\addlegendentry{\DGROM{}: Maximum true error}
    		\end{axis}
    		\begin{axis}[
    			DefaultStyle,
    			anchor=west,
    			at=(greedy.east),
    			xshift=2cm,
    			width=.48\textwidth,
    			xlabel={Greedy step~$\dimRedSpaceFTA$},
    			xmin=-0.1, xmax=11.1,
    			xtick={0,...,11},
    			xticklabels={0,...,11},
    			ytick={1e-7,1e-6,1e-5,1e-4,1e-3,1e-2,1e-1,1e0,1e1},
    			ymode=log,
    			log basis y={10},
    			ylabel={Estimated errors},
    			legend style={yshift=111pt, xshift=6pt},
    			]
    			\addplot[very thick, mark=none, colorUnused, dashed, samples=2] coordinates {(-0.1,1e-4) (11.1,1e-4)};
    			\addlegendentry{$\eps=10^{-4}$: Greedy tolerance}
    			\addplot[very thick, mark=none, colorUnused, dashdotted, samples=2] coordinates {(-0.1,1e-6) (11.1,1e-6)};
    			\addlegendentry{$\epsFTA=10^{-6}$: Greedy tolerance for final time adjoints}
    			\addplot[very thick, mark=none, colorUnused, dotted, samples=2] coordinates {(-0.1,1e-5) (11.1,1e-5)};
    			\addlegendentry{$\epsinner=10^{-5}$: Inner greedy tolerance}
    			\addplot[thick, colorSRGROM, mark=x, mark size=3, mark options={solid, fill opacity=0.5}] table[y index=6] {data/SR-G-ROM.txt};
    			\addlegendentry{\SRGROM{}: Maximum Gramian error}
    			\addplot[thick, colorSRGROM, mark=pentagon, mark size=3, mark options={solid, fill opacity=0.5}] table[y index=7] {data/SR-G-ROM.txt};
    			\addlegendentry{\SRGROM{}: Maximum reduced error estimate}
    			\addplot[thick, colorGCROM, mark=x, mark size=3, mark options={solid, fill opacity=0.5}] table[y index=6] {data/GC-ROM.txt};
    			\addlegendentry{\GCROM{}: Maximum Gramian error}
    			\addplot[thick, colorGCROM, mark=pentagon, mark size=3, mark options={solid, fill opacity=0.5}] table[y index=7] {data/GC-ROM.txt};
    			\addlegendentry{\GCROM{}: Maximum reduced error estimate}
    			\addplot[thick, colorDGROM, mark=x, mark size=3, mark options={solid, fill opacity=0.5}] table[y index=6] {data/DG-ROM.txt};
    			\addlegendentry{\DGROM{}: Maximum Gramian error}
    			\addplot[thick, colorDGROM, mark=pentagon, mark size=3, mark options={solid, fill opacity=0.5}] table[y index=7] {data/DG-ROM.txt};
    			\addlegendentry{\DGROM{}: Maximum reduced error estimate}
    		\end{axis}
    	\end{tikzpicture}
    	\caption{The left plot shows the maximum estimated errors over the iterations of the greedy algorithm within the construction of the~\SRGROM{}, the~\GSRROM{}/\GROM{} and the~\GCROM{}. For the~\DGROM{} we here show the outer greedy iteration (see~\Cref{fig:double-greedy-results} for results on the inner iteration). The plot on the right-hand side presents maximum values of the components of the fully reduced error estimator for the~\SRGROM{}, the~\GCROM{} and the~\DGROM{} over the respective training set. The Gramian error in the step~$\dimRedSpaceFTA=0$ is zero (and therefore not shown in the plot on the right) due to the empty reduced basis for the final time adjoint states at initialization of the three fully reduced models.}
    	\label{fig:greedy-results}
    \end{figure}
    
    As expected from the discussion in the previous paragraph, already~$4$~greedy iterations are sufficient for the~\SRGROM{}, the~\GSRROM{}/\GROM{} and the~\DGROM{} to reach the desired tolerance~$\eps$. Since the~\GCROM{} starts with empty bases for the system dynamics, it requires more iterations within the greedy procedure to construct appropriate reduced bases for the system dynamics and the optimal final time adjoints simultaneously. However, as shown in~\Cref{tab:cookie-baking-results}, the total offline costs of the~\GCROM{} are still much smaller compared to the other reduced models since the optimal control problem is solved exactly solely for the selected parameters. In the plot on the right-hand side of~\Cref{fig:greedy-results} we furthermore show the individual error components of the fully reduced error estimator~$\FullRedEstFTA{\mu}{\dimRedSpaceFTA}$ used in the greedy construction of the~\SRGROM{}, the~\GCROM{} and the~\DGROM{}. The error in the free-dynamics of the initial condition, i.e.~$\const\norm{\Mop}_{\Lin{\X}{\dualX}}\EstPrim{\mu}{0}(T)$ in~\cref{equ:a-posteriori-error-estimator-fully-reduced-model}, is zero for all parameters~$\mu\in\params$ since the initial condition is zero in this example. We therefore only display the maximum error in the Gramian~$\max_{\mu\in\paramstrain}\const\norm{\Mop}_{\Lin{\X}{\dualX}}\EstGram{\mu}{\compredFTA{\mu}{\dimRedSpaceFTA}}$ and the reduced error estimate~$\max_{\mu\in\paramstrain}\const\,\RedEstFTA{\mu}{\dimRedSpaceFTA}(\compredFTA{\mu}{\dimRedSpaceFTA})$ over the training set (the same training set~$\paramstrain=\paramstrainfta$ is used in the greedy algorithms for both the~\GCROM{} and the~\DGROM{}). For the~\SRGROM{} the smaller training set considered before for the system reduction is used in the greedy procedure. One observes that the~\SRGROM{} already starts with a Gramian error close to the greedy tolerance. Hence, the previous system reduction is sufficient to obtain a reduced model that fulfills the prescribed tolerance. However, we emphasize here again that this is only possible by choosing the same training sets for the system and the final time adjoint reduction. After~$4$~iterations of the greedy algorithm, the reduced basis for the optimal final time adjoint states has been sufficiently enriched as well such that the greedy iteration stops because the prescribed tolerance~$\epsFTA$ was reached. For the~\GCROM{}, the Gramian error is much larger at the beginning since the system reduction also takes place within the greedy algorithm. Therefore, both errors, the Gramian error and the reduced error estimate, decay over the first~$4$~iterations. Then, since a reduced space of dimension~$4$~is sufficient for the optimal final time adjoint states as discussed above, the estimate of the reduced error estimator becomes much smaller than the prescribed tolerance~$\eps$. The Gramian error, however, requires some additional iterations to become sufficiently small such that the total estimated error reaches the desired tolerance. For the~\DGROM{} we observe that the maximum Gramian error is always below the prescribed tolerance~$\epsinner=10^{-5}$ as anticipated by the double greedy structure of the algorithm. Hence, after~$4$~iterations also the reduced space of final time adjoint states is enriched in such a way that the reduced error estimator drops below the prescribed tolerance and the algorithm terminates. These observations confirm the anticipated behavior of the reduction procedures and reflects the differences between the approaches. Depending on the application at hand, one or the other strategy might be preferable.
    \par
    In~\Cref{fig:double-greedy-results} we present in detail the results of the double greedy algorithm to construct the~\DGROM{}, see~\Cref{alg:DGROM}. In particular, we show the estimated and true Gramian errors over the inner iterations in the left plot. The evolution of the reduced basis sizes for the system reduction is shown in the right plot of~\Cref{fig:double-greedy-results}.
    Furthermore, the outer iterations are highlighted, as well as the maximum estimated and true errors
    in the outer greedy iterations.
    \par
    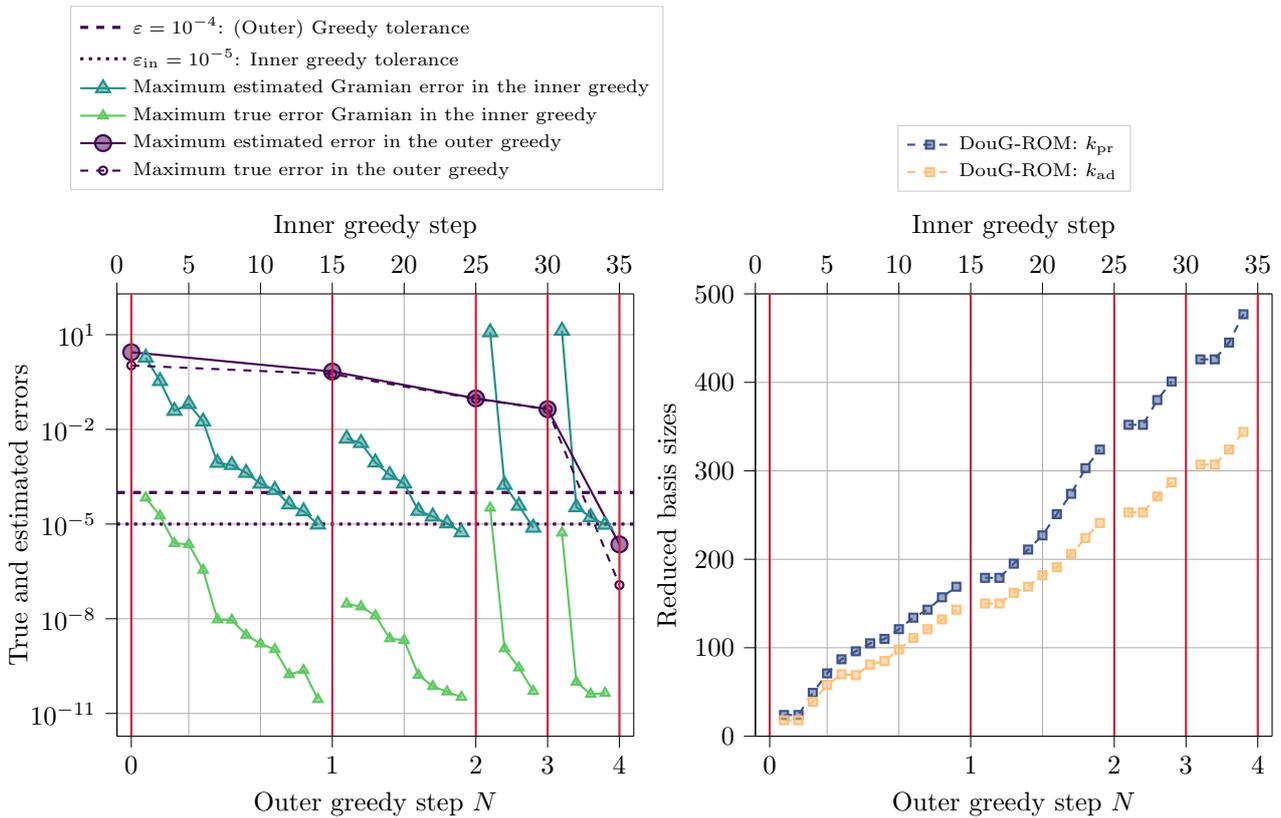
\begin{figure}[htb]
    	\centering
    	\begin{tikzpicture}
    		\begin{axis}[
    			legend cell align={left},
    			legend style={font=\scriptsize, fill opacity=1, draw opacity=1, text opacity=1, draw=white!80!black, yshift=112pt, xshift=15pt},
    			tick align=outside,
    			xmajorgrids,
    			x grid style={white!69.0196078431373!black},
    			xtick style={color=black},
    			ymajorgrids,
    			y grid style={white!69.0196078431373!black},
    			ytick style={color=black},
    			width=.4\textwidth,
    			scale only axis,
    			axis y line*=left,
    			axis x line*=top,
    			xmin=0,
    			xmax=36,
    			xlabel={Inner greedy step},
    			xlabel near ticks,
    			ymode=log,
    			log basis y={10},
    			]
    			\addplot[very thick, mark=none, colorUnused, dashed, samples=2] coordinates {(0,1e-4) (44,1e-4)};
    			\addlegendentry{$\eps=10^{-4}$: (Outer) Greedy tolerance}
    			\addplot[very thick, mark=none, colorUnused, dotted, samples=2] coordinates {(0,1e-5) (44,1e-5)};
    			\addlegendentry{$\epsinner=10^{-5}$: Inner greedy tolerance}
    			\addplot[restrict x to domain=2:14, thick, viridisTeal, mark=triangle*, mark size=3, mark options={solid, fill opacity=0.5}] table[y index=2, forget plot] {data/DG-ROM_inner_greedy_plotting.txt};
    			\addplot[restrict x to domain=16:24, thick, viridisTeal, mark=triangle*, mark size=3, mark options={solid, fill opacity=0.5}] table[y index=2, forget plot] {data/DG-ROM_inner_greedy_plotting.txt};
    			\addplot[restrict x to domain=26:29, thick, viridisTeal, mark=triangle*, mark size=3, mark options={solid, fill opacity=0.5}] table[y index=2, forget plot] {data/DG-ROM_inner_greedy_plotting.txt};
    			\addplot[restrict x to domain=31:34, thick, viridisTeal, mark=triangle*, mark size=3, mark options={solid, fill opacity=0.5}] table[y index=2] {data/DG-ROM_inner_greedy_plotting.txt};
    			\addlegendentry{Maximum estimated Gramian error in the inner greedy}
    			\addplot[restrict x to domain=2:14, thick, viridisGreen, mark=triangle*, mark size=2, mark options={solid, fill opacity=0.5}] table[y index=3, forget plot] {data/DG-ROM_inner_greedy_plotting.txt};
    			\addplot[restrict x to domain=16:24, thick, viridisGreen, mark=triangle*, mark size=2, mark options={solid, fill opacity=0.5}] table[y index=3, forget plot] {data/DG-ROM_inner_greedy_plotting.txt};
    			\addplot[restrict x to domain=26:29, thick, viridisGreen, mark=triangle*, mark size=2, mark options={solid, fill opacity=0.5}] table[y index=3, forget plot] {data/DG-ROM_inner_greedy_plotting.txt};
    			\addplot[restrict x to domain=31:34, thick, viridisGreen, mark=triangle*, mark size=2, mark options={solid, fill opacity=0.5}] table[y index=3] {data/DG-ROM_inner_greedy_plotting.txt};
    			\addlegendentry{Maximum true error Gramian in the inner greedy}
    			\addplot[thick, colorDGROM, mark=*, mark size=3, mark options={solid, fill opacity=0.5}] table[x index=1, y index=2] {data/DG-ROM_outer_greedy_plotting.txt};
    			\addlegendentry{Maximum estimated error in the outer greedy}
    			\addplot[thick, colorDGROM, mark=o, mark size=1.5, mark options={solid, fill opacity=0.5}, dashed] table[x index=1, y index=3] {data/DG-ROM_outer_greedy_plotting.txt};
    			\addlegendentry{Maximum true error in the outer greedy}
    		\end{axis}
    		\begin{axis}[
    			name=double-greedy,
    			tick align=outside,
    			xmajorgrids,
    			x grid style={matchingRed, thick},
    			xtick style={color=black},
    			width=.4\textwidth,
    			scale only axis,
    			axis x line*=bottom,
    			xtick=data,
    			xticklabels={0,...,4},
    			xlabel={Outer greedy step~$\dimRedSpaceFTA$},
    			xlabel near ticks,
    			xmin=0,
    			xmax=36,
    			ymajorticks=false,
    			ylabel={True and estimated errors},
    			ylabel shift=30pt,
    			]
    			\addplot[draw=none] table[x index=1, y index=1] {data/DG-ROM_outer_greedy_plotting.txt};
    		\end{axis}
    		\begin{axis}[
    			anchor=west,
    			at=(double-greedy.east),
    			xshift=1.6cm,
    			legend cell align={left},
    			legend style={font=\scriptsize, fill opacity=1, draw opacity=1, text opacity=1, draw=white!80!black, xshift=-49pt, yshift=67pt},
    			tick align=outside,
    			xmajorgrids,
    			x grid style={white!69.0196078431373!black},
    			xtick style={color=black},
    			ymajorgrids,
    			y grid style={white!69.0196078431373!black},
    			ytick style={color=black},
    			width=.4\textwidth,
    			scale only axis,
    			axis y line*=left,
    			axis x line*=top,
    			xmin=0,
    			xmax=36,
    			xlabel={Inner greedy step},
    			xlabel near ticks,
    			ylabel={Reduced basis sizes},
    			ymin=0,
    			ymax=500,
    			]
    			\addplot[restrict x to domain=2:14, thick, viridisBlue, mark=square*, mark size=1.5, dashed, mark options={solid, fill opacity=0.5}] table[y index=7, forget plot] {data/DG-ROM_inner_greedy_plotting.txt};
    			\addplot[restrict x to domain=16:24, thick, viridisBlue, mark=square*, mark size=1.5, dashed, mark options={solid, fill opacity=0.5}] table[y index=7, forget plot] {data/DG-ROM_inner_greedy_plotting.txt};
    			\addplot[restrict x to domain=26:29, thick, viridisBlue, mark=square*, mark size=1.5, dashed, mark options={solid, fill opacity=0.5}] table[y index=7, forget plot] {data/DG-ROM_inner_greedy_plotting.txt};
    			\addplot[restrict x to domain=31:34, thick, viridisBlue, mark=square*, mark size=1.5, dashed, mark options={solid, fill opacity=0.5}] table[y index=7] {data/DG-ROM_inner_greedy_plotting.txt};
    			\addlegendentry{\DGROM{}: $\kpr$}
    			\addplot[restrict x to domain=2:14, thick, matchingOrange, mark=square*, mark size=1.5, dashed, mark options={solid, fill opacity=0.5}] table[y index=8, forget plot] {data/DG-ROM_inner_greedy_plotting.txt};
    			\addplot[restrict x to domain=16:24, thick, matchingOrange, mark=square*, mark size=1.5, dashed, mark options={solid, fill opacity=0.5}] table[y index=8, forget plot] {data/DG-ROM_inner_greedy_plotting.txt};
    			\addplot[restrict x to domain=26:29, thick, matchingOrange, mark=square*, mark size=1.5, dashed, mark options={solid, fill opacity=0.5}] table[y index=8, forget plot] {data/DG-ROM_inner_greedy_plotting.txt};
    			\addplot[restrict x to domain=31:34, thick, matchingOrange, mark=square*, mark size=1.5, dashed, mark options={solid, fill opacity=0.5}] table[y index=8] {data/DG-ROM_inner_greedy_plotting.txt};
    			\addlegendentry{\DGROM{}: $\kad$}
    		\end{axis}
    		\begin{axis}[
    			anchor=west,
    			at=(double-greedy.east),
    			xshift=1.6cm,
    			legend cell align={left},
    			legend style={font=\scriptsize, fill opacity=1, draw opacity=1, text opacity=1, draw=white!80!black},
    			tick align=outside,
    			xmajorgrids,
    			x grid style={matchingRed, thick},
    			xtick style={color=black},
    			width=.4\textwidth,
    			scale only axis,
    			axis x line*=bottom,
    			xtick=data,
    			xticklabels={0,...,4},
    			xlabel={Outer greedy step~$\dimRedSpaceFTA$},
    			xlabel near ticks,
    			xmin=0,
    			xmax=36,
    			ymajorticks=false,
    			ymin=0,
    			ymax=500,
    			]
    			\addplot[draw=none] table[x index=1, y index=0] {data/DG-ROM_outer_greedy_plotting.txt};
    		\end{axis}
    	\end{tikzpicture}
    	\caption{True and estimated errors (left) and reduced basis sizes (right) during the inner and outer iterations of the double greedy algorithm.}
    	\label{fig:double-greedy-results}
    \end{figure}
    One immediately notices that the error in the Gramian application is significantly overestimated.
    We believe that this is a consequence of the fact that
    the considered estimator measures the error in the full primal solution at final time instead of the
    error in the low-dimensional output~(also see~\Cref{rem:improved-error-estimator-adjoint-problem}).
    We further observe that even in the last iterations of the outer greedy algorithm, large errors
    in the Gramian can occur which then require a further enrichment of the reduced primal and adjoint
    state spaces. This is also reflected by the observation that the reduced basis
    sizes~$\kpr$ and~$\kad$ still increase by about~$150$ and~$100$, respectively,
    in the last two outer iterations. Interestingly, in order to capture the behavior
    of the Gramian over the whole training set (which is ensured by the inner greedy loop),
    a significantly larger primal reduced basis, compared to the size of the adjoint basis, is required.

	\section{Conclusion and outlook}\label{sec:conclusion-outlook}
	In this contribution we combined different reduced order models in a fully reduced model to speed up the solution of parametrized linear-quadratic optimal control problems. Due to the structure of the objective functional and the state equation, the optimal control is already uniquely characterized by the optimal adjoint state at final time. Hence, one part of the fully reduced model consists of a reduced space to approximate the manifold of optimal final time adjoint states over the parameter set. In this paper we extended this approach to linear time-varying systems in a general Hilbert space setting. The runtime complexity of such a reduced model for the final time adjoint states still depends on the dimension of the state space. This constitutes a serious limitation of the method in particular when applied to discretizations of PDEs in multiple dimensions and on grids with fine resolutions which typically result in high-dimensional spaces. We therefore extend the approach by employing an additional reduction step that aims at lowering the computational costs for solving the primal and the adjoint system.
	This is achieved by a Petrov-Galerkin projection of the primal and adjoint state equations onto
	reduced state spaces which are constructed by compressing the full trajectories
	for certain training parameters.
	\par
	To assess and evaluate the accuracy of the fully reduced model, we developed an a posteriori error estimator and proved its reliability. This estimator can be evaluated very efficiently utilizing an
	offline-online decomposition of its high-dimensional components which results in a complexity
	independent of the potentially large state dimensions.
	\par
	We further propose different strategies for constructing the reduced order models.
	One the one hand, it is possible to compute the reduced bases for the dynamical systems
	and for the optimal final time adjoint states independent from each other.
	However, one can also make use of the a posteriori error estimator within a greedy algorithm to extend
	all involved reduced bases simultaneously.
	\par
	The considered numerical experiment of a two-dimensional heat equation problem shows the enormous speedup which is possible by combining the two reduced order models. The fully reduced models achieve errors in the final time adjoint state that are only slightly worse than those of the reduced order model without any system reductions. We further observe that the fully reduced models benefit from the additional reduction of the final time adjoint states without loosing too much in terms of accuracy. Regarding the offline runtimes it turns out that the~\GCROM{}, which is built by performing a greedy algorithm directly on the fully reduced model thereby making use of the available error estimator, can be constructed much faster than the other reduced order models. The~\DGROM{} constructed by a double greedy algorithm shows superior performance in terms of accuracy and efficiency of the error estimator compared to the other fully reduced models. In particular, we believe that the efficiency of the error estimator can also be
	theoretically verified which cannot be expected for the~\GCROM{}.
	However, this increased robustness of the~\DGROM{} also results in a smaller speedup.
	\par
	As we have seen in the numerical example and discussed extensively in~\Cref{sec:computational-results}, the fully reduced model is particularly useful for problems with a low-dimensional output. In such a case, which frequently occurs in practical applications, the set of optimal final time adjoint states inherits this low-dimensional structure and is thus amenable for projection-based model order reduction approaches.
	\par
	The problem setting discussed in this contribution might be extended to other optimal control problems. For instance, further work could concern applications where the objective functional is not quadratic with respect to the control and the state, or problems with bounds on the state or the control. Another interesting research direction would be the penalization of the whole state trajectory instead of measuring solely the deviation of the state at final time from a prescribed target.
	In that case, the full trajectory of the adjoint state is required to uniquely determine the
	optimal control. Hence, a compression of the adjoint state trajectory as already investigated in this paper constitutes a possible way to speed up the approximate solution of the optimal control problem in such a setting, see for instance~\cite{gubisch2017proper} for an investigation of the parameter-independent case. In order to make an approach in that direction certified, a suitable error estimator would be required. Furthermore, the reduced basis construction for the system reduction by compressing trajectories using~POD as presented in this paper could be replaced by other approaches from parametrized model order reduction of dynamical systems. In particular when the deviation in an output quantity instead of a high-dimensional state is considered, ideas such as interpolatory methods, building for instance on balanced truncation for non-parametrized systems, might be better suited, see~\cite{baur2017comparison} for an overview of existing approaches. Regarding the reduction of the system trajectories, the POD-greedy approach introduced in~\cite{haasdonk2008reduced} and further analyzed in~\cite{haasdonk2013convergence} might also be of interest. The ideas for combining the two reduced order models and the error estimation as depicted in~\Cref{sec:fully-reduced-model} are independent of the construction of the reduced bases involved. It is therefore readily possible to include other basis generation algorithms in the computation of the fully reduced model. Moreover, improving the error estimator for the fully reduced model by taking into account the output operator directly, might be of interest for future research, see also~\Cref{rem:improved-error-estimator-adjoint-problem}.
	\par
	In~\cite{kleikamp2024application}, an adaptive model hierarchy for parametrized problems consisting of a full order model, a reduced order model and a machine learning surrogate was applied to parametrized optimal control problems with linear time-invariant system dynamics. The model hierarchy was first developed and discussed in detail in~\cite{haasdonk2023certified} for parametrized parabolic~PDEs. The fully reduced models described in this paper could be integrated as an additional level in the model hierarchy and might be combined with the machine learning surrogate introduced in~\cite{kleikamp2024greedy} to obtain an additional speedup. Due to the available reliable error estimator of the fully reduced model, the extended hierarchy could still return results that meet a prescribed tolerance.

    \printbibliography

    \section*{Statements and declarations}

	\paragraph{Acknowledgments}
	The authors would like to thank Michael Kartmann for helpful advice on how to properly incorporate the mass operator in the optimality system as well as Julia Schleuß for several constructive conversations in particular regarding practical aspects concerning the cookie baking test case. Further, the authors would like to thank Mario Ohlberger, Stephan Rave and Jens Saak for fruitful discussions and advice.

    \paragraph{Funding}
    The authors acknowledge funding by the Deutsche Forschungsgemeinschaft (DFG, German Research Foundation) under Germany's Excellence Strategy EXC 2044 –390685587, Mathematics Münster: Dynamics–Geometry–Structure.

    \paragraph{Competing interests}
    The authors have no relevant financial or non-financial interests to disclose.

    \paragraph{Author contribution statements}
    \begin{itemize}
		\item H.~Kleikamp: Conceptualization, Methodology, Software, Validation, Formal analysis, Investigation, Writing -- Original Draft, Writing -- Review \& Editing, Visualization
		\item L.~Renelt: Methodology, Formal analysis, Writing -- Review \& Editing
    \end{itemize}

    \paragraph{Code availability}
    The source code used to carry out the numerical experiments presented in this contribution can be found in~\cite{sourcecode}.

	\begin{appendices}
		\crefalias{section}{appendix}
		\crefalias{subsection}{appendix}	    
	    \section{Proof of~\texorpdfstring{\Cref{thm:optimality-system}}{Theorem~\ref{thm:optimality-system}}}\label{app:proof-optimality-system}
	    For simplicity of notation, we neglect the parameter~$\mu\in\params$ in the following derivation.
	    \begin{proof}
	    	Let us denote by~$u^*\in\G$ an optimal control and let~$x^*\in\H$ be the respective state trajectory. We are going to prove that the first variation of~$\optFunc{}$ vanishes if~$x^*$, $\varphi^*$ and~$u^*$ solve the boundary value problem stated in~\Cref{thm:optimality-system}. To this end, let~$v\in\G$ be arbitrary and consider the perturbation~$u\in\G$ of~$u^*$ given by
	    	\begin{align*}
	    		u \coloneqq u^* + \varepsilon v
	    	\end{align*}
	    	for~$\varepsilon\in\R$. Therefore, the state equation when applying the control~$u$ reads
	    	\begin{align*}
	    		\E\frac{d}{dt}x(t) = \A{t}x(t)+\B{t}u^*(t)+\varepsilon \B{t}v(t)\qquad\text{for }t\in[0,T].
	    	\end{align*}
	    	We write the solution of this state equation as
	    	\begin{align*}
	    		x(t) = x^*(t) + \varepsilon z(t)
	    	\end{align*}
	    	for~$z\in\H$ where~$z$ satisfies the differential equation
	    	\begin{align*}
	    		\E\frac{d}{dt}z(t) = \A{t}z(t)+\B{t}v(t)\quad\text{for }t\in[0,T],\qquad z(0)=0.
	    	\end{align*}
	    	The initial condition~$z(0)=0$ follows from the fact that~$x(0)=x^*(0)=x^0$. We now introduce the adjoint state~$\varphi\in\H$ and define the Hamiltonian function~$\Ham\colon\X\times\U\times\X\times[0,T]\to\R$ as
	    	\begin{align*}
	    		\Ham(x(t),u(t),\varphi(t),t) = \frac{1}{2} \innerProd{u(t)}{R(t)u(t)}_{\U\times\dualU} + \innerProd{\varphi(t)}{\A{t}x(t)+\B{t}u(t)}_{\X\times\dualX}.
	    	\end{align*}
	    	Using further that
	    	\begin{align*}
	    		\normY{\C\left(x(T)-x^T\right)}^2=\innerProd{\C^*\Riesz{\Y}\C\left(x(T)-x^T\right)}{x(T)-x^T}_{\dualX\times\X}=\innerProd{x(T)-x^T}{\Mop\left(x(T)-x^T\right)}_{\X\times\dualX},
	    	\end{align*}
	    	we can thus rewrite the functional~$\optFunc{}$ as
	    	\begin{align*}
	    		\optFunc{}(u) = \frac{1}{2}\innerProd{x(T)-x^T}{\Mop\left(x(T)-x^T\right)}_{\X\times\dualX} + \int\limits_0^T \Ham(x(t),u(t),\varphi(t),t) - \innerProd{\varphi(t)}{\E\frac{d}{dt}x(t)}_{\X\times\dualX} \d{t},
	    	\end{align*}
	    	since~$x\in H$ solves the state equation for the control~$u$. This equation holds for any adjoint state~$\varphi\in H$. Similarly, for the optimal control~$u^*$ and corresponding state trajectory~$x^*$ we have
	    	\begin{align*}
	    		\optFunc{}(u^*) = \frac{1}{2}\innerProd{x^*(T)-x^T}{\Mop\left(x^*(T)-x^T\right)}_{\X\times\dualX} + \int\limits_0^T \Ham(x^*(t),u^*(t),\varphi(t),t) - \innerProd{\varphi(t)}{\E\frac{d}{dt}x^*(t)}_{\X\times\dualX}\d{t}.
	    	\end{align*}
	    	Then, the difference~$\optFunc{}(u)-\optFunc{}(u^*)$ is given as
	    	\begin{equation}\label{equ:difference-cost-functional}
	    		\begin{aligned}
	    			\optFunc{}(u)-\optFunc{}(u^*) &= \frac{1}{2}\Big[\innerProd{x(T)-x^T}{\Mop\left(x(T)-x^T\right)}_{\X\times\dualX} - \innerProd{x^*(T)-x^T}{\Mop\left(x^*(T)-x^T\right)}_{\X\times\dualX}\Big] \\
	    			& \hphantom{==}+ \int\limits_0^T \Ham(x(t),u(t),\varphi(t),t)-\Ham(x^*(t),u^*(t),\varphi(t),t)\d{t} \\
	    			& \hphantom{==}+ \int\limits_0^T \innerProd{\varphi(t)}{\E\frac{d}{dt}(x^*(t)-x(t))}_{\X\times\dualX}\d{t}.
	    		\end{aligned}
	    	\end{equation}
	    	We obtain for the first term in~\cref{equ:difference-cost-functional} the identity
	    	\begin{align*}
	    		& \frac{1}{2}\Big[\innerProd{x(T)-x^T}{\Mop\big(x(T)-x^T\big)}_{\X\times\dualX}-\innerProd{x^*(T)-x^T}{\Mop\big(x^*(T)-x^T\big)}_{\X\times\dualX}\Big] \\
	    		&= \frac{1}{2}\Big[\innerProd{x^*(T)+\varepsilon z(T)-x^T}{\Mop\big(x^*(T)+\varepsilon z(T)-x^T\big)}_{\X\times\dualX}-\innerProd{x^*(T)-x^T}{\Mop\big(x^*(T)-x^T\big)}_{\X\times\dualX}\Big] \\
	    		&= \varepsilon \innerProd{z(T)}{\Mop\big(x^*(T)-x^T\big)}_{\X\times\dualX}+\mathcal{O}(\varepsilon^2),
	    	\end{align*}
	    	where we used that~$\Riesz{\X}^{-1}\Mop$ is self-adjoint. Moreover, for the difference of the Hamiltonians in the second term in~\cref{equ:difference-cost-functional} we have
	    	\begin{align*}
	    		& \Ham(x(t),u(t),\varphi(t),t)-\Ham(x^*(t),u^*(t),\varphi(t),t) \\
	    		&= \frac{1}{2} \innerProd{u(t)}{R(t)u(t)}_{\U\times\dualU} + \innerProd{\varphi(t)}{\A{t}x(t)+\B{t}u(t)}_{\X\times\dualX} \\
	    		&\hphantom{==}- \frac{1}{2} \innerProd{u^*(t)}{R(t)u^*(t)}_{\U\times\dualU} - \innerProd{\varphi(t)}{\A{t}x^*(t)+\B{t}u^*(t)}_{\X\times\dualX} \\
	    		&= \frac{1}{2}\innerProd{u^*(t) + \varepsilon v(t)}{R(t)\big(u^*(t) + \varepsilon v(t)\big)}_{\U\times\dualU} + \innerProd{\varphi(t)}{\A{t}x^*(t)+\varepsilon \A{t}z(t)+\B{t}u^*(t)+\varepsilon \B{t}v(t)}_{\X\times\dualX} \\
	    		& \hphantom{==}- \frac{1}{2}\innerProd{u^*(t)}{R(t)u^*(t)}_{\U\times\dualU} - \innerProd{\varphi(t)}{\A{t}x^*(t)+\B{t}u^*(t)}_{\X\times\dualX} \\
	    		&= \varepsilon \innerProd{u^*(t)}{R(t)v(t)}_{\U\times\dualU} + \innerProd{\varphi(t)}{\varepsilon \A{t}z(t)+\varepsilon \B{t}v(t)}_{\X\times\dualX} + \mathcal{O}(\varepsilon^2) \\
	    		&= \varepsilon\Big[\innerProd{u^*(t)}{R(t)v(t)}_{\U\times\dualU} + \innerProd{\varphi(t)}{\A{t}z(t)}_{\X\times\dualX} + \innerProd{\varphi(t)}{\B{t}v(t)}_{\X\times\dualX}\Big] + \mathcal{O}(\varepsilon^2) \\
	    		&= \varepsilon\Big[\innerProd{v(t)}{R(t)u^*(t)+\B{t}^*\varphi(t)}_{\U\times\dualU} + \innerProd{\varphi(t)}{\A{t}z(t)}_{\X\times\dualX}\Big] + \mathcal{O}(\varepsilon^2)
	    	\end{align*}
	    	for~$t\in[0,T]$, where we used that~$\Riesz{\U}^{-1}\Rop$ is self-adjoint as well. Furthermore, recall that~$x(t)=x^*(t)+\varepsilon z(t)$ and therefore~$\E\left(\frac{d}{dt}x^*(t)-\frac{d}{dt}x(t)\right)=-\varepsilon\E\frac{d}{dt}z(t)$ for all~$t\in[0,T]$. For the last term in~\cref{equ:difference-cost-functional} it hence holds
	    	\begin{align*}
	    		\int\limits_0^T \innerProd{\varphi(t)}{\E\frac{d}{dt}(x^*(t)-x(t))}_{\X\times\dualX}\d{t} &= -\varepsilon\int\limits_0^T \innerProd{\E^*\varphi(t)}{\frac{d}{dt}z(t)}_{\X\times\dualX}\d{t} \\
	    		&= \big[-\varepsilon\innerProd{\E^*\varphi(t)}{z(t)}_{\X}\big]_0^T + \varepsilon\int\limits_0^T \innerProd{\frac{d}{dt}\left(\E^*\varphi(t)\right)}{z(t)}_{\dualX\times\X}\d{t} \\
	    		&= -\varepsilon\innerProd{z(T)}{\E^*\varphi(T)}_{\X} + \varepsilon\int\limits_0^T \innerProd{z(t)}{\Ead\frac{d}{dt}\varphi(t)}_{\X\times\dualX}\d{t},
	    	\end{align*}
	    	where we used integration by parts and the initial condition~$z(0)=0$. In the last line, the identity
	    	\begin{equation*}
	    		\frac{d}{dt}(\E^*\varphi(t)) \;=\; \Riesz{\X}\E^*\Riesz{\X}^{-1}\frac{d}{dt}\varphi(t)
	    		\;=\; \Ead\frac{d}{dt}\varphi
	    	\end{equation*}
	    	was used which can be seen as follows: Let~$\psi\in C^1([0,T];\X)$, $x\in\X$ and~$t\in[0,T]$. Then it holds
	    	\begin{align*}
	    		\innerProd{\frac{d}{dt}(\E^*\psi(t))}{x}_{\dualX\times\X} &= \lim\limits_{h\to 0}\innerProd{\frac{\E^*\psi(t+h)-\E^*\psi(t)}{h}}{x}_{\X\times\X} \\
	    		&= \lim\limits_{h\to 0}\innerProd{\frac{\psi(t+h)-\psi(t)}{h}}{\E\Riesz{\X}x}_{\X\times\dualX} \\
	    		&= \innerProd{\frac{d}{dt}\psi(t)}{\E\Riesz{\X}x}_{\dualX\times\dualX} \\
	    		&= \innerProd{\Riesz{\X}^{-1}\frac{d}{dt}\psi(t)}{\E\Riesz{\X}x}_{\X\times\dualX} \\
	    		&= \innerProd{\Riesz{\X}\E^*\Riesz{\X}^{-1}\frac{d}{dt}\psi(t)}{x}_{\dualX\times\X}.
	    	\end{align*}
	    	The result for general~$\psi\in\H$ follows by density arguments. Combining the previous statements yields
	    	\begin{align*}
	    		\optFunc{}(u)-\optFunc{}(u^*) &= \varepsilon\left[\innerProd{z(T)}{\Mop\big(x^*(T)-x^T\big)}_{\X\times\dualX}\vphantom{\int\limits_0^T}\right. \\
	    		& \hphantom{=\varepsilon=}\left.+ \int\limits_0^T \innerProd{v(t)}{R(t)u^*(t)+\B{t}^*\varphi(t)}_{\U\times\dualU} + \innerProd{\varphi(t)}{\A{t}z(t)}_{\X\times\dualX}\d{t} \right. \\
	    		& \hphantom{=\varepsilon=}\left. + \int\limits_0^T \innerProd{z(t)}{\Ead\frac{d}{dt}\varphi(t)}_{\X\times\dualX}\d{t} - \innerProd{z(T)}{\E^*\varphi(T)}_{\X}\right] + \mathcal{O}(\varepsilon^2) \\
	    		&= \varepsilon\left[\innerProd{z(T)}{\Mop\left(x^*(T)-x^T\right) - \Riesz{\X}\E^*\varphi(T)}_{\X\times\dualX} + \int\limits_0^T \innerProd{v(t)}{R(t)u^*(t)+\B{t}^*\varphi(t)}_{\U\times\dualU}\d{t} \right. \\
	    		& \hphantom{=\varepsilon=}\left. + \int\limits_0^T \innerProd{z(t)}{\Ead\frac{d}{dt}\varphi(t)+\A{t}^*\varphi(t)}_{\X\times\dualX}\d{t}\right] + \mathcal{O}(\varepsilon^2).
	    	\end{align*}
	    	Since~$u^*$ is assumed to be an optimal control, it has to hold for all~$v\in \U$ that
	    	\begin{align*}
	    		0 &= \lim\limits_{\varepsilon\to 0}\frac{\optFunc{}(u^*+\varepsilon v)-\optFunc{}(u^*)}{\varepsilon} \\
	    		&= \lim\limits_{\varepsilon\to 0}\frac{\optFunc{}(u)-\optFunc{}(u^*)}{\varepsilon} \\
	    		&= \innerProd{z(T)}{\Mop\big(x^*(T)-x^T\big)-\Riesz{\X}\E^*\varphi(T)}_{\X\times\dualX} \\
	    		& \hphantom{==}+ \int\limits_0^T \innerProd{v(t)}{R(t)u^*(t)+\B{t}^*\varphi(t)}_{\U\times\dualU}\d{t} + \int\limits_0^T \innerProd{z(t)}{\Ead\frac{d}{dt}\varphi(t)+\A{t}^*\varphi(t)}_{\X\times\dualX}\d{t}.
	    	\end{align*}
	    	This results in the necessary conditions for~$u^*$, $x^*$ and~$\varphi^*$ as claimed in~\Cref{thm:optimality-system}.
	    \end{proof}
	    
	    \section{Details on offline-online decompositions of the error estimators and the fully reduced model}
	    In this section we provide some practical details for the abstract formulation given above. To this end, all operators are treated as matrices, i.e.~$\E\in\R^{n\times n}$, $\A{\mu;t}\in\R^{n\times n}$, $\B{\mu;t}\in\R^{n\times m}$, $\Mop\in\R^{n\times n}$ and~$\Rop(t)\in\R^{m\times m}$. The reduced space~$\spaceVpr$, $\spaceWpr$, $\spaceVad$, $\spaceWad$ and~$\spaceVred{\dimRedSpaceFTA}$ are further represented as matrices~$\Vpr\in\R^{n\times\kpr}$, $\Wpr\in\R^{n\times\kpr}$, $\Vad\in\R^{n\times\kad}$, $\Wad\in\R^{n\times\kad}$ and~$\Vred\in\R^{n\times\dimRedSpaceFTA}$ as shown in~\Cref{rem:matrix-case}. We assume for simplicity that the matrices are assembled in such a form that the respective Riesz maps are already incorporated in the matrices and do not have to be treated separately. This is often the case when assembling the system matrices of a finite element method where the reduced coefficients correspond to the coefficient representation of the associated element from the finite element space. We further assume that the inner product on the space~$\X$ is induced by a matrix~$\Gmat$ as discussed in~\Cref{sec:numerical-experiments}.
	    
	    \subsection{Error estimators for the primal and adjoint systems}\label{app:details-offline-online-decompositions-primal-and-adjoint}
	    We can rewrite the squared norm of the primal residual for~$\hat{x}\in\R^{\kpr}$, $u\in\R^m$, $\mu\in\params$ and~$t\in[0,T]$ as
	    \begin{align*}
	    	\normX{\E^{-1}\ResPrim{\mu}{\hat{x}}{u}(t)}^2 &= \ResPrim{\mu}{\hat{x}}{u}(t)^\top\E^{-\top} \Gmat\E^{-1}\ResPrim{\mu}{\hat{x}}{u}(t) \\
	    	&= \hat{x}^\top \Vpr^\top \A{\mu;t}^\top \E^{-\top} \Gmat\E^{-1}\A{\mu;t} \Vpr\hat{x} + u^\top \B{\mu;t}^\top \E^{-\top} \Gmat\E^{-1}\B{\mu;t} u\\
	    	&\hphantom{=}+\left(\frac{d}{dt}\hat{x}\right)^\top \Vpr^\top \Gmat\Vpr\left(\frac{d}{dt}\hat{x}\right)+2u^\top \B{\mu;t}^\top \E^{-\top} \Gmat\E^{-1}\A{\mu;t} \Vpr\hat{x}\\
	    	&\hphantom{=}-2\left(\frac{d}{dt}\hat{x}\right)^\top \Vpr^\top \Gmat\E^{-1}\A{\mu;t} \Vpr\hat{x}-2\left(\frac{d}{dt}\hat{x}\right)^\top \Vpr^\top \Gmat\E^{-1}\B{\mu;t} u \\
	    	&= \hat{x}^\top M_1(\mu;t)\hat{x} + u^\top M_2(\mu;t)u+\left(\frac{d}{dt}\hat{x}\right)^\top M_3\left(\frac{d}{dt}\hat{x}\right) \\
	    	&\hphantom{=}+2u^\top M_4(\mu;t)\hat{x}-2\left(\frac{d}{dt}\hat{x}\right)^\top M_5(\mu;t)\hat{x}-2\left(\frac{d}{dt}\hat{x}\right)^\top M_6(\mu;t)u,
	    \end{align*}
	    with
	    \begin{align*}
	    	M_1(\mu;t) &\coloneqq \Vpr^\top \A{\mu;t}^\top \E^{-\top} \Gmat\E^{-1}\A{\mu;t} \Vpr\in\R^{\kpr\times \kpr}, \\
	    	M_2(\mu;t) &\coloneqq \B{\mu;t}^\top \E^{-\top} \Gmat\E^{-1}\B{\mu;t}\in\R^{m\times m}, \\
	    	M_3 &= \Vpr^\top \Gmat\Vpr\in\R^{\kpr\times \kpr}, \\
	    	M_4(\mu;t) &= \B{\mu;t}^\top \E^{-\top} \Gmat\E^{-1}\A{\mu;t} \Vpr\in\R^{m\times \kpr}, \\
	    	M_5(\mu;t) &= \Vpr^\top \Gmat\E^{-1}\A{\mu;t} \Vpr\in\R^{\kpr\times \kpr}, \\
	    	M_6(\mu;t) &= \Vpr^\top \Gmat\E^{-1}\B{\mu;t}\in\R^{\kpr\times m}.
	    \end{align*}
	    The size of the matrices defined above only depends on the dimension~$\kpr$ of the reduced space and on the dimension~$m$ of the control space. Hence, under the assumption of parameter-separability of the system matrices, see~\cref{equ:parameter-separability}, one can efficiently compute the reduced matrices during the online phase. To precompute the parameter-independent components of~$M_1$, $M_2$, $M_4$, $M_5$, and~$M_6$ during the offline phase, we define
	    \begin{align*}
	    	M_1^{q,q'} \coloneqq \Vpr^\top \left(A^q\right)^\top \E^{-\top} \Gmat\E^{-1}A^{q'}\Vpr\in\R^{\kpr\times\kpr}
	    \end{align*}
	    for~$q,q'=1,\dots,Q_A$ and similarly for~$M_2^{q,q'}$, $M_4^{q,q'}$, $M_5^{q,q'}$, and~$M_6^{q,q'}$. Then it holds
	    \begin{align*}
	    	M_1(\mu;t) = \sum\limits_{q,q'=1}^{Q_A}\theta_A^q(\mu;t)\theta_A^{q'}(\mu;t)M_1^{q,q'}.
	    \end{align*}
	    Furthermore, it holds
	    \begin{align*}
	    	x_\mu^0-\Projection{\spaceVpr}x_\mu^0 = x_\mu^0-\Vpr\Vpr^\top\Gmat x_\mu^0 = (\eye{n}-\Vpr\Vpr^\top\Gmat)x_\mu^0
	    \end{align*}
	    and thus
	    \begin{align*}
	    	\normX{x_\mu^0-\Projection{\spaceVpr}x_\mu^0}^2 = \left(x_\mu^0\right)^\top(\eye{n}-\Vpr\Vpr^\top\Gmat)^\top\Gmat(\eye{n}-\Vpr\Vpr^\top\Gmat)x_\mu^0.
	    \end{align*}
	    This can be decomposed by defining
	    \begin{align*}
	    	m_0^{q,q'}\coloneqq \left(x_0^q\right)^\top(\eye{n}-\Vpr\Vpr^\top\Gmat)^\top\Gmat(\eye{n}-\Vpr\Vpr^\top\Gmat)\left(x_0^{q'}\right)
	    \end{align*}
	    for~$q,q'=1,\dots,Q_{x^0}$ and then assembling the squared norm as
	    \begin{align*}
	    	\normX{x_\mu^0-\Projection{\spaceVpr}x_\mu^0}^2 = \sum\limits_{q,q'=1}^{Q_{x^0}}\theta_{x^0}^q(\mu)\theta_{x^0}^{q'}(\mu)m_0^{q,q'}.
	    \end{align*}
	    Altogether, we obtain an efficient decomposition of the error estimator~$\EstPrim{\mu}{\,\cdot\,}$ which can be evaluated in a complexity independent of the high-dimensional state space.
	    \par
	    A similar decomposition can be derived for the error estimator of the adjoint system as well. Again, we obtain for the squared norm of the adjoint residual a representation of the form
	    \begin{align*}
	    	\normX{\E^{-\top}\ResAdjo{\mu}{\hat{\varphi}}(t)}^2 &= \ResAdjo{\mu}{\hat{\varphi}}(t)^\top \E^{-1}\Gmat\E^{-\top}\ResAdjo{\mu}{\hat{\varphi}}(t) \\
	    	&= \hat{\varphi}^\top \Vad^\top \A{\mu;t} \E^{-1}\Gmat\E^{-\top}\A{\mu;t}^\top \Vad\hat{\varphi} + \left(\frac{d}{dt}\hat{\varphi}\right)^\top \Vad^\top\Gmat\Vad\left(\frac{d}{dt}\hat{\varphi}\right) \\
	    	&\hphantom{=}+2\left(\frac{d}{dt}\hat{\varphi}\right)^\top \Vad^\top\Gmat\E^{-\top}\A{\mu;t}^\top \Vad\hat{\varphi} \\
	    	&= \hat{\varphi}^\top M_7(\mu;t)\hat{\varphi} + \left(\frac{d}{dt}\hat{\varphi}\right)^\top M_8\left(\frac{d}{dt}\hat{\varphi}\right) + 2\left(\frac{d}{dt}\hat{\varphi}\right)^\top M_9(\mu;t)\hat{\varphi},
	    \end{align*}
	    where
	    \begin{align*}
	    	M_7(\mu;t) &\coloneqq \Vad^\top \A{\mu;t} \E^{-1}\Gmat\E^{-\top}\A{\mu;t}^\top \Vad\in\R^{\kad\times\kad},\\
	    	M_8 &\coloneqq \Vad^\top\Gmat\Vad\in\R^{\kad\times\kad},\\
	    	M_9(\mu;t) &\coloneqq \Vad^\top\Gmat\E^{-\top}\A{\mu;t}^\top \Vad\in\R^{\kad\times\kad}.
	    \end{align*}
	    The full offline-online decomposition can now be performed similarly as for the error estimator of the primal system.
	    
	    \subsection{Solution of the fully reduced model}\label{app:solution-fully-reduced-model}
	    To compute~$\compredFTA{\mu}{\dimRedSpaceFTA}$ in practice, we consider the reduced basis~$\redBasisVred{\dimRedSpaceFTA}=\{\varphi_1,\dots,\varphi_\dimRedSpaceFTA\}\subset\R^n$ such that~$\Span{\redBasisVred{\dimRedSpaceFTA}}=\spaceVred{\dimRedSpaceFTA}$ and the columns of~$\Vred$ are given by the vectors~$\varphi_1,\dots,\varphi_\dimRedSpaceFTA$. The least squares problem in~\eqref{equ:def-fully-reduced-final-time-adjoint} is solved by projecting the right-hand side of the linear system onto the space spanned by the vectors reachable from~$\spaceVred{\dimRedSpaceFTA}$ using only reduced systems. The coefficients~$\alpha^\mu=[\alpha_i^\mu]_{i=1}^{\dimRedSpaceFTA}\in\R^\dimRedSpaceFTA$ of this~$\X$-orthogonal projection can be computed as solution of the (small) linear system
	    \begin{align}\label{equ:linear-system-coefficients-fully-reduced-solution}
	    	\ProjectionMatrix{\mu}^\top\Gmat\ProjectionMatrix{\mu} \alpha^\mu = \ProjectionMatrix{\mu}^\top\Gmat\Mop\left(\Vpr\RedPrStateTrans{\mu}{T}{0}\Vpr^\top\Gmat x_\mu^0-x_\mu^T\right),
	    \end{align}
	    where~$\ProjectionMatrix{\mu}=\left(\E^\top+\Mop\Vpr\RedGramian{\mu}\Vad^\top\Gmat\right)\Vred\in\R^{n\times\dimRedSpaceFTA}$ is not constructed, but instead we compute~$\ProjectionMatrix{\mu}^\top\Gmat\ProjectionMatrix{\mu}$ as
	    \begin{align*}
	    	\ProjectionMatrix{\mu}^\top\Gmat\ProjectionMatrix{\mu} &= \Vred^\top\E\Gmat\E^\top\Vred + \Vred^\top\E\Gmat\Mop\Vpr\RedGramian{\mu}\Vad^\top\Gmat\Vred \\
	    	&\hphantom{==}+ \left(\RedGramian{\mu}\Vad^\top\Gmat\Vred\right)^\top\Vpr^\top\Mop^\top\Gmat\E^\top\Vred + \left(\RedGramian{\mu}\Vad^\top\Gmat\Vred\right)^\top\Vpr^\top\Mop^\top\Gmat\Mop\Vpr\left(\RedGramian{\mu}\Vad^\top\Gmat\Vred\right)^\top.
	    \end{align*}
	    We can therefore precompute
	    \begin{align*}
	    	\Vred^\top\E\Gmat\E^\top\Vred\in\R^{\dimRedSpaceFTA\times\dimRedSpaceFTA},\qquad\Vred^\top\E\Gmat\Mop\Vpr\in\R^{\dimRedSpaceFTA\times\kpr}\qquad\text{and}\qquad\Vpr^\top\Mop^\top\Gmat\Mop\Vpr\in\R^{\kpr\times\kpr}
	    \end{align*}
	    during the offline phase and assemble the remaining components online for a new parameter with complexity independent of the state space~$\X$. To this end, we compute~$\RedGramian{\mu}\Vad^\top\Gmat\Vred\in\R^{\kpr\times\dimRedSpaceFTA}$ by solving only reduced order systems (the matrix~$\Vad^\top\Gmat\Vred\in\R^{\kad\times\dimRedSpaceFTA}$ is also precomputed offline) by making use of the identity in~\eqref{equ:reduced-gramian-product}. It is therefore required to solve the reduced adjoint and primal systems once for every basis function in~$\redBasisVred{\dimRedSpaceFTA}$, i.e.~a total of~$2\dimRedSpaceFTA$ reduced systems need to be solved. We furthermore compute the right-hand side of the system in a similar fully reduced manner. This way, the linear system in~\cref{equ:linear-system-coefficients-fully-reduced-solution} can be set up by only performing reduced computations. Having solved the linear system, the approximate final time adjoint state is then given as
	    \begin{align*}
	    	\compredFTA{\mu}{\dimRedSpaceFTA} = \sum\limits_{i=1}^{\dimRedSpaceFTA}\alpha_i^\mu\varphi_i.
	    \end{align*}
	    
	    \subsection{Reduced error estimator for the final time adjoint}\label{app:details-offline-online-decompositions-reduced-error-estimator}
	    Similar to~\Cref{app:details-offline-online-decompositions-primal-and-adjoint} we consider the squared norm and can decompose~$\RedEstFTA{\mu}{\dimRedSpaceFTA}(\varphi^\dimRedSpaceFTA)$ as
	    \begin{align*}
	    	\left[\RedEstFTA{\mu}{\dimRedSpaceFTA}(\varphi^\dimRedSpaceFTA)\right]^2 &= \normX{\Mop(\Vpr \RedPrStateTrans{\mu}{T}{0}\Vpr^\top\Gmat x_\mu^0-x_\mu^T)-(\E^\top+\Mop\Vpr\RedGramian{\mu}\Vad^\top \Gmat)\Vred\varphi^\dimRedSpaceFTA}^2 \\
	    	&= \left(\RedPrStateTrans{\mu}{T}{0}\Vpr^\top\Gmat x_\mu^0\right)^\top M_{10}\left(\RedPrStateTrans{\mu}{T}{0}\Vpr^\top\Gmat x_\mu^0\right) + (x_\mu^T)^\top \Mop^\top \Gmat\Mop(x_\mu^T)+\left(\varphi^\dimRedSpaceFTA\right)^\top M_{12}\varphi^\dimRedSpaceFTA\\
	    	&\phantom{==}+\left(\RedGramian{\mu} M_{11}\varphi^\dimRedSpaceFTA\right)^\top M_{10}\left(\RedGramian{\mu} M_{11}\varphi^\dimRedSpaceFTA\right)-2\left(\Vpr^\top \Mop^\top \Gmat\Mop x_\mu^T\right)^\top\left(\RedPrStateTrans{\mu}{T}{0}\Vpr^\top\Gmat x_\mu^0\right)\\
	    	&\phantom{==}-2\left(\varphi^\dimRedSpaceFTA\right)^\top M_{13}\left(\RedPrStateTrans{\mu}{T}{0}\Vpr^\top\Gmat x_\mu^0\right)-2\left(\RedPrStateTrans{\mu}{T}{0}\Vpr^\top\Gmat x_\mu^0\right)^\top M_{10}\left(\RedGramian{\mu} M_{11}\varphi^\dimRedSpaceFTA\right)\\
	    	&\hphantom{==}+2\left(\Vred^\top\E \Gmat\Mop x_\mu^T\right)^\top \varphi^\dimRedSpaceFTA+2\left(\Vpr^\top \Mop^\top \Gmat\Mop x_\mu^T\right)^\top \left(\RedGramian{\mu} M_{11}\varphi^\dimRedSpaceFTA\right)\\
	    	&\hphantom{==}+2\left(\varphi^\dimRedSpaceFTA\right)^\top M_{13}\left(\RedGramian{\mu} M_{11}\varphi^\dimRedSpaceFTA\right),
	    \end{align*}
	    where
	    \begin{align*}
	    	M_{10} &= \Vpr^\top \Mop^\top\Gmat\Mop\Vpr\in\R^{\kpr\times\kpr},\\
	    	M_{11} &= \Vad^\top\Gmat\Vred\in\R^{\kad\times \dimRedSpaceFTA},\\
	    	M_{12} &= \Vred^\top\E\Gmat\E^\top\Vred\in\R^{\dimRedSpaceFTA\times \dimRedSpaceFTA},\\
	    	M_{13} &= \Vred^\top\E\Gmat\Mop\Vpr\in\R^{\dimRedSpaceFTA\times\kpr}.
	    \end{align*}
	    Computing the quantities~$\RedPrStateTrans{\mu}{T}{0}\Vpr^\top\Gmat x_\mu^0$ and~$\RedGramian{\mu} M_{11}\varphi^\dimRedSpaceFTA$ can be done very efficiently using only the reduced systems as presented in~\Cref{sec:reduced-basis-mor-for-system-dynamics}. Furthermore, we have to compute~$(x_\mu^T)^\top \Mop^\top \Gmat\Mop (x_\mu^T)$, $\Vpr^\top \Mop^\top \Gmat\Mop x_\mu^T$, and~$\Vred^\top\E \Gmat\Mop x_\mu^T$: We thus define~$m_T^{q,q'}\coloneqq\left(x_T^q\right)^\top \Mop^\top\Gmat\Mop\left(x_T^{q'}\right)\in\R$ for~$q,q'=1,\dots,Q_{x^T}$ and can then compute
	    \begin{align*}
	    	(x_\mu^T)^\top \Mop^\top \Gmat\Mop (x_\mu^T) = \sum\limits_{q,q'=1}^{Q_{x^T}}\theta_{x^T}^q(\mu)\theta_{x^T}^{q'}(\mu)m_T^{q,q'}.
	    \end{align*}
	    We further set~$V_{\Mop,x^T}^{q}\coloneqq \Vpr^\top \Mop^\top \Gmat \Mop x_T^q\in\R^{\kpr}$ for~$q=1,\dots,Q_{x^T}$ and compute online the reduced quantity
	    \begin{align*}
	    	\Vpr^\top \Mop^\top \Gmat\Mop x_\mu^T = \sum\limits_{q=1}^{Q_{x^T}}\theta_{x^T}^q(\mu)V_{\Mop,x^T}^q.
	    \end{align*}
	    Similarly, we define~$\tilde{m}_{T}^{q}\coloneqq \Vred^\top\E \Gmat\Mop x_T^q\in\R^{\dimRedSpaceFTA}$ for~$q=1,\dots,Q_{x^T}$ and assemble
	    \begin{align*}
	    	\Vred^\top\E \Gmat\Mop x_\mu^T = \sum\limits_{q=1}^{Q_{x^T}}\theta_{x^T}^q(\mu)\tilde{m}_{T}^{q}
	    \end{align*}
	    online.
	\end{appendices}
\end{document}